\documentclass[12pt,reqno]{amsart}
\usepackage{fourier}

\input{preamble}
\begin{document}

\title[FCs and Rep.\ Theory]{Feynman categories and Representation Theory}

\author[R.~M.~Kaufmann]{Ralph M.\ Kaufmann}
\email{rkaufman@math.purdue.edu}

\address{Purdue University Department of Mathematics, West Lafayette, IN 47907
}

\begin{abstract}
We give a presentation of Feynman categories from a rep\-resentation--theoretical viewpoint.
Feynman categories are a special type of monoidal categories and their representations  are monoidal functors. They can be viewed as a far reaching generalization of groups, algebras and modules.
Taking a new algebraic approach, we provide more examples and more details for several key constructions. This leads to new applications and results.

The text is intended to be a self--contained basis for a crossover of more elevated constructions and results in the fields of representation theory and Feynman categories, whose applications so far include number theory, geometry, topology and physics.
\end{abstract}

\maketitle

\pagenumbering{Roman}
\section*{Introduction}

This paper concentrates on the algebraic aspects of Feynman categories.
Feynman  categories where introduced to have an
enveloping theory for several types of generalizations of algebras \cite{feynman}. Going beyond this and the original
intended application \cite{KWZ} the theory has found  applications outside its initial intention in algebra, category theory,
geometry, number theory and physics, \cite{decorated,BergerKaufmann, HopfPart1, HopfPart2,  BenSix}. Further applications to representation theory are present, but an extension of results and approaches is
desirable and anticipated. The present treatment is designed to aid this development.

Although treated within the general theory, the algebraic aspects have not been presented in full detail. There are many constructions and results that are subtle when going beyond set based categories, which is necessary to study algebra representations, that is modules. The current paper bridges this gap providing at an algebraically motivated, example based introduction to the theory  while at the same time providing new level of detail for these constructions. This clarifies previous results and constructions, while providing new results and concrete examples along the way.

 The basic idea underlying the formalism of Feynman categories is to separate objects and their structures. This is in a similar spirit as Galois' insight to separate the group from its representations, or,  in  modern terminology, the  category of its representations. Taking this approach  leads to a hierarchy of abstraction, and allows one to operate on a higher level.

 Continuing with the group analogy,
many things about groups and their representations just follow from the axioms of a group, and are hence true in general for all groups and their representations ---for instance restriction and Frobenius reciprocity. Other results depend only on the group and hence work in all representation of that group. Finally there are results about particular representations.
In keeping with this theme, the Feynman categories are the analogues of the groups, and their representations are given by  functors; that is monoidal functors to be precise.

There is a natural  categorical transition from groups to groupoids or  quivers, which is discussed in the first section.
In this version, the groups are indeed an example of Feynman categories and restriction, induction and Frobenius reciprocity are generalized to a pair of adjoint functors, see \S\ref{reppar} and \S\ref{feynmanpar}. More generally, Feynman categories can be understood as having two constituents, a groupoid providing  basic objects and isomorphisms, and a set of morphism encoding operations and their relations. Up to isomorphism the morphisms further decompose into tensor products of a basic morphisms, those whose target are the basic objects. The morphisms can be thought of as ``proto--operations'' on ``proto--objects'' that get realized to operations on objects if a representation functor is applied.
A presentation of a Feynman category, will be a set of basic generating morphism and relations among them.

A good example for the presentation of proto--operations, or morphisms, and their relations is  given by considering commutative algebras. This example also illustrates the hierarchy of abstraction.
An algebra is a linear object $A$ together with a multiplication $\mu:A\otimes A\to A$ which is bi--linear and associative.
 The structure so far is an object in a linear category and a 2-variable morphism with a relation ---associativity.
There are two ways of writing down the associativity equation. The first is in terms of elements $(ab)c=a(bc)$.  Using just the structure morphism $\mu$, one can  alternatively rewrite the equation as $\mu \circ_1\mu=\mu\circ_2\mu_2$ considered as morphisms of three variables obtained by substitution, where $\circ_i$ means plug into the
$i$--th variable. In this form, one has the following data: (1) an object in a linear monoidal category, we used the monoidal structure to write down $A\otimes A$ and
need  $A^{\otimes 3}$ for associativity.  Moreover, if we are allowed to plug into variables, we will get $n$--linear maps that is we will use $A^{\otimes n}$.
(2) A morphism $\mu$ and its iterates and the relations among them given by associativity. Moving to having this data as the value of a functor,  the simple data
which will be encoded in the Feynman category whose monoidal functors are commutative algebras. Is simply as follows: one basic object $*$ and one basic generating morphism $\pi_2:*\otimes*\to *$, the proto--multiplication, with the relation of associativity given by a commutative diagram corresponding to the equation $\pi_2 \circ_1\pi_2=\pi_2\circ_2\pi_2$. Commutativity corresponds to the fact that $\pi_2\circ (12)=\pi_2$, where $(12)$ is the elementary transposition. The monoidal part of the Feynman category is the category finite sets with surjections with disjoint union as monoidal product,  see \S\ref{finsetpar} and Proposition \ref{monoidprop} for full details.
  The morphisms are surjection and the basic morphisms are the surjections $\pi_S:S\to \{*\}$ for a chosen atom $*$.
 The multiplication  is the value of the functor on the surjection $\pi_2:\{1,2\}\ta \{1\}$ and the associativity corresponds to the fact that there is only one surjection $\{1,2,3\}\ta \{1\}$.
 Note that we now do not have to specify that the functor takes values in a linear category. In general, a functor out of the Feynman category into {\em any} monoidal target category $\C$ is equivalent to the data of a commutative monoid.

This begs the questions, which we will answer in the text:

 \begin{enumerate}
 \item \label{Q1} What are the natural generalizations of groups, algebras and representations in terms of Feynman categories?
  \item \label{Q2} Are there are  similar Feynman categories for modules and their generalizations?
   \item \label{Q3} What type of operations on algebras translate to the Feynman categories?
   \end{enumerate}
The answer to \eqref{Q1} is that there are indeed many Feynman categories naturally generalizing groups and algebras. There are even constructions, like the plus construction, which  build more complex Feynman categories from simpler ones. In particular,
there are  two constructions, which allow one to give more structure to the objects and the morphisms. The first is called decoration and the second indexed enrichment.  Decorations, which are a form of Grothendieck (op)--fibration, lead to a factorization system for morphisms in the category of Feynman categories analogous to Galois covers, see \S\ref{decopar} and \cite{BergerKaufmann}.
Other constructions allow one to consider lax--monoidal functors or regular functors instead of strong monoidal functors as representations.

Indexed enrichments are tied to the so--called plus construction, which gives rise to several hierarchies.  
The most basic one  starts with the trivial Feynman category whose representations are objects (\S\ref{trivialpar}). progresses to the category whose representations are associative algebras and the next step is given by non--symmetric operads. Beyond that one finds hyper-operads, which are necessary for the bar and co--bar construction and so on.  This may provide a first point of contact and exhibit the naturalness of the notions.  To obtain symmetric versions, one can use a forgetful functor which induces a cover by a decoration. In this fashion, there is a boot-strap, which generates a large part of the theory simply from the trivial Feynman category.
Another hierarchy starts with a Feynman category $\GG$ based on graphs, see Appendix \ref{graphsec}. This leads to modular operads, hyper-(modular)-operads, etc., which are intimately related to moduli spaces of curves, and, among other things, lead to twisted modular operads see \cite[4.1]{feynman} and \cite{KWZ,decorated,Ddec} as well as \S\ref{outlookpar}.
In general, the plus construction is a formalization of the fact that the morphisms in a category regarded as having inputs and outputs give rise to flow charts, see also \cite{pluspaper}.
For the reader not so familiar with these notions, the self--contained presentation   using the approach outlined above gives to a natural construction which makes their existence transparent. For instance, operads naturally appear when considering modules over algebras.

The plus construction is also important in comparing Feynman categories to operadic categories of
 \cite{BMopcat}, see \cite{BaKaMa}.  One application of these categories is to higher category theory
as they produce Batanin's $n$--operads \cite{Bathigher,Batnoper} which lead to an approach to higher categories.
 In the context of Feynman categories, the construction goes back to indexing \cite{feynmanarxiv}  as a 
codification and generalization the notion of  hyper--operads and twists as introducted in  \cite{GKmodular}. The fact that this is related to  so--called plus constructions, was explained to us by M.\ Batanin and C.\ Berger, which lead to the  formulation given in \cite[\S3,\S4]{feynman} that is presented below.
The origin of plus constructions goes back to  \cite{BaezDolanPlus}, see also \cite{PolynomialMonadplus} for a plus construction for polynomial monads. 
Iterations of plus constructions can be found in \cite{BaezDolanPlus}, under the name of opetopes, and also in \cite{Bathigher,VogtiteratedMonoidal,Leinster}.

As to question \eqref{Q2}:
there  are Feynman categories which allow to encode modules (\S\ref{pluspar}).
 It turns out, that in the analogy with groups, algebras and modules are formalized by indexed enrichment using the aforementioned plus construction: the hierarchies are more like ladders on which there  are two ways to move: ``up'' creating a new Feynman category and ``down'' using the upper rung to define an indexed enrichment. The representations of these indexed enriched versions are then the sought after modules.
For representations and modules it is important that these categories can be enriched.
Enrichment several different flavors, namely, combinatorial, topological and algebraic. The native constructions are combinatorial in nature as categories are based on sets. The other two are more complicated and are enriched, either in a Cartesian category,  which behave very much like $\Set$, or in non--Cartesian category, e.g.\ a linear ones, such as $\kVect$. These  are of basic interest in representation theory.
In the analogy with groups enrichment over $\kVect$ is the transition from group representations to $k[G]$ algebras. We give the details in \S\ref{enrichedpar} stressing the intricacies that are presented in the  non--Cartesian/linear case.

As far as question \eqref{Q3} is concerned,  there are analogues of the bar and cobar constructions (\S\ref{barpar}), as well as of a dual transform, aka.\ Feynman transform, which is the cobar on the dual of $A$.
This yields a generalization of Maurer--Cartan equations in the form of Master Equations (\S\ref{masterpar}), which are important in deformation theory.
There is a topological version of this, the so--called W--construction given in \S\ref{wpar}. The construction has a cubical nature and the cubical setting gives a natural wall structure. We give the construction for monoids and show that as expected one obtains the cubical decomposition of associahedra which also appear in the stability conditions for $A_n$ type algebras \cite{Igusamodulated}.

\subsection*{Organization of the text}
The text is designed to be as self-contained as possible and is aimed at a diverse audience.

 We start in \S\ref{reppar} with collecting classic results for groups and quiver representations, but reformulated in categorical language. This presentation might be of independent interest as a primer.

 The next paragraph, \S\ref{feynmanpar}, contains the definition of Feynman category introduced as a special type of monoidal category. The representations are then given by strong monoidal functors. The development is parallel to that of \S\ref{reppar}. The presentation of indexed Feynman categories is new. The section ends with examples which are based on finite sets. Here we provide new details. The representations of these are various kinds of algebras. The group(oid) representations are also included as a basic example. Further examples are provided by graphical Feynman categories. The theory of graphs we use is detailed in Appendix \ref{graphsec}.

 We then turn to  various constructions for Feynman categories in \S\ref{constructionpar}. These yield Feynman categories whose representations are lax-monoidal or simply functors. At this level the finite set based Feynman categories have  FI--modules and (co)--(semi)--simplicial objects as objects. The next operation is that of decoration. It yields  the graphical Feynman categories that encode operad--like types, see Table \ref{zootable}.
 The next construction is the plus construction. Here we give a detailed exposition of the condensed presentation in \cite[\S3.6]{feynman}, providing several explicit calculations.
  The new precision yields gcp--version   of the plus--construction, which is a  generalization of hyper version contained in \cite[\S3.7]{feynman}. The relationship to indexing is also made more explicit here then previously.
 A detailed graphical based analysis is given in Appendix \ref{graphpluspar}, where we also give a careful discussion of levels.

 In \S\ref{enrichedpar} we tackle the enriched version. This is technically the most demanding and contains many new details.
 The bar/co-bar transformation and a dual transformation, aka.\ Feynman transform along with the master equations are discussed in \S\ref{barpar}. Traditionally the bar/co-bar adjunction can be used to define resolutions. For this one needs a model structure in general. The relevant details are
 reviewed in Appendix \ref{modelpar}.
 The W--construction is reviewed in \S\ref{wpar}. Here we also reconstruct the associahedra in their cubical decompositions.

 We end the paper with an outlook, \S\ref{outlookpar} that contains further applications as well as speculations about cluster transformations, relations to moduli spaces and 2--Segal objects.

Appendix \ref{graphsec} also has several details not present in other discussions, such as more details about the composition, the composition of ghost graphs, and grafting into vertices. The presentation of the category of graphs following \cite{BorMan} is of independent interest as it captures just the correct amount of combinatorics for subtle considerations.

There is an additional Appendix \ref{twocatapp}, which gives the definition of 2--categories double categories and their relationship to Feynman categories and indexing. These  can provide a rather technical, but natural, background.

\subsection*{New Results}
For the reader already familiar with (some of) the notions, there are several new results. The connection to Frobenius reciprocity is new. Algebras receive a  full treatment at all levels  and the relation to classical results are pointed out along the way.
For instance, the examples of \S\ref{examplepar} are partially new and partially given in fuller detail. The treatment of noncommutative sets is entirely novel and provides a new avenue of construction. Tables \ref{table1} and \ref{table2} are the most exhaustive and detailed up to date.

The role of monoidal units is treated
 more carefully. First, for the free and nc construction in \S\ref{ncpar} which leads to more precise theorems.
Again a more careful treatment of units has lead to the definition and construction of $\FF^\gcp$ in \S\ref{enrichedpar}.
The graphical description of $\FF^+,\FF^\gcp$ and $\FF^\hyp$ in Appendix \ref{graphpluspar} is new in its level of detail and clarifies the short description in \cite{feynman}. This is a key result  in the current paper. Key  results and definitions in this direction  are Proposition \ref{functorequivprop}, Definition \ref{splitdef}, Theorems \ref{weakthm} and \ref{forestthm}.
The explicit computations of $\FF^+,\FF^\gcp$ and $\FF^\hyp$ are new, cf.\ Propositions \ref{algebraprop} and \ref{operadprop}.

The classification of FI enrichments in \S\ref{FIpar} is new and shows how the perspective of Feynman category leads to a natural classification and extension of previous results.

Throughout the paper, The role of indexing is paid more attention to. Especially in the section \S\ref{indexedpar},\S\ref{pluspar} and Appendix \ref{twocatapp}. The latter appendix also has a full development of the details of the constructions only briefly introduced in \cite{feynman} in terms of double categories and 2--categories. The incorporation of holonomy and connections is new.
 In  similar vein, in the section on decoration and covers \S\ref{decopar}, the criterion for being a cover is now given in Proposition \ref{coverprop}.

\subsection*{Acknowledgements}

This work would have not been possible without the 2018 Maurice Auslander International Conference. We thank
the organizers G.\ Todorov, A.\ Martsinkovsky,   and K.\  Igusa for making this possible and creating a perfect atmosphere. 

Special thanks goes to A.\ Martsinkovsky for his encouragement and patience during the preparation of the manuscript.
We are grateful to K.\ Igusa for key discussions, and, as well as S.\ Schroll and  A.\ Hanany  for further enlightening discussions.

We also thank  J.\ Miller and P.\ Patzt, for discussions about FI--algebras and the organizers and participants of the conference on ``Combinatorial Categories in Algebra and topology" in Osnabr\"uck, especially M.\ Spitzweck and and H.\ Krause, where the results on FI--algebras
were first presented,

It is a pleasure to thank and acknowledge  my collaborators C.\ Berger, M.\ Markl and M.\ Batanin for related and ongoing discussions.

RK gratefully acknowledges support from  the Simons Foundation the
Institut des Hautes Etudes Scientifiques and the Max--Planck--Institut
for Mathematics in Bonn for their hospitality and
 their support.

\subsection*{Conventions} For a complex $C_\bullet$, we define the shift by
 $(\Sigma C)_n=(C[1])_n=C_{n+1}$. The effect is that the complex shifts down and suspension $s:=\Sigma^{-1}$ shifts the complex up.  If $f:C\to D[k]$ them we set $|f|=k$. This means
 that $f$ is given by a collection of maps $f_n:C_n\to D_{n+k}$ and $|f(c)|=|f||c|$.

 \tableofcontents

\cleardoublepage\pagenumbering{arabic}
\section{Representations from a categorical viewpoint}
\label{reppar}
\subsection{Representations}
A $k$-representation of a group $G$ is a pair $(V,\rho)$ of a vector space $V$ over $k$ and a morphism of groups $\rho:G\to \Aut(V)$, where $\Aut(V)$ is the group of $k$-linear automorphisms of $V$. This data is neatly organized and generalized as follows.

\begin{df}
A {\em groupoid} is a category whose morphisms are all invertible.
\end{df}
\begin{ex}
\label{groupoidex}
Let $\underline{G}$ be the category with one object $*$ and morphisms $Hom_{\underline G}(*,*)=G$ where the composition map is given by group multiplication: $f\circ g =fg$. The unit $id_*$ is the unit $e\in G$, the inverses of the morphisms are the inverse group elements $g^{-1}$, hence this is indeed a groupoid.
\end{ex}

\begin{df}
A {\em  representation of a groupoid $\Gr$} is a functor $F:\Gr\to \C$.
\end{df}

\begin{ex}
Let $k$-$\Vect$ be the category of $k$ vector spaces. A functor $F:\underline{G}\to k$-$\Vect$ is exactly a $k$ representation of $G$. Since $\underline{G}$ only has one object $*$, on the object level the functor is completely fixed by $V:=F(*)$. For the morphisms, we have an additional morphism
$\rho:=F:\Aut(*)\to \Aut(V)$. Thus the functor is determines and is uniquely determined by the pair $(V,\rho)$.
\end{ex}
As the example illustrates, one can quickly get generalizations. Groupoid representations  are given by collections of objects, automorphisms of them and isomorphisms between them. Another generalization is given by changing the target category $\C$ from $\kVect$ to some other category to obtain
groupoid representations in different categories.

\begin{ex}
For any category $\C$, we let $\Iso(\C)$ be the wide sub--category with $\Obj(\Iso(\C))=\Obj(\C)$ and $\Mor(\Iso(\C)))= \Iso(\C)\subset \Mor(\C)$ the subset of isomorphisms. This is a groupoid sometimes called the underlying groupoid.

In the example above the functor $\rho:\underline{G}\to \C$ actually factors through $\Iso(\C)$ and more generally so does any functor whose source is a groupoid. Note, a category $\V$ is a groupoid if and only if $\Iso(\V)=\V$.
\end{ex}

\begin{ex}
Another typical groupoid is a action groupoid. Let $X$ be a set and $G$ be a group action $\rho$ on $X$. Let $Mor=G\times X$ and $Obj=X$. Furthermore set
 $s=\pi_2: Mor\to Obj: s(g,x)=x$,  define $t:Mor\to Obj: t(g,x)=\rho(g)(x)$ and $id:Obj\to Mor: id(x)=(e,x)$ where $e\in G$ is the unit. Lastly, define composition
  $\circ:Mor \leftidx{_s}\times{_t}Mor\to Obj$ given by $(g,\rho(h )(x))\circ (h,x)=(gh,x)$.
  Then this  data forms  a category and moreover a groupoid.
\end{ex}

\begin{rmk}
If one wants to add more geometry or more conditions to a groupoid, one can consider a groupoid internal to a category $\C$.  Like a category internal to $\C$ this is a pair of objects $Mor$ and $Obj$  of  $\C$  which form the morphisms and objects of a category together with morphisms in $\C$ $s,t:Mor\to Obj$, $id:Obj\to Mor$ and $\circ:Mor \leftidx{_s}\times{_t}Mor\to Obj$ satisfying all the conditions of a category.
\end{rmk}
\subsubsection{Intertwiners as Natural Transformations}
Morphisms between representations $(V,\rho_V)$, $(W,\rho_W)$ aka.\ intertwiners are morphisms $N:V\to W$ such that
\begin{equation}
\label{intertweq}
N\circ \rho_V=\rho_W\circ N
\end{equation}

This equation is also the equation for natural transformations. Recall that functors from $\C$ to $\D$ form a category ${\it Fun}(\C,\D)$ whose morphisms are natural transformations. Where  $Nat(F,F')$ are the natural transformations from $F$ to $F'$ and a natural transformation $N$ is given by a collection of morphisms $N_X:F(X)\to F'(X)$ indexed  by the objects of $\C$ that satisfy
\begin{equation}
\forall \phi\in Hom_\C(X,Y): N_X\circ F'(\phi)=N_Y\circ F(\phi)
\end{equation}
\begin{ex}
In the example where $\C=\underline{G	}$, there is only one object $*$ and hence only one morphism $N_*=N$, and the equation becomes \eqref{intertweq}.
\end{ex}
\subsection{Graphs and Quivers}
A groupoid gives rise to a graph and vice--versa any graph gives rise to a (free) groupoid, as we will review, see also e.g.\ \cite{graphsym}.

We will use the Borisov--Manin definition of graphs and morphisms, \cite{BorMan,feynman}. Full details are in Appendix \ref{graphsec}.
In this formalism, a graph $\Gamma$ is a collection $(V,F,\del,\imath)$, where $V$ is a set of vertices, $F$ is a set of flags aka.\ half edges, $\del:F\to V$ assigns a base vertex to each flag and $\imath:F\to F$ is an involution $\imath^2=id$. Edges are orbits of order two of $\imath$, that is an edge $e=\{f,\imath(f)\}$ comprises two half edges and the obits of order one, are  ``unpaired'' half edges aka.\ tails. The set of edges will be denoted by $T$ and that of tails by $T$. A {\em directed edge $\vec{e}$}  is a pair $(f,\imath(f))$. Each edge gives rise to two directed edges and by picking the first flag the set of directed edges is in bijection with the flags, that {\em are not} tails.
A path is a sequence of directed edges $\vec{e}_i=(f_i,\imath(f_i))$, such that $\del(\imath(f_i))=\del(f_{i+1})$.
The set of all paths on $\G$ is denoted by $\P(\G)$.
A {\em directed graph}, aka. quiver, is a graph, with a choice of direction for all of its edges.

\subsection{Graphs and Groupoids}
Given a groupoid $\Gr$ the corresponding graph has $V=\Obj(\Gr)$ and flags $F=\Mor(\Gr)$, $\del(\phi)=s(\phi)$ and $\imath(\phi)=\phi^{-1}$. This graph has no tails, and hence the directed edges are in bijections with the flags. Notice that (i) each object has an identity map, thus there is at least one loop at each vertex, (ii) the graph structure does not encode the composition. We do however have a morphisms $\P(\G)=\Mor(\Gr)$, by sending the sequence $\vec{e}_i, i=1\kdk n$ to $f_n\circ\dots\circ f_1$.

Vice--versa, given a graph $\G$ without tails, setting  $\Obj(\Gr)=V$ and $F$ as the set of directed edges as the basic morphisms, where the source and target maps are given by $\del$. One cannot read off relations, but one can simply regard the free category on a given graph $\G$ to be the category generated by the flags of $\G$, that is morphisms $\Mor(\Gr)=\P(\G)$ together with the identity morphisms of all the objects and concatenation being composition of composable paths.

\begin{rmk}
One can use functorial language to say that there is a forgetful functor $\G:\mathcal{C}at_{sm}\to \Graphs^{nt}$ from the category of small category to that of graphs without tails and that this functor has a left adjoint free functor $F$, $F\dashv G$, see \S\ref{adjointpar} below.
\end{rmk}

\subsection{Monoids and quiver representations}
\label{monoidpar}
On can relax the invertibility and see that using the construction of Example \ref{groupoidex}
actually any unital monoid $A$ gives rise to a category $\underline{A}$.

Similarly, relaxing the invertibility, but adding a direction for each edge, a directed graph, gives rise to a category, where the generating morphisms are exactly the directed edges. This is usually known as a quiver.

\subsection{Restriction, Induction and Kan Extensions}
Given a morphism of groups $f:H\to G$, in particular a subgroup $H\subset G$, there are two canonical operations for representations. Restriction $res_H^G:Rep(G)\to Rep(H)$ and
induction $ind_H^G:Rep(H)\to Rep(G)$.

In the categorical formulation this amounts to the following triangles.
\begin{equation}
\xymatrix{
\underline{H}\ar[rr]^f\ar[dr]_{F\circ f=f^*F}&&\underline{G}\ar[dl]^F\\
&\C&
} \qquad
\xymatrix{
\underline{H}\ar[dr]_F\ar[rr]^f&&\underline{G}\ar[dl]^{Lan_fF=f_!F}\\
&\C&
}
\end{equation}
Here pull--back $f^*$ is simply restriction. Induction is more complicated and is given by the so--called left Kan extension.

In general, the situation is that one has a functor $f:\D\to \E$ this gives rise
to a morphism in the category of functors and natural transformations $f^*: {\it Fun}[\E,\C]\to {\it Fun}[\E,\C]$ by sending $F\in {\it Fun}[\E,\C]$ to $f^*F=F\circ f$.
 A left adjoint functor $f_!$, $f_! \dashv f^*$, if it exists if it exists gives a functor in the other direction
\begin{equation}
\label{adjointeq}
f_!: {\it Fun}[\E,\C]\leftrightarrows {\it Fun}[\E,\C]: f^*
\end{equation}
The left Kan extension, if it exists provides such a left adjoint: $Lan_f=f_!$.

\subsubsection{Adjoint functors}
\label{adjointpar}
A functor $F\in {\it Fun}(\C,\D)$ is called the left adjoint to a functor $G\in {\it Fun}(\D,\C)$
if there are natural bijections
\begin{equation}
Hom_\C(X,GY)\leftrightarrow Hom_\D(FX,Y)
\end{equation}

Typical pairs are $G=$ for\underline{g}et  and $F=$ \underline{f}ree. E.g.\ if $\C=\Set$ and $\D=\mathcal{G}roup$.
In the case at hand, the functors are $F=f^*$ and $f_!$ which run between the categories indicated in \eqref{adjointeq}.

 \subsubsection{Left Kan extension}
The Kan extension,  gives a left adjoint functor, the putative formula for the point-wise Kan left extension is
\begin{equation}
Lan_f F(Y)=\colim_{(f \downarrow Y)} F\circ s
\end{equation}
if the colimit exists. We will now discuss how to calculate such a beast in the situation above.

First, $(f\downarrow Y)$ is a so--called comma or relative slice category which has as objects pairs $(X,\phi \colon f(X)\to Y)$ where $X\in \D$ and $\phi\in Hom_\E(f(X),Y)$. The morphisms from $(X,\phi \colon f(X)\to Y)$ to $(X',\phi' \colon f(X')\to Y)$
are morphisms $\psi: X\to X'$ such that  $\phi=\phi'\circ f(\psi)$

\begin{equation}
\label{commamorpheq}
\xymatrix{
X\ar[d]_\psi& f(X)\ar[r]^{\phi}\ar[d]_{f(\psi)}&Y\\
X'&F(X')\ar[ur]^{\phi'}\\
}
\end{equation}
Second: the source map $s$ sends $(X,\phi)\mapsto X$ and hence the evaluation at the object $(X,\phi)$ is $F\circ s(X,\phi)=F(X)$.

Finally in a category $\C$ with direct sums $\bigoplus$ or more generally co--products $\coprod$, we can compute small co--limits as follows. Given an index category $\mathcal{I}$ and a functor $F:\mathcal{I}\to \C$,  the co-limit is:
\begin{equation}
\colim_{\mathcal{I}}F=\bigoplus_{X\in \Obj(\mathcal{I})}F(X)/\sim
\end{equation}
where $\sim$ is the equivalence relation induced by $F(X)\ni y\sim F(\phi)(y)\in F(X')$, where $\phi\in Hom_{\mathcal{I}}(X,X')$.
\begin{ex}
Consider a set $X$ with a group action $\rho:G\times X\to X$.
Set $\mathcal{I}=\underline{G}$ and consider $\C=\Set$. A functor $F:\mathcal{I}\to \Set$ then given by $X=F(*)$ and
a morphism $F:G\to \Aut(X)$ which is equivalent to the action $\rho:G\times X\to X$.
The colimit $\colim_\mathcal{I}F=X_G$ is given by the co--invariants of $X$, that is $X/\sim$ where $x\sim x'$ if there is a $g\in G$ such that $g(x)=x'$. These is exactly $F(*)/\sim$ above. Note that there is a natural map  quotient  $X\to X_G$
\end{ex}

The co-limit is actually more that an object, but it is an object together with  morphisms and as such is defined by its universal property. The co-limit $\colim_\mathcal{I}F$ is a coherent collection (aka.\ co-cone) $(C,(\pi_X:F(X)\to C)_{X\in \Obj(\mathcal{I})})$ where coherent means that for all $\phi:X\to X':\pi_{X'}\circ F(\phi)=\pi_X$.
The universal property is that if $(C',\pi'_X)$ is any other co--cone, there is a map $\psi:C\to C'$ which commutes with all the data.

In computing co-limits the following slogan is useful:
\begin{quote}{\em A co--limit of a given functor can be  computed by using an equivalent indexing category.}
\end{quote} Hence one can compute a co--limit using a skeleton, that is an equivalent category that only has one object in each  isomorphism class.
For a category $\F$, we will denote its skeleton by $sk(\F)$.

\begin{ex}
Taking the Kan extension as an example, we see that there is a component for each object $(X,\phi)$ and the equivalence relation is given by the morphisms $\psi$ as in \eqref{commamorpheq}, that is
\begin{equation}
Lan_f F(Y)=\bigoplus_{(X,\phi: f(X)\to Y)} F(X)/\sim
\end{equation}
where $F(x)\ni x\sim F(\psi)(x)\in F(X')$.
Note that this allows one to omit components $(X,\phi)$ whenever there is a morphisms $\psi$ with this as a source.
This means that only co--final objects, namely those which are not the source of a non--automorphism play a role when computing the limit.
\end{ex}

Finally notice that the Kan extension yields a functor. It also has values on morphisms
$Lan_fF(\psi:Y\to Y'):Lan_f F(Y)\to Lan_f F(Y')$. This is obtained by mapping the component $(X,\phi)$ to $(X,\psi\circ\phi)$ using $F(\psi)$.

\subsection{Restriction, Induction and Frobenius reciprocity}
With this setup, we can retrieve classical results. Consider an  inclusion of groups $H\hookrightarrow G$. This gives rise to the functor $f:\underline{H}\to \underline{G}$. Given a representation $\rho \in {\it Fun}(\underline{G},\kVect)$, we have
$res_H^G(\rho)=\rho\circ f$. Furthermore for $\rho \in {\it Fun}(\underline{H},\C)$
\begin{equation}
ind_H^G\rho(*)=Lan_f\rho(g)=\bigoplus_{(*,g):g\in G} V/\sim \; =G\times_H V
\end{equation}
where $V=\rho(*)$, $g\in Hom_{\underline{G}}(*,*)=G$, $H$ acts on the right by multiplication on $G$ and on the left via $\rho$ on $V$. In terms of the colimit, $g\sim g'$ if there is an $h$ such that $g=g'h$ and the functorial action of $h$, is given by sending the component of $g$ to that of $g'=gh^{-1}$ using $\rho(h)$ as the morphism on $V$, which is exactly what the relative  product encodes.

The using the equivalence ${\it Fun}[\underline{G},\kVect]=k[G]\text{-}mod$, the adjointness of $ind_H^G =f_!\dashv f^*=res_H^G$, yields one version of {\em Frobenius reciprocity}.

\begin{equation}
\label{frobrepeq}
Hom_{k[G]}(ind^G_H\rho,\lambda)\leftrightarrow Hom_{k[H]}(\rho,res^G_H\lambda)
\end{equation}

\subsection{Algebra and dual co--algebra structure}
\label{coalgpar}
 Concatenation operation for morphisms gives a partial composition for morphisms of a category $\C$.
 Let $C=k[\Mor(\C)]$ then $C$ is an associative algebra with the multiplication
 \begin{equation}
 \phi\cdot\psi=\begin{cases}\phi\circ\psi&\text{If they are composable}\\
0&\text{else}
  \end{cases}
 \end{equation}
 there is an approximate unit for the multiplication which is $\sum_X id_X$. This is a unit, if there are only finitely many objects.
$\C$ is called of {\em  decomposition finite} if for each $\phi$ there are only finitely many pairs $(\phi_0,\phi_1)$ such that $\phi=\phi_0\circ \phi_1$. In this case
\begin{equation}
\Delta(\phi)=\sum_{(\phi_0,\phi_1):\phi_0\circ\phi_1=\phi}\phi_0\otimes \phi_1
\end{equation}
 is a co--associative multiplication, see e.g.\ \cite{JR,HopfPart1,HopfPart2}  and has a co--unit
 \begin{equation}
 \eps(\phi)=\begin{cases}1&\text{if $\phi=id_X$ for some $X$}\\
 0&\text{else}
  \end{cases}
 \end{equation}
This is one of the instances, where cutting is simpler than gluing, in the sense that in order to glue, one usually has conditions and hence only partial structures, while when cutting, the cut pieces have no conditions as they could be re--glued.
\begin{ex} \mbox{}
\begin{enumerate}
\item For a groupoid $\underline{G}$, $C$  is the group algebra $k[G]$ and the co-multiplication is the dual
$\Delta(g)=\sum_{(g_1,g_2):g_1g_2=g} g_1\otimes g_2$. This is the usual co--multiplication one gets for the functions on the group.

\item  For a quiver this is the quiver algebra $k\vec{\G}$ or the path algebra. The algebra is free and hence decomposition finite and thus has a dual co--algebra. The co--product is de--concatenation of paths.

\item For a groupoid, $C$ is the groupoid algebra and the co--product is the de--composition co--product. If the groupoid is  given by a graph $\G$ there are many relations. The path--algebra is a groupoid algebra and represents the fundamental groupoid. The groupoid may or may not be of finite type. Characteristics which keep it from being of finite type are isomorphisms of infinite order, i.e.\ loops, infinitely many paths connecting two vertices, or an infinite path connecting two vertices.

\end{enumerate}
\end{ex}
\section{Feynman categories}
\label{feynmanpar}
Feynman categories are a special type of monoidal category which generalize the groupoids in two ways.
\subsection{Monoidal Categories}
A monoidal category $\C$ is a category with a functor $\otimes:\C\times \C\to \C$.
This means the for any two objects $X,Y$ there is an object $X\otimes Y$ and
for any two morphisms $\phi\in Hom_\C(X,Y), \phi'\in Hom_\C(X',Y')$ a morphism $\phi\otimes \phi'\in Hom_\C(X\otimes X',Y\otimes Y')$, such that
\begin{equation}
\label{interchangeeq}
(\phi\otimes \phi')\circ(\psi\otimes \psi')=(\phi\circ\psi)\otimes (\phi'\circ \psi')
\end{equation}

There are several other structures needed for a monoidal category
\begin{enumerate}
\item A unit object $\unit$ together with isomorphisms aka.\ unit constraints $X\otimes \unit \stackrel{\lambda}{\to}X \stackrel{\rho}{\leftarrow} \unit \otimes X$.
\item Associativity isomorphisms $A_{XYZ}:(X\otimes Y)\otimes Z\stackrel{\sim}{\to} X\otimes (Y\otimes Z)$
\item In the case of a {\em symmetric} monoidal category isomorphisms, aka.\ commutativity constraints
$C_{XY}:X\otimes Y\stackrel{\sim}{\to} Y\otimes X$ with $C_{YX}\circ C_{XY}=id$
\end{enumerate}
that satisfy various compatibility conditions, such as the pentagon axiom, see e.g.\ \cite{Kassel}.
A monoidal category is called strict if $\lambda,\rho$ and $A$ are identities. MacLane's coherence axiom states, that every monoidal category is equivalent (even monoidally) to a strict one \cite{MacLane}.

\begin{ex}
The most well--known example is $(\kVect,\otimes_k)$, here $\unit= k$. For $\Z$ graded $k$ vector spaces $\kVect^\Z$ the usual Koszul sign conventions are  commutativity constraints given by $C_{VW}(v\otimes w)=(-1)^{|v||w|}w\otimes v$, where $|v|$ is the $\Z$ degree. This formula also is used in the $\Z/2\Z$ case. The unit is $k$ in degree $0$.

 Another important example is $(\Set,\amalg)$ where $\amalg$ is disjoint union. One can also consider
 $(Top,\times)$.
\end{ex}

A strong monoidal functor between two monoidal categories $(\C,\otimes_C)\to (\D,\otimes_\D)$ is a functor $F:\C\to \D$
and natural isomorphisms $\Phi_{XY}:F(X)\otimes_\C F(Y)\stackrel{\sim}{\to} F(X)\otimes_\D F(Y)$, $F(\unit_\C)=\unit_\D$
and $F(\lambda_\C)=\lambda_\D$ as well as for $\rho$ and all is compatible with the other constraints.

If the $\Phi_{XY}$ are identities the functor is called {\em strict monoidal}. If they are just natural morphisms, the functor
is called {\em lax monoidal}. A co--lax or {\em op--lax monoidal} functor has morphisms $\hat\Phi: F(X)\otimes_\D F(Y)\to F(X)\otimes_\C F(Y)$

Strong monoidal functors  form a category ${\it Fun}_{\otimes}(\C,\D)$, the same is the case for the other versions.
The natural transformation have to respect the other structure maps.

\subsubsection{Free monoidal categories}

The strict free monoidal category on a category $\C$ is given by objects that are words in objects of $\C$ and morphisms induced by morphisms in $\C$. For the non--strict case, the free monoidal category are bracketed words with associativity morphisms between them and the symmetric version adds commutation isomorphisms.
The free monoidal category in both the symmetric and non--symmetric case will be denoted by $\V^\otimes$.
The strict and non--strict versions are equivalent as (symmetric) monoidal categories, see e.g.\ \cite[\S2.4]{matrix} for further discussion. For the examples, we mostly use the strict version, as it has less objects. The non--strict version is more natural when one is dealing with functors into non--strict monoidal categories, see below.

\begin{ex}
If $\V=\underline{1}$ then the {\em strict symmetric version} is: $\V^\otimes \simeq \SS$; viz.\  the category whose objects are natural numbers $\N_0$ corresponding to the powers $n=*^{\otimes n}$ and $Hom_{\V^\otimes}(n,n)=\SS_n$ the symmetric group, with all other Hom--sets being empty. Here $\unit=0=*^{\otimes 0}$, that is the empty word. For the {\em strict  non--symmetric }version: $\V^\otimes=\N_0$, that is the discrete category of
natural numbers.\footnote{A category is discrete if the only morphisms being are identity morphisms $id_X$. This defines a way to identify  sets with small  discrete categories.}
\end{ex}

The free monoidal category has a universal property. For this notice that there is an inclusion $j:\C\to \C^\otimes$ by one letter words. The property can now be phrased as follows, every functor $f:\C\to \D$ into a
monoidal category $(\D,\otimes)$ has a lift $f^\otimes:\C^\otimes \to \D$ as a monoidal functor such that
$f=f^\otimes \circ j$. This association is functorial and
\begin{equation}
\label{univeq}
{\it Fun}(\C,\D)\simeq {\it Fun}_\otimes (\C^\otimes, \D)
\end{equation}

\begin{ex}
For instance, we have the $k[G]\text{-}mod={\it Fun}(\underline{G},\kVect)={\it Fun}_\otimes(\underline{G}^\otimes,\kVect)$.
Similarly for $k\vec{\G}$.
\end{ex}
\subsection{Algebra structure for strict monoidal categories}
\label{algebrapar}
If $(\C,\otimes)$ is a strict monoidal category there is a unital algebra structure on $C=\Mor(\C)$ given by $\otimes$.
The unit is $id_\unit$.

\begin{rmk}
Thus on a monoidal  category, $C$ has two algebra structures, which are compatible by the intertwining relation \eqref{interchangeeq}, or if it is decomposition finite. (1) a unital algebra structure given by $\mu=\otimes$ with unit $id_\unit$ and (2) a co--unital co--algebra structure $(\Delta, \eps)$ given by deconcatenation, see \S\ref{coalgpar}.

It is {\em not true in general that these structure from a bi--algebra}. This is the case for non--symmetric Feynman categories, and for the induced structures on isomorphism classes
 for Feynman categories \cite{HopfPart2}.
\end{rmk}

\subsection{Feynman categories}
Consider a triple $\FF=(\V,\F,\imath)$ of a groupoid $\V$ a (symmetric) monoidal category $\F$ and a functor $\imath:\V\to \F$. By universality of the free (symmetric) monoidal category, there is a functor $\imath^\otimes:\V^\otimes \to \F$ which factors through $\Iso(\F)$ since $\V^\otimes$ is again a groupoid --- words in isomorphisms are isomorphisms.
Among the morphisms in $\F$ there are {\em basic morphisms} $X\to \imath(*)$ which are the objects of the comma category $(\F\downarrow \V)$ the morphisms in this category are commutative squares
\begin{equation}
\xymatrix{
X\ar[r]^{\phi}\ar[d]_{\psi}&\imath(*)\ar[d]^{\sigma}\\
X'\ar[r]^{\phi'}&\imath(*')\\
}
\end{equation}
where $\sigma$ is (necessarily) an isomorphism.

The tensor product induces a morphism $(\F\downarrow\V)^\otimes$ to the category of arrows $(\F\downarrow \F)$. It sends a word in $\phi_i:X_i\to \imath(*_i): \phi_1\cdots \phi_n \mapsto \phi_1\odo\phi_n:X_1\odo X_n\to \imath(*_1)\odo \imath (*_n)= \imath^\otimes(*_1\dots *_n)$.
The isomorphisms in $(\F\downarrow\F)$,  are commutative diagrams
\begin{equation}
\label{sdseq}
\xymatrix{
X\ar[r]^\phi\ar[d]_\sigma^\simeq&Y\ar[d]^{\sigma'}_\simeq\\
X'\ar[r]^{\phi'}&Y'}
\raisebox{-.6cm}{\quad denoted by $\sds(\phi):\phi\to\phi'$}
\end{equation}

We will abbreviate to $\sds$, if the source is clear. Alternatively $\sds$ can be interpreted be a map $\Hom(s(\sigma),s(\sigma')\to \Hom(t(\sigma),t(\sigma')$.
These morphisms can also be considered as 2--morphisms in a double category, see   Example \S\ref{sdsex} in  Appendix \ref{twocatapp}.

\begin{df}\cite{feynman}
\label{feynmandef}
A triple $\FF$ as above is a Feynman category if
\begin{enumerate}
\renewcommand{\theenumi}{\roman{enumi}}
\item \label{objectcond} $\imath^\otimes:\V^\otimes \to \Iso(\F)$ yields an equivalence of categories.
\item \label{morphcond} The monoidal product yields an equivalence $(\Iso(F\downarrow \V))^\otimes \simeq \Iso(\F\downarrow\F)$.
\item \label{smallcond} Every slice category $(\F\downarrow \imath(*))$ is essentially small.
\end{enumerate}
\end{df}
A Feynman category is called {\em strict} if the equivalences are identities. Using MacLane's coherence, one can show that every Feynman category is equivalent to a strict one.
We call a Feynman category {\em strictly strict}, if the equivalences become identities when using the strict free monoidal structures. $\FF$ is {\em skeletal} if $\F$ is.

The first condition says that each object $Y$ decomposes up to isomorphism into a word in $\V$:
 $Y\simeq \bigoplus_{v\in V}\imath(*_v)$, such a decomposition is unique up to unique isomorphism and all isomorphisms of $\F$ are induced from the (iso)morphisms in $\V$ acting on the letters of the word. This means that each object has a well defined length $|X|$ given by the length of an isomorphic word.
The second condition means that every morphisms $\phi:X\to Y$ in $\F$ is  decomposes isomorphically into a tensor product of basic morphisms according to a decomposition of $Y\simeq \bigoplus_{v\in V}\imath(*_v)$. This decomposition is unique up to unique isomorphism.
\begin{equation}
\label{feydecompeq}
\xymatrix{
X\ar[r]^{\phi}\ar[d]^{\simeq}_{\hat \sigma}&Y\ar[d]_{\simeq}^{\sigma}\\
\bigotimes_{v\in V}X_v\ar[r]^{\bigotimes_{v\in V}\phi_v}&\bigotimes_{v\in V}\imath(*_v)
}
\end{equation}
with $\phi_v:X_v\to \imath(*_v)$.

\subsubsection{Native length and element--type morphisms}
Notice by condition \eqref{objectcond} the length of an object $|X|=\text{tensor length}$ is well defined. If $X\simeq \bigotimes_{v\in V}*_v$ then $|X|=|V|$.
This defines the length of a morphism $\phi:X\to Y$ by $|\phi|=|X|-|Y|$. Isomorphisms necessarily have length $0$.
There are morphisms of negative length. These come from the fact that morphisms $\unit\to *_v$ are allowed, by the axioms. We will call these morphisms of element--type or simply elements. It follows from the axioms that any morphisms  factors as a tensor product into morphisms of positive length and element--type morphisms.

\subsubsection{Non-Sigma version}
Leaving out the ``symmetric'' in the monoidal categories, one arrives at the notion of a non--Sigma Feynman category. Let $\V$ be a groupoid, $\F$ be a monoidal category, $i:\V\to \F_<$ a functor, $\V^\otimes$ the free monoidal category and $\imath^\otimes:\V^{\otimes}\to \F_<$ be the induced functor.\footnote{We will use the notation $\F_<,\FF_<$ to indicate that these are non--symmetric, aka.\ ordered, versions.}
\begin{df}\cite{feynman}
\label{nsfeynmandef}
A triple $\FF_<=(\V,\F_<,\imath)$ as above is a non--Sigma Feynman category if
\begin{enumerate}
\renewcommand{\theenumi}{\roman{enumi}}
\item \label{nsobjectcond} $\imath^\otimes:\V^\otimes \to \Iso(\F)$ yields an equivalence of categories.
\item \label{nsmorphcond} The monoidal product yields an equivalence $(\Iso(F\downarrow \V))^\otimes \simeq \Iso(\F\downarrow\F)$.
\item \label{nssmallcond} Every slice category $(\F\downarrow \imath(*))$ is essentially small.
\end{enumerate}
\end{df}
Note that now the decompositions \eqref{feydecompeq} are unique up to isomorphisms in the letters ---permutations are not possible anymore.

\subsection{A bi-algebra and Hopf algebras structures for Feynman categories.}
The following result from \cite{HopfPart2} is a surprising feature of Feynman categories.

\begin{thm}{\rm \cite{HopfPart2}}
\label{Hopfthm}
The algebra structure of \S\ref{algebrapar} and the co--algebra structure of  \S\ref{coalgpar} for a decomposition finite monoidal category $\F$
\begin{enumerate}
\renewcommand{\theenumi}{\alph{enumi}}
\item satisfy the bi-algebra equation if $\F=\F_<$  belongs to a non--Sigma Feynman category $\FF_<$.
\item induce  a bi-algebra structure on the co-invariants  $\B$ of $C$ taken with respect to isomorphisms if  $\F$ is part of a Feynman category $\FF$.
\end{enumerate}
In the symmetric case let $\C$ be the ideal spanned by $[id_X]-[id_\unit]$ in the bialgebra $\B$ then
\begin{enumerate}
\renewcommand{\theenumi}{\alph{enumi}}
\setcounter{enumi}{2}
\item $\C$ is  a co-ideal.
\item If $\FF$ satisfies additional natural conditions listed in {\rm \cite[\S1.6]{HopfPart2}} then $H=\B/\C$ is a Hopf algebra.
\end{enumerate}
For a non--Sigma skeletal  strictly strict $\FF_<$ the corresponding ideal is given by the relations
$\frac{1}{|\Aut(X)|}id_X-id_\unit$. With a modified co--unit,
the quotient $\B/\C$ yields a Hopf algebra.
\end{thm}
Examples are the various Hopf algebras of Connes and Kreimer for trees and graphs \cite{CK1,CK2}, the Hopf algebra of Baues for double loop spaces \cite{BauesHopf}
and the Hopf algebra of Goncharov for multiple zeta values.

\begin{rmk}
Note  if  $\FF$ is a Feynman category with co-algebra $\C$, then $\FF^{op}$ will have the co-algebra structure $\C^{op}$. Thus $\FF^{op}$ although not a Feynman category will also yield a bi-algebra. One can speculate that  up to taking the opposite category, the bi-algebra structure is a defining feature.
\end{rmk}
\subsection{Representations of Feynman categories, aka.\ $\ops$}
Fix a (symmetric) monoidal target category $(\C,\otimes)$. We define:
\begin{equation}
\Fops_\C:={\it Fun}_\otimes (\F,\C) \text{ and } \Vmods_\C:={\it Fun}(\V,\C)
\end{equation}
where ${\it Fun}_\otimes$ means strong monoidal functors.
We will denote such functors by $\O$ and suppress $\C$ when it is fixed.
The functors $\Fops$ are ``representations''. They are usually operators or
operations, which is why we call them $\ops$. One such functor gives an $\F$-operation,
that is an $\F$-$\op$ or an $\op$ for short.

\subsubsection{Intertwiners and monoidal category structure}
 Using natural transformations as morphisms both $\Fops$ and $\Vmods$ are categories. These natural transformations correspond to  intertwiners.
This yields the natural definition of equivalence of $\ops$  and $\mods$ as isomorphic objects in these categories.

 $\ops$ and $\mods$  are symmetric monoidal categories for the level--wise tensor product.
 That is for $\O,\P\in \Fops$ or $\smodcat$
\begin{equation}
(\O\otimes \P)(X):=\O(X)\otimes \P(X)
\end{equation}
The monoidal unit is given by the trivial functor $\trivial$, which is defined by $\trivial(X)=\unit_\C$ and $\trivial(\phi)=id_{\unit_\C}$.
The unit, associativity and commutativity constraints are those  induced from  $\C$.

\subsubsection{A second monoidal structure} Due to the fact that in the setting of Theorem \ref{Hopfthm}, $\B$ is a bi--algebra there is an additional
monoidal structure on the co--completion of $\ops$, which has as of yet not been explored. For the co-completion and universal operations see \cite[\S6]{feynman}.

\subsubsection{Free $\ops$ and monadicity}

There is a forgetful functor $G:\Fops_\C \to \Vmods_\C$ were $G(\O)=\imath^*\O=\O\circ\imath$.
This functor is strong symmetric monoidal functor.

\begin{thm}{\rm \cite[Theorem 1.5.3.]{feynman}}
\label{freethm}
The functor $G$ has a left adjoint (free) functor $F\dashv G$, which is lax symmetric monoidal.
\end{thm}

There is another way to understand the operations as an algebra over a triple or a monad.
Given a pair of adjoint functors $F\colon\C\leftrightarrows\D\colon G$, there is the endofunctor $T=G\circ F\in {\it Fun}(\D,\D)$, which is a unital monoid as follows:
\begin{enumerate}
\item There is a natural transformation $\mu\colon T\circ T \to T$ given by the structure morphism of adjoint functors
$\eps: F\circ G\to id_\D$ (here $id_\D$ is the identity functor):
$T\circ T=(G\circ F)\circ (G\circ F)=G\circ(F\circ G)\circ F \stackrel{\eps}{\to} G \circ id_\C \circ F=T$.
\item The other structure map of the adjunction $\eta: id_\D\to F\circ G=T$ yields a unit for $T$:
$T=T\circ  id_\D\stackrel{\eps}{\to} T\circ T\stackrel{\mu}{\to}T$ is the identity transformation. Likewise for the left unit equation.
\end{enumerate}
An algebra over a triple $(T,\mu,\eps)$ is an object $\mathcal{M}$ in $\D$ together with a transformation
$\rho:T\mathcal{M}\to \mathcal {M}$ that is associative $\mu\circ\rho=\rho\circ\rho:T^2\mathcal{M}\to \mathcal{M}$.
The $T$--algebras in $\D$ form a category denoted by $\D^T$.
In the case at hand, $\mathcal{M}\in \Vmods$ is a $\V$ module and $\rho$ gives the operation of $T$ on $\mathcal{M}$.
\begin{thm}{\rm \cite[Corollary 1.5.5]{feynman}}
The adjunction $F\dashv G$ is monadic, that is $(\Vmods)^T=\Fops$.
\end{thm}

The image of $F$ in $\Fops$ are the free $\Fops$ and these are equivalent to the so--called Kleisly category
$(\Vmods)_T$.

\subsection{The category of Feynman categories}

Feynman categories again form a category. A morphism of Feynman categories $\FF=(\V,\F,\imath)$ to $\FF'=(\V',\F',\imath')$ is a pair $\ff=(v,f)$, $v\in {\it Fun}(\V,\V'), f\in {\it Fun}_\otimes(\F,\F')$ such that $f\circ \imath=\imath'\circ v$.
Theorem \ref{BKthm} give an important structural theorem by establishing a factorization system.

\subsubsection{Indexed Feynman categories}
\label{indexedpar}
\begin{df}
A Feynman category $\FF=(\V,\F,\imath)$ indexed over a Feynman category $\mathfrak{B}=(\V_\B,\B,\imath_\B)$ is a morphism of Feynman categories
$\mathfrak{b}=(v_b,f_b):\FF\to\mathfrak{B}$ whose underlying functor $f_b:\F\to \mathcal{B}$ is surjective on objects.

An indexing is called strong, if it is bijective on objects and surjective on morphisms.
A strong indexing is strict, if induces an equivalence of $\V\simeq \V_\B$.\footnote{Similar conditions are necessary to obtain morphisms of the associative Hopf algebras \cite[\S1.7]{HopfPart2}.}

\end{df}

\begin{rmk}
\label{decomprmk}
Let $\fb=(v_b,f_b):\FF=(\V,\F,\imath)\to \fB=(\V_\B,\B,\imath_\B)$ be an indexing, then
\begin{enumerate}
\renewcommand{\theenumi}{\roman{enumi}}
\item Morphisms decompose fiberwise:
\begin{equation}
\label{fiberhomeq}
\Hom_\F(X,Y)=\amalg_{\phi\in \Hom_{\B}(f_b(X),f_b(Y))}f_b^{-1}(\phi)
\end{equation}
\item Composition  and the monoidal product are partially defined  fiberwise
\begin{equation}
\begin{tabular}{l}
\xymatrix{
 f_b^{-1}(\phi)\times f^{-1}_b(\psi)\ar@{.>}[r]^\circ& f_b^{-1}(\phi\circ\psi)}\\
  \xymatrix{
f_b^{-1}(\phi)\times f^{-1}_b(\psi)\ar@{.>}[r]^\ot&
 f_b^{-1}(\phi\otimes\psi)
}
\end{tabular}
\end{equation}
These two partial products  are associative, satisfy  the  interchange relation
\eqref{interchangeeq} and are compatible with the groupoid action of $\Iso((\B\downarrow \B)$ lifted to $\Iso(\F\downarrow\F)$.

If the indexing is strong, then these products are fully defined.
\item Any invertible $\hat\sigma\in \Mor(\V)$ has $f_b(\hat \sigma)\in \Mor(V_\B)$.  And for $\sigma\in \Mor(\V)$:
$f_b^{-1}(\sigma)=f_b^{-1}(\sigma)^\times \amalg \overline{f_b}^{-1}(\sigma)$, where the first set is made up of all the invertible elements in the fiber.
If the indexing is strong, then the fiber has exactly one element: $f_b^{-1}(\sigma)=\sigma$.
\item There are unit elements
\begin{equation}
id_{X}\in f_b^{-1}(id_{f_b(X)})
\end{equation}

\item The monoidal unit, since native length is preserved and a monoidal unit is unique up to isomorphism,
$\unit_\F\in f_b^{-1}(\unit_\B)=\Gpd$, where $\Gpd$ is a discrete groupoid. If the indexing is strong then $f_b^{-1}(\unit_\B)=\unit_\F$.

\end{enumerate}

\end{rmk}
Examples of indexing are given by decoration, see \S\ref{decopar} and enrichment see \S\ref{pluspar}.

\begin{rmk}
\label{enrichfunctorrmk}
Using the fact that a monoidal category is a two--category with one object, see Appendix \ref{twocatapp}, one can rephrase Remark \ref{decomprmk} as saying that for a strong indexing $f_b^{-1}$ is a lax monoidal lax 2--functor to $\underline{\Set}$.
This relationship is the basis of indexed enrichment, see \S\ref{enrichedpar}, where $\Set$ is allowed to be replaced by some other symmetric monoidal category.
\end{rmk}

\subsubsection{Pull--back/Push--forward adjunction (restriction, induction and Frobenius reciprocity)}
There is a natural pull--back  or restriction for $\ops$: $\ff^*:\F'\text{-}\ops_\C\to \Fops_\C$ given by $\ff^*\O=\O\circ f$, which
is again a strong symmetric monoidal functor.

\begin{thm}{\rm \cite[Theorem 1.6.2]{feynman}}
\label{adjointthm}
The functor $\ff^*$ has a left adjoint $\ff_!\dashv \ff^*$ which is symmetric monoidal.
\end{thm}

The formula is again given by a left Kan extension. $\ff^*\O=Lan_f\O$. What is not obvious and is proven in {\it loc.\ cit.} is that this
Kan extension yields a {\em monoidal} functor.

\subsection{Examples}
\label{examplepar}
We will go through several examples. These examples are part of the fundamental ladder mentioned in the introduction whose base is the
trivial Feynman category. The next level is given by finite sets and their variations.
The different Feynman categories we discuss are collected in Table \ref{table1}, their non--Sigma analogues are in Table\ref{table2}.
The corresponding $\ops$ are collected in Table \ref{algtable}.

\subsubsection{Trivial Feynman category}
\label{trivialpar}
More generally, the trivial Feynman category on a groupoid
$\V$ is $\frakV=(\V,\V^\otimes,j)$. It has the following properties:

\begin{enumerate}
\item $\V^{\otimes}\text{-}\opcat_\C\simeq\Vmods_\C$, by the universal property of the free monoidal category.
    \item For $\V=\triv$, we will denote $\frakV$ by $\FFtriv$. We have $\V^{\otimes}\text{-}\opcat_\C\simeq\Vmods_\C=\Obj(\C)$. This is {\em the} trivial Feynman category.
\item If $\V=\underline{G}$ and $\C$ is $k$--linear then
$\V^{\otimes}\text{-}\opcat_\C\simeq\Vmods_\C=k[G]-mods$ in $\C$.

\item If we consider the inclusion $\imath:\underline{H}\to \underline{G}$. Then $i^*=res^G_H$ and
$i_!=ind^G_H$. The adjoitness of the functors is Frobenius reciprocity in the form \eqref{frobrepeq}.

\item More generally, given any Feynman category $\FF=(\V,\F,\imath)$ we can consider $\frakV$ and
the morphism given by ${\mathfrak i}=(id, \imath^\otimes)$. The using the isomorphism $j^*:\V^{\otimes}\text{-}\opcat_\C\stackrel{\sim}{\to}\Vmods_\C$, $ {\mathfrak i}_! \circ (j^*)^{-1}=F$
and $j^*\circ{\mathfrak i}^*=G$ are the adjoint pair of the free and forgetful functor in Theorem \ref{freethm}. Thus showing that this is a special case of Theorem \ref{adjointthm}.

\end{enumerate}
In general there may be more basic morphisms apart from those coming from $\V$.
In particular there may be basic morphisms $\unit \to \imath(*)$ and $X\to \imath(*)$

\subsubsection{Finite Sets}
\label{finsetpar}

Consider the symmetric monoidal category $(\FinSet$, $\amalg)$ whose unit is $\unit=\emptyset$, consider the inclusion functor $\imath:\underline{1}\to \FinSet$ that sends $*$ to the atom $\{*\}$. Then
$\FFinSet=(\underline{1},\FinSet,\imath)$ is a Feynman category.
The axioms are satisfied:
\begin{enumerate}
\renewcommand{\theenumi}{\roman{enumi}}
\item $\underline{1}^{\otimes}\simeq \SS=\Iso(sk(\FinSet))\simeq \Iso(\FinSet))$ where $sk(\FinSet)$ is the skeleton of $\FinSet$ whose objects are the sets $\underline{n}=\{1,\dots, n\}$, $n\in \N_0$ with $\underline{0}=\emptyset$.
\item Given any morphisms $S\to T$ between finite sets, we can decompose it using fibers as.
\begin{equation}
\label{finseteq}
\xymatrix
{
S \ar[rr]^{f}\ar[d]_{=}&& T\ar[d]^{=} \\
 \amalg_{t\in T} f^{-1}(t)\ar[rr]^{\amalg f_t}&&\amalg_{t\in T} \{*\}
 }
    \end{equation}
where $f_t$ is the unique map $f^{-1}(t)\to \{*\}$. Note that this map exists even if $f^{-1}(t)=\emptyset$.
This shows the condition (ii), since any isomorphisms of this decomposition must preserve the fibers.
\item The slice category $(\FinSet \downarrow *)$ is equivalent to its skeleton $\SS$.
\end{enumerate}

\begin{rmk}
We can also regard a skeletal version of $\FinSet$, this category has as objects the sets $\underline{n}=\{1,\dots,n\}$ with all morphisms between them.  The disjoint union is  $\underline{n}\amalg \underline{m}=\underline{n+m}$ with the unit $\underline{0}=\emptyset$. The isomorphisms are $\SS_n$ for $\underline{n}$, that is  $\Iso(sk(\FinSet))=\SS$.
This yields the strictly strict skeletal Feynman category $(\triv, sk(\FinSet),\imath)$.
\end{rmk}

$\FFinSet$ has the Feynman subcategories
$\Surj=(\triv,\surj,\imath)$ and $\Inj=(\triv,\inj,\imath)$, where the maps are restricted to be surjections resp.\ injections; see Table \ref{table1}.
This means that none of the fibers are empty in the surjective case and or all of the fibers are empty or singletons in the injective case.

\begin{prop} The following Feynman categories have   $\Vmods_\C=\Obj(\C)$ and he following $\ops$:
\label{monoidprop}
\begin{enumerate}
\item $\Surj$:
$\surj$-$\ops_\C$ is equivalent to the category of non--unital commutative monoids in $\C$.
\item $\Inj$: $\inj$-$\ops_\C$ are equivalent to pointed objects in $\C$.
\item $\FFinSet$:  $\FinSet$-$\ops$ are unital commutative monoids.
\end{enumerate}
\end{prop}

\begin{proof} The  statement about $\Vmods$ is clear, as ${\it Fun}(\trivial,\C)=\Obj(\C).$
For the first statement, let $\O\in \surj$-$\ops_\C$ and set $C:=\O(\imath(*))$.
By compatibility with the tensor product, up to equivalence, we may assume that $\O$ is strict and replace $\surj,\inj$ or $\FinSet$ with its skeleton. In all cases, the objects are the sets $\underline{n}=\{1,\dots, n\}$ with $\underline{0}=\emptyset=\unit$. Thus up to equivalence, $\O$ is fixed on objects as $\O(\underline{n})=C^{\otimes n}$ with the symmetric group acting by permuting the tensor factors using the commutativity constraints in $\C$.

The basic maps in $\surj$ are the unique surjections $\pi_n:\underline{n}\ta \underline{1}$.
Set $\mu=\O(\pi_2):\C^{\otimes 2}\to \C$. Then $(C,\mu)$ is a commutative non--unital monoid in $\C$. The multiplication
$\mu$ is associative as $\pi_2\circ (\pi_2\amalg id)=\pi_3=\pi_2\circ (id\amalg \pi_2)$
and hence $\O(\pi_2\circ (\pi_2\amalg id))=\mu\circ (\mu \otimes id)=\mu\circ (id\otimes \mu)=\O(\pi_2\circ (id\amalg \pi_2))$. For the commutativity let $\tau_{12}$ be the transposition that interchanges $1$ and $2$, then $\pi_2=\pi_2\circ \tau_{12}$ and hence $\mu_2=\O(\pi_2)
=\O(\pi_2\circ \tau_{12})=\O(\pi_2)\circ \O(\tau_{12})=\mu_2\circ C_{CC}$ where $C_{CC}$ is the commutativity constraint.

The basic morphisms for $\inj$ are $i:\emptyset=\unit\to \underline{1}$ and $id_1:\underline{1}\to \underline{1}$. Any injection can be written as a tensor product of these two maps.
The map $\eta:=\O(i):\O(\unit)=\unit_\C\to \O(\underline{1})=C$ makes $C$ into a pointed object. The values of $\underline{1}$ and $i$  determines the functor $\O$
uniquely up to  isomorphism.

Finally, the morphisms in $\FinSet$ are generated by $id_1,$ $\pi_2$ and $i$ using both the monoidal structure and concatenation. There is one more relation, that is $\pi_2\circ (id _1\amalg i)=id_1$, where we have tacitly used a strict unit constraint $\underline{1}=\underline{1}\amalg \emptyset$. Applying $\O$, we see that $\O(\pi_2\circ (id _1\amalg i))=\mu\circ (id \otimes \eta)=id=\O(id_1)$ again suppressing unit constraints. The fact that $\eta$ is a left identity follows from commutativity.
\end{proof}

\begin{rmk}
Judging by the name we chose for these categories, one could expect that to see find $FS$ and $FI$ algebras and indeed ${\it Fun}(\inj,\C)$ are $FI$--algebras and ${\it Fun}(\surj,\C)$ are $FS$ algebras.
By definition, however, $\ops$  are {\it monoidal} functors and not ordinary functors.
But, there is a free monoidal construction, see \S\ref{freemonoidpar} which to every Feynman category $\FF$ associated a Feynman category $\FF^\boxtimes$ with $\F^\boxtimes$-$\ops_\C={\it Fun}_\otimes(\F^{\otimes})={\it Fun}(\F,\C)$, and this way, we obtain $FI$--algebras as $\ops$.
\end{rmk}

\subsubsection{Ordered finite Sets}

In the non--$\Sigma$ case, a basic example is $\FFinSet_<=(\triv,\FinSet_<,\imath)$, where $\FinSet_<$ is the category of ordered finite sets with order preserving maps  with $\amalg$ as monoidal structure; the order of $S\amalg T$ is lexicographic, $S$ before $T$.
The functor $\imath$ is given by sending $*$ to the atom $\{*\}$.
 Viewing an order on $S$ as a bijection to $\{1,\dots,|S|\}$, we see that ${\bf N}_0$ is the skeleton of $\Iso(\FinSet_<)$. The diagram \eqref{finseteq} translates to this situation, and we obtain a non--$\Sigma$ Feynman category. The skeleton of this Feynman category is the strictly strict Feynman category $(\triv, \Delta_+,\imath)$, where $\Delta_+$ is the augmented simplicial category and  $\imath(*)=[0])$.
Restricting to order preserving surjections and injections, we obtain the Feynman subcategories $\Surj_<=(\triv,OS,\imath)$ and $\Inj_<=(\triv,OI,\imath)$.
We can also restrict the skeleton of $\FinSet_<$ given by $\Delta_+$ and the subcategory of order preserving surjections and injections.
See Tables  \ref{table2}. In $\Delta_+$ the image of $*^{\otimes n}$ under $\imath^{\otimes}$ will be the set $\underline{n}$ with its natural order.

NB: to make contact with the standard notation of $n$--simplices, $[n]=\underline{n+1}$, so that $[0]=\underline{1}$ and $[-1]=\underline{0}=\emptyset$ with the monoidal structure $[n]\amalg [m]=[n]*[m]=[n+m+1]$, where $*$ is the join operation.

\begin{table}
\begin{tabular}{l|l|l}
FC:&underlying $\F$ &definition\\
\hline
$\FFinSet$&$\FinSet$&Finite sets and set maps\\
$\Surj$&$FS$&Finite sets and surjections\\
$\Inj$&$\mathcal{I}\!\it nj$&Finite sets and injections\\
$\FNCSet$&$\NCSet$&Finite sets and set map with orders on the fibers\\
&&aka.\ noncommutative sets\\
$\Delta_+S$&$\Delta_+S$&Augmented crossed simplicial group\\
$\Surj^<$&$FS^<$& Finite sets and surjections with orders on the fibers
\end{tabular}
\caption{\label{table1} Set based Feynman categories Feynman categories. $\V=\underline{1}$ is trivial.}
\end{table}
\begin{table}
\begin{tabular}{l|l|l}
non-$\Sigma$ FC&underlying $\F$&definition\\
\hline
$\FF_<\fS$&$\FinSet_<$&Ordered finite sets and order preserving maps. \\
$\FF_<\fS,$&$OS$&Ordered finite sets and order preserving surjections\\
$\FF_<\fI$&$OI$&Ordered finite sets and order preserving injections\\
$\Delta_+$&$\Delta_+$&Augmented Simplicial category, Skeleton of $\FinSet_<$\\
$\Int^{op}$&$OI_{*,*}^{op}$&Subcategory of $\Delta^{op}_+$ of double base--point \\
&&preserving injections\\
\end{tabular}
\caption{\label{table2} Set based non-$\Sigma$ Feynman categories. $\V=\triv$ is trivial.}

\end{table}

\begin{prop} The following Feynman categories have   $\Vmods_\C=\Obj(\C)$ and he following $\ops$:
\begin{enumerate}
\item For $\FF_<\fS$: the
$OS$-$\ops_\C$ is equivalent to the category of non--unital associative monoids in $\C$.
\item For $\FF_<\fI$: the $OI$-$\ops_\C$ are equivalent to pointed objects in $\C$.
\item For $\FFinSet_<$: the $\FinSet_<$-$\ops$ are pointed unital associative monoids.
\end{enumerate}
\end{prop}

\begin{proof}
The proof is as above, save the action of the symmetric groups, which is not present. Hence there is no commutativity condition. For the unit, since there is no commutativity, we have two relations between $\pi_2$ and
$i:\pi_2\circ (id _1\amalg i)=\pi_2\circ (i \amalg id _1)=id _1$ giving the left and right unit equations.
\end{proof}

\begin{rmk}
Again, at this point the $\Fops$ are monoidal functors not simply functors, but see \S\ref{freemonoidpar} below.
\end{rmk}

\subsubsection{Hybrids}
To obtain the symmetric Feynman category whose $\ops$ are associative algebras or unital associative algebras one has to consider ordered sets with set maps and orders on the fibers.    $\Aut(n)$ acts trivially on the morphism $\pi_n$, which was the reason for the commutativity. To remedy the situation, we notice that  on an ordered $(S,<)$, $\Aut(S)$ acts transitively on the orders of $S$. Thus adding an order to the fibers of $\pi_n$, the different orders will prevent from  obtaining the same map by pre--composing with elements of $\Aut(S)$.

Let $\NCSet$ (noncommutative sets), cf.\ \cite{Loday,NCSetPira,NCSet} be the category whose objects are finite sets. And whose morphisms from $S$ to $T$ are pairs $(f,<_f)$ where $f:S\to T$ and $<_f$ is a
collection of orders $<_{f^{-1}(t)}, t \in T$ on the fibers $f^{-1}(t)$ of $f$. Composition is given by lexicographic composition of orders. For $S\stackrel{g}{\to}T\to \stackrel{g}{\to} U$,
$(f\circ g)^{-1}(u)=f^{-1}(g^{-1}(u))=\amalg_{s\in g^{-1}(u)}f^{-1}(s)$
so that every and the order is given by $t'<_{f\circ g}t$ if $t$ and $t'$ are in the same fiber $f^{-1}(s)$ and $t'<_f t$, or if $t'$ is in the fiber of $s'$ and $t$ is in the fiber of $s$ and $s'<_g s$.
Since isomorphisms in $\FinSet$ have one element fibers, they remain isomorphisms in $\NCSet$.

The skeleton of this category is known by the name of augmented crossed simplicial group defined by the symmetric groups  $\Delta_+S$ (aka. $\Delta\Sigma)_+$), \cite{Lodaycrossed}. In the simplicial notation $\Aut([n])=\SS_{n+1}$.

We let $\FNCSet$ be the Feynman category $(\triv,\NCSet,\imath)$ and $\Surj_<$ the Feynman subcategory whose morphisms have underlying maps that are surjections.

It is easy to check that these are Feynman categories. They are also examples of enriched Feynman categories as discussed in \S\ref{enrichedpar}. It is also obtained from a plus construction.

\begin{prop}
The category $\NCSet\dashops_\C$ and respectively $\Surj_<\dashops_\C$ are equivalent to unital associative monoids (aka.\ algebras) in $\C$ and to the category of possibly non--unital associative monoids  (aka.\ algebras) in $\C$ respectively.

There is an embedding $i:\Surj^<\to \FNCSet$, $i^*$ forgets the unit and $i_!$ is the free adjunction of a unit to an algebra.
\end{prop}

\begin{proof}
As above, on objects, the monoidal functors $\O$ are fixed by the value $\O(\uone)=:A$ up to equivalence, since then up to equivalence $\O(S)=A^{\otimes S}$.
Starting with surjections, we see that these are generated up to isomorphism by the $\pi_n:\un\ta\underline{1}$ and a choice of order on the fiber, that is a choice of order on $\un$.
Let $\mu:=\O(\mu_2,1<2)$ then $\mu:A\otimes A\to A$ yields the multiplication. Associativity follows directly.
If we are in $\FinSet_<$, then we add the inclusion $i:\emptyset \to \uone$ as a generating morphism. The unique fiber is empty and has the empty order.
As before  $\eta:=\O(i):\unit_\C\to A$ provides the unit. This yields the functor from $\ops$ to algebras exhibiting the equivalence.

For the other direction, notice that
if $\tau_{1,2}:\underline{2}\to \underline{2}$ exchanges $1$ and $2$ $(\pi_2,2<1)=(\pi_2,1<2)\circ \tau_{1,2})$. More generally transpositions generate $\SS_n$, which acts transitively on the orders of the fiber of $\pi_n$. Hence, the identity map $id_{\uone}$, $(\pi_2,1<2)$  generate all surjections up to isomorphism, which are permutations of the orders of the fibers. The latter are generated by  transpositions. These maps  together with $i$ generate all maps, thus fixing their values yields a functor in the reverse direction.
Here one uses,  that  transpositions are mapped to commutativity constraints; hence e.g.\ $\O(\pi_2,2<1)=\O((\pi_2,1<2)\circ \tau_{1,2})= \mu \circ C_{A,A}=\mu^{op}$.

The last statement is straight-forward.
\end{proof}

The following is straightforward:
\begin{prop}
\label{asscoveprop}
$\fb(v_b,f_b):\FNCSet\to \FinSet$ with $f_b:\NCSet\to \FinSet$ given by the identity on objects and defined on morphisms as the forgetful functor $f_b:(f,<_f)=f$ is a strong cover, but not strict.

The pull--back $f_b^*$ is the inclusion of unital commutative algebras into unital algebras.
The push--forward is the Abelianization.

For the non--unital versions {\it mutatis mutandis} the same results hold for the restriction of $\fb:\Surj^{<}\to \Surj$.
\qed
\end{prop}
We show in Lemma \ref{assocenrichedlem} that this is an indexed enrichment. Going one level higher,
the enrichment is by the associative operad, which can be obtained via a push--forward along a forgetful map from a plus construction, see Lemma \ref{nscoverlem}. The relation between the two is in Remark \ref{assdecoenrmk}.

\begin{table}
\begin{tabular}{l|l}
$\FF$&$\fops$ equivalent to\\
\hline
$\FinSet$&unital commutative monoids/algebras\\
$\Surj$&commutative monoids/algebras\\
$\FNCSet$&unital associative monoids/algebras\\
$\Surj_<$& associative monoids/algebras\\
$\Inj$&pointed objects.
\end{tabular}
\caption{\label{algtable}Feynman categories based on finite sets and their $\ops$}
\end{table}

\subsubsection{Graphical interpretation}
 \label{graphpar}
A convenient graphical notation to write down a map with ordered fibers is given by planar planted corollas.

First, the fibers of a morphism $f:S\to \{t\}$ give a planted corolla $*_{S\amalg \{t\},t}$, that is a corolla with flags $S\amalg \{t\}$ and root $t$.
Vice--versa, any morphism $S\ta T$ can then be encoded by a forest of planted corollas
$\amalg_{t\in T} \; *_{f^{-1}(t)\amalg \{t\},t}$.
Note that empty fibers correspond to 0-ary corollas. The map is a surjection, if there are no zero-ary corollas, and an injection, if all the corollas are either 1-ary or 0-ary.

An order on the fibers $<_f$ extends to an order $<$ on $f^{-1}(t)\amalg \{t\}$, by considering $t$ to be the first element. That is the fibers can be viewed as planar planted corollas $*_{S\amalg \{t\},t,<}$. And, any morphism in $\FinSet_<$ can be written as an  forest of planar planted corollas indexed by ${t\in T}$.
An example is given in Figure \ref{plcorfig}.

\begin{figure}
  \includegraphics[height=2cm]{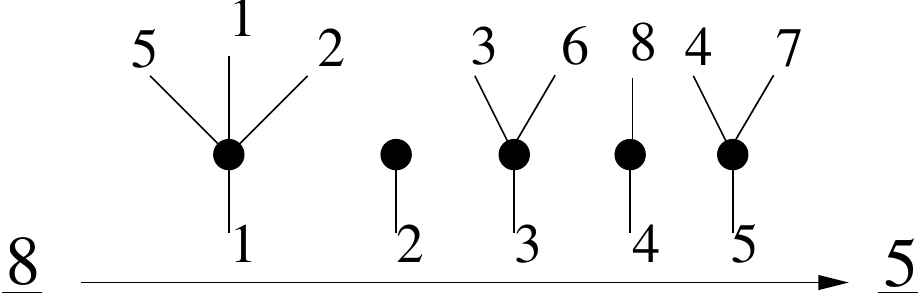}
  \caption{\label{plcorfig} A morphism $f$ from $\underline{8}=\{1\kdk,8\}$ to $\underline{5}=\{1\kdk 5\}$ with the orders and fibers:
  $2<1<5$ on $f^{-1}(1)$,
$\emptyset$ on $f^{-1}(2)$,
$3<6$ on $f^{-1}(3)$,
$8$ on $f^{-1}(4)$, and
$7<4$ on $f^{-1}(3)$.
  }
\end{figure}

\subsubsection{Graphical Feynman categories}
There is a Feynman category $\GG=(\Crl,\Agg,\imath)$ whose technical definition with all details is in Appendix \ref{GGdefsec}. It uses the technical framework of \cite{BorMan,feynman} which is also given in Appendix \ref{graphsec} to be self--contained.

To give the basic structure without all the details, we note that $\Crl$ is equivalent to a the category $\Iso(\FinSet)$ of finite sets and isomorphisms. If $S$ is a finite set, then it determines an object $*_S$ of $\Crl$ which is a corolla with $S$ with external flags and any bijection $S\leftrightarrow S'$ yields an isomorphism $*_S\leftrightarrow *_{S'}$. Consequently $\Iso(\Agg)\simeq \Iso(\FinSet)^\otimes$. Its elements are collections or aggregates of corollas. The morphisms between aggregates are rather complicated.
They are morphisms of the aggregates thought of as graphs without edges as defined in \cite{BorMan,feynman}. Given  a morphism between aggregates, $\phi$, there is an underlying graph, the ghost graph of $\phi$, which is denoted by $\gh(\phi)$. The ghost graph does not fix the morphism uniquely. It does fix the isomorphism class of a basic morphism ---that is morphism whose target is a single corolla. The extra data of for a basic morphism is given by an identification of the tails of the ghost graph with the tails of the target corolla.

A {\em graphical Feynman category} is a Feynman category indexed over $\GG$.

The $\ops$ for graphical Feynman categories include all the known operad types as well as rather new ones, see Table \ref{zootable} in Appendix \ref{operadtypepar}.
Moreover all these examples can be obtained via the operations below, especially decoration,  and taking sub--Feynman categories, aka.\ restriction.

\section{Constructions and Examples}
\label{constructionpar}
\subsection{Functors and lax monoidal functors as $\ops$}
\label{freemonoidpar}
\label{ncpar}
By definition $\Fops_\C$ are strong monoidal functors ${\it Fun}(\F,\C)$.
This begs the question if there are constructions of Feynman categories to obtain lax monoidal functors ${\it Fun}_{lax \otimes}(\F,\C)$ or
simply functors ${\it Fun}(\F,\C)$ as $\ops$, which is indeed the case:
\begin{thm}{\rm \cite[\S3]{feynman}}
Given a Feynman category $\FF$, there are Feynman categories
$\FF^\boxtimes=(\V^\boxtimes, \F^\boxtimes, \imath^\boxtimes)$ and
$\FF^{nc}=(\V^{nc},\F^{nc},\imath^{nc})$\footnote{nc stands for non--connected}
such that

\begin{eqnarray}
\F^\boxtimes\text{-}\ops&=&{\it Fun}(\F,\C)\\
\F^{nc}\text{-}\ops&=&{\it Fun}_{lax \otimes}(\F,\C)
\end{eqnarray}

\end{thm}

The original statements are in \cite[\S3]{feynman}. We give the  constructions below, filing in some details concerning units.

\begin{ex}
As announced, $\Delta_+^{\boxtimes}\text{-}\ops$ are augmented simplicial objects,
$FI^{\boxtimes}\text{-}\ops$ are $FI$-modules etc., see Table \ref{boxtable}.

\begin{table}
\begin{tabular}{l|l|l}
\label{boxtable}
$\FF$&$\FF^{\boxtimes}\dashops_{\C}$&$\FF^{\boxtimes}\dashops_{\C^{op}}$\\
\hline
$\Delta_+$&augmented co--simplicial objects&augmented simplicial objects\\
$\F_<\fI$&co--semi--simple objects&semi--simple objects\\
$\Inj$&$FI$--algebras&$FI$-co-algebras\\
$\Surj$&$FS$--algebras&$FS$--co--algebras\\
\end{tabular}
\caption{Examples of $\FF^\boxtimes$ whose $\ops$ in $\C$ and $\C^{\it op}$ are familiar notions}
\end{table}

\end{ex}

\subsubsection{Free construction $\FF^{\boxtimes}$}
For this $\F^\boxtimes$ is the free (symmetric) category on $\F$. We use $\boxtimes$ for the new free monoidal structure, which we also call outer.
$\V^{\boxtimes}=\V^\otimes \simeq \F$.
The basic morphisms ``are" the morphisms of $\F$: $(\F^{\boxtimes}\downarrow \V^{\boxtimes})\simeq (\F^\boxtimes \downarrow \F)=(\F\downarrow \F)$ under the equivalence $\imath^{\otimes}: \V^\otimes  \to \F$ and by the definition of the free (symmetric) monoidal category.

In the free monoidal category $\F^\boxtimes$ there is a free unit $\unit_\boxtimes$ which can be different from $\unit_\otimes$, thus for $\hat\O\in \F^\boxtimes\text{-}\ops_\C$, we have
that $\hat\O(\unit_\boxtimes) \simeq\unit_\C$, but no condition on $\unit_\otimes$.

\begin{ex} Examples are given by the Feynman category indexed over  finite sets.

\end{ex}

\subsubsection{NC-construction}
Here again $\V^{nc}=\V^\otimes$ and $\Obj(\F^{nc})=\Obj(F^\boxtimes)$,
but  the basic morphisms $(\F^{nc}\downarrow \V)$ are defined as
$Hom_{\F^{nc}}(\boxtimes_{i}X_i,Y)=Hom_\F(\bigotimes_{i}X_i,Y)$.
This effectively adds the data of functor $\mu:\F\boxtimes \F\to \F$, that is a natural family of  morphisms $\mu_{X,Y}:(X\boxtimes Y)=X\otimes Y$,  and a  morphisms $\eps:\unit_\boxtimes\to \unit_\otimes$, compatible with the unit constraints, to the morphisms of $\F^{\boxtimes}$. The data of $\eps$ was not addressed separately in \cite[\S3.1]{feynman}.

The construction of $\ops$ is as follows: If $\hat \O\in \F^{nc}\dashops_\C$, one defines $\O\in {\it Fun}_{lax \otimes}(\F\,\C)$ as $\O(X):=\hat\O(X)$
and the same on morphisms $\hat\O(\phi)=\O(\phi)$ for $\phi:X\to Y$.

The two structural morphisms are defined as follows:
\begin{equation}
\label{multconstreq}
\O(X)\otimes \O(Y)=\hat\O(X)\otimes \hat\O(Y)\simeq\hat \O(X\boxtimes Y) \stackrel{\hat\O(\mu_{X,Y})}{\to}\hat\O(X \otimes Y)=\O(X\otimes Y)
\end{equation}
and
\begin{equation}
\label{unitconsteq}
\hat\O(\unit_\boxtimes)\simeq \unit_\C\stackrel{\hat\O(\eps)}{\to} \hat\O(\unit_\otimes)
 \end{equation}
 yields the structural unit morphism for the underlying lax-monoidal functor.

Vice--versa, using the structure for the free monoidal category, we can extend a $\O\in \F^{nc}\dashops$ functor to all the morphisms of $\F^\boxtimes$ by using the functor underlying $\O$ and then extending it to the free (symmetric) monoidal category by the universal property as $\O^{\boxtimes}$. Then one only needs to define $\O^{nc}$, we only need the values on $\mu_{X,Y}$ and $\eps$, which are fixed  by the equations \eqref{multconstreq} and \eqref{unitconsteq}.

\begin{rmk}\mbox{}
\begin{enumerate}
\item It is most natural to take $\eps$ to be an isomorphism and moreover to identify $\unit_\boxtimes$ and $\unit_\otimes$. An example is taking the empty forrest to be given by an empty tree, or more generally, the empty sentence (the empty collection of  words) is identified with the empty word.
\item The nc--construction plays a crucial role in the connection to Hopf algebras \cite{HopfPart1,HopfPart2}.
\item An example of nc--construction first appeared in \cite{KWZ}.
\end{enumerate}
\end{rmk}

\subsection{Decorations, covers  and factorizations of morphisms in the category of Feynman categories}
\label{decopar}
Given an $\O\in \Fops_{\Set}$ there is a new Feynman category $\FF_{\dec\O}$ defined in \cite{decorated}.
This decoration is a variation of the Grothendieck (op)--fibration construction.
It also establishes a theory  of covers that is compatible with Galois covers in the sense of Grothendieck \cite[\S3]{BergerKaufmann}.

The objects of $\Fdeco$ are pairs $(X,a_x)$ with $X\in \Obj(\F)$ and $a_x\in \O(X)$ and morphisms are given by $$Hom_{\Fdeco}((X,a_x),(Y,a_y))=\{\phi\in Hom_\F(X,Y)|\O(\phi)a_x=a_y\}$$
Likewise, the objects of $\Vdeco$ are pairs $(*,a_*)$ with $*\in \Obj(\V)$ and $a_*\in \O(\imath(*))$ with the morphisms analogous to the ones above. The inclusion $\ideco$ is 
given by $(*,a_*)\mapsto (\imath(*),a_*)$.

\begin{rmk}
If $\O\in \Fops_{\C^{\op}}$ the condition reads $\O(\phi)(a_y)=a_x$.

\end{rmk}
\begin{thm}{\rm \cite[Theorem 4.1]{decorated}}
Given a functor $\O\in \FF\dashops_{\Set}$ then  $\FFdeco=(\Vdeco,\Fdeco,\ideco)$ as above is a Feynman category and there is a morphism of Feynman categories
$\mathfrak{p}_\O=(p_\V,p):\FFdeco\to \FF$, which forgets the decoration, i.e.\  $p((X,a_x))=X$ and $p_\V$ is its restriction to $\V$.

This construction is functorial in $\O$ and with respect to morphisms of Feynman categories;,i.e.\ the following diagrams commute.
For a morphism (natural transformation) $\eta:\O\to\P$ in $\FF\dashops$ there is a  commutative diagram
\begin{equation}
\xymatrix{
\FFdeco\ar[rr]\ar[dr]_{\mathfrak{p}_\O}&&\FF_{\rm dec \P}\ar[dl]^{\mathfrak{p}_\P}\\
&\FF&
}
\end{equation}
where the upper arrow is given by $(X, a_X)\to (X,\eta_X(a_X))$, where $\eta_X:\O(X)\to \P(X)$ is the natural transformation.

Given a morphism of Feynman categories $\ff=(v,f):\FF\to \FF'$, there is a commutative diagram
\begin{equation}
\label{decosquareeq}
\xymatrix{
\FFdeco\ar[r]^{\ff_\O}\ar[d]^{\fp_\O}&\FF'_{\dec \; f_!(\O)}\ar[d]^{\fp_{f_!(\O)}}\\
\FF\ar[r]^{\ff}&\FF'
}
\end{equation}
\end{thm}

The maps $\mathfrak{p}$ in the theorem above called covers, viz.\ $\mathfrak{f}:\GG \to \FF$ is a cover if $\GG=\FFdeco$ for some $\O\in \Fops$ and $\mathfrak{f}=\mathfrak{p}:\FFdeco\to \FF$.

\begin{prop}
\label{coverprop}
A morphism $\ff=(v,f)$ of  Feynman categories is a cover if
\begin{enumerate}
\item Any morphism whose source is in the image of $v$ respectively $f$, has a lift.
\item Any  lift of a  morphism in the image of $v$ respectively $f$ is uniquely determined by its source.
\end{enumerate}
\end{prop}

\begin{proof}
Given $\ff:\FF_{\dec\O}\to \FF$, we verify the two conditions.
The objects in the image of $f$ are the $X\in \Obj(\F)$ with $\O(X)\neq \emptyset$.
If $X$ is in the image of $f$, i.e.\ $\O(X)$ is not empty then given any $\phi:X\to Y$, $\O(Y)\supset \O(\phi)(X)\neq \emptyset$, so that $Y$ is in the image of $f$ as well.
 Moreover,  for any $a_x\in \O(X)$ fix $a_y:=\O(\phi)(a_x)\in O(Y)$ then $\phi:(X,a_x)\to (Y,a_y)$ is a lift of $\phi$ and  any lift of $\phi$ is uniquely fixed by the choice of  by $a_x$.

Vice--versa, given any morphism $\ff:\FF'\to \FF$ satisfying (1) and (2), gives rise to a functor $\O\in \Fops$, such that $\F_{\dec\O}={\F'}$. On objects $\O(X)=f^{-1}(X)$, that is the set of fibers, aka.\ elements. This will be $\emptyset$ if $X$ is not in the image.
Given a morphism $\phi\in Hom_\F(X,Y)$ in the image fix a $\hat X\in \O(X)$ using (1) and (2)
 there is a unique lift $\hat \phi(\hat X)\in Hom_{\F'}({\hat X}, {\hat Y})$ with $\hat Y=t(\hat \phi(\hat X))$. Define $\O(\phi):\O(X)\to \O(Y)$ by $\O|_{\hat X}:=\hat\phi(\hat X)$.
 If $\phi\in Hom_\F(X,Y)$ and $X$ not in the image of $f$, that is $\O(X)=\emptyset$, then $\O(\phi)$ is the unique map with source $\emptyset:\O(X)=\emptyset\to \O(Y)$.
 Finally, we check that the functor is (symmetric) monoidal. We have that $\O(X\otimes Y)=f^{-1}(X\otimes Y)= \{Z|f(Z)=X\otimes Y\}$ using that $\ff$ is a morphism of Feynman categories,we can decompose $Z$ and
 we have $f^{-1}(X\otimes Y)= \{(Z',Z'')|f(Z')=X, F(Z'')=Y\}=\O(X)\times \O(Y)\}$.

 Decorating with this functor, we get $\Fdeco$ with objects $(X,\hat X)$ and morphism $\hat\phi:(X,\hat X)\to (Y,\hat Y)$ lifting $\phi$. The isomorphism is given by sending $(X,\hat X)$ to $\hat X$,
 the inverse is fixed by $X=f(\hat X)$.
 
 The determination of the groupoid part and  is analogous and the inclusion is then clearly the restriction.
\end{proof}

\begin{ex}
Table \ref{decotable} contains shows examples decorations for graphical Feynman categories $\GG$.
\end{ex}

There is a second type of standard morphism, which is called connected.
\begin{df}
A morphism of Feynman categories $\ff:\FF'\to \FF$ is connected if $f_!(\trivial_{\F'})=\trivial_\F$ where $\trivial_{\F'}:\F'\to \Set$ and $\trivial_\F\to \Set$ are the trivial $\ops$ to $\Set$ with the Cartesian monoidal product $\times$.
\end{df}

The two sets of morphisms form an orthogonal factorization system in the sense of  \cite{BergerKaufmann}, where these types of morphisms are linked to comprehension schemes \cite{LawvereComp} and a general theory of Galois type covers.
The following theorem follows from \cite[Proposition 2.3]{BergerKaufmann}:

\begin{thm}\label{BKthm}
Any morphism of Feynman categories $\ff:\mathfrak{G}\to \FF$ has a unique factorization as $\ff=\mathfrak{p}\circ\mathfrak{i}$
where $\mathfrak{i}$ is connected and $\mathfrak{p}$ is a cover.
\end{thm}
\begin{rmk}\mbox{}
\begin{enumerate}
\item A cover $\mathfrak p=(v,p):\FF'\to \FF$ is isomorphic to the decoration by $p_!(\trivial_{\F'})\in \Fops$, \cite{decorated}.
   \item The decoration construction is also intimately tied to cyclic operads, modular operads and moduli spaces; see \cite{BergerKaufmann,Ddec} and \S\ref{modulispacepar} and Appendix \ref{connectionspar} below.
   \item The existence of a factorization follows already from \cite{decorated}, the uniqueness requires a finer analysis.
 \end{enumerate}
\end{rmk}

\subsubsection{Covers, connected morphisms and indexing}
\begin{lem} \mbox{}
\begin{enumerate}
\item The forgetful functor $\mathfrak{p}:\FFdeco\to \FF$ is an indexing if $\O(X)\neq \emptyset$ for all $X$. Thus restricting to the full image of $p$, we obtain a indexing, which is not strong in general.
\item A strong indexing is a cover, if and only if it is the an isomorphism and hence strict.
\end{enumerate}
\end{lem}

\begin{proof} The image of a cover are precisely those $X$ for which $\O(X)\neq \emptyset$. If a cover is a strong indexing, then there is only one object in the inverse image and thus every morphism has a unique left. Hence, the cover is an isomorphism.
\end{proof}

\begin{rmk}\mbox{}
\begin{enumerate}
\item Any indexing factors as a connected morphism and a cover.
 \item By the above, we see that a strong indexing is connected.
 \item Strong and strict indexings  give rise to enrichments discussed below, \ref{enrichedpar}.
 \end{enumerate}
\end{rmk}

There are interesting connected morphisms which are not strong. These typically arise from inclusions given by restrictions, see \S\ref{operadtypepar} in the Appendix for examples.

\subsubsection{$\Fdeco\dashopcat$}
In \cite{BergerKaufmann}, we showed
\begin{prop}
The $\Fdeco\dashopcat$ are $\fops$ over $\O$, that is $\P\in\fops$ with a natural transformation $\P\to \O$.
\qed
\end{prop}

\subsection{Modules and Plus construction}
\label{pluspar}
Remarkably, Feynman categories often can be used to encode modules as well as algebras.
As constructed above Feynman categories can be used to ``encode'' algebras , see Table \ref{algtable}. One can ask the question if there is a way to encode modules over a given algebra or more generally modules over $\F\dashops$. We will give the answer in two parts. Here we fist consider the case of Feynman categories over $\Set$, in \S\ref{enrichedpar} we will then deal with more general types of modules.
In particular, one would like to consider modules in linear categories. This is possible along the same lines presented here, but technically more involved. This is why we postpone it to the next section. In this section, we will fix the target category $\C$ to be $(\Set,\times)$. The arguments generalize in a straight--forward fashion to a Cartesian target category $\C$.

\begin{ex}[Paradigmatic example]
\label{Assex}
Consider an associative monoid $A$, then
there is an $\CalA\in \surj^<\dashops_{\Set}$ such that $A=\CalA(*)$. The set--modules, aka.\ set-- algebras, over the  associative monoid $A$ are sets $M$ with structure maps $\rho:A\times M\to A, (a,m)\mapsto am$ that satisfy
$a_1(a_2m)=(a_1a_2)m$. The morphisms of modules are  intertwiners. If $A$ is unital, with unit $1\in A$ then there is another  condition for modules: $1m=m$.
In this case $A$ is actually the value of $\CalA\in NCSet\dashops_{\Set}$ and

Assuming that $A$ is unital, we can consider   $\underline{A}$ cf.\ \S\ref{monoidpar}. The category of $A$--modules in $\Set$ is ${\it Fun}(\underline A,\Set) \simeq {\it Fun}_\otimes (\underline{A}^\otimes,\Set)$ and intertwiners
are natural transformations.
To separate out the isomorphisms, splits as a disjoint union $A=A^\times\amalg A^{red}$ where $A^\times=G(A)$ are the invertible elements, then $\V_\A=\underline{A^\times}$ and $\F=\underline{A}^\otimes$ together with the natural
inclusion form a Feynman category $\FFtriv_\A=(\underline{A^\times},\underline{A}^\otimes,\imath)$.

Let $\mathfrak{V}_\A$ be the trivial Feynman category on $\V_\A$.
There is a natural functor $\FFtriv_\CalA\to \mathfrak{V}_\A$ which is a strong indexing. It is identity on $A^\times$ and sends $A^{\red}$ to $id_*$.
The indexing is strict, if $A$ is reduced that is $A^\times=1$.
\end{ex}

The basic results  in the theory, see \cite[\S\S 3.6,3.7]{feynman} are that there is a   plus construction $\FF^+$ for a given Feynman category and that there is a quotient of it $\FF^{\it hyp}$ which is called the hyper construction. The latter is important for twisting as in \cite{GKmodular}, see \S\ref{Kpar}, \S\ref{suspensionpar} and is needed in \S\ref{barpar}..
Being more careful with the units,
we give a new construction $\FF^{+ \it gcp}$ which allows one to define modules via indexed enrichments, see \cite[\S4.1]{feynman} and \S\ref{enrichedpar}, especially \S\ref{Assenrichedex} below.

In the example above $(\FFtriv)^+=\Surj^<, (\FFtriv)^{+ \it gcp}=\FNCSet,(\FFtriv)^{\it hyp}=\FFtriv$ and $\CalA\in \NCSet\dashops_\Set$ gives rise to
$\FFtriv_\CalA$ indexed over $\mathfrak{V}_\A$. If $A$ is reduced, then $\mathfrak{V}_\A=\FFtriv$.

We now describe these construction  adding more  details to the condensed presentation in \cite[\S3]{feynman}.
A graphical version of these constructions is given in Appendix \ref{graphpluspar}.
This graphical treatment provides a solid combinatorial language to write out the proofs in detail, improving the level of precision over that of \cite{feynman}.
These, however, use the full strength of graph formalism of \cite{BorMan,feynman}, which is reviewed in Appendix \ref{graphsec}.
With this in mind, we relegated the more detailed proofs to the appendix to not hinder the flow.

\subsubsection{A look ahead}
For clarity, we will deal first with the case of $\C=\Set$ and relay the subtleties of enrichment to \S\ref{enrichedpar}.
In full generality, for any split $\O\in  \F^{+ \it gcp}\dashops_\E$ there is an indexed enriched Feynman category  $\FF_\O$ enriched over $\E$ whose $\V_\O$ is a freely enriched groupoid. With this construction, we can defined the sought after Feynman category for modules.

\begin{df}
An algebra over $\O\in \F^{+ \it gcp}\dashops$, aka.\ $\O$--module,  is an element of $\F_\O\dashops$.
\end{df}

Thus we can define modules over an $\O\in \fops_\C$ if $\FF=\tilde{\FF}^+$, for some $\tilde{\FF}$ and $\O$  lifts to $\FF^{+ \it gcp}\dashops_\C$.

\subsubsection{Plus construction}

Fix a Feynman category $\FF=(\V,\F,\imath)$ and define a new Feynman category as follows.
Set $\V^+=\Iso(\F\downarrow \V)$, that is basic morphisms and their isomorphisms. The objects of $\F^+$ are the morphisms of $\F$: $\Obj(\F^+)=\Mor(\F)$
and $Iso(\F^+)=Iso(\F\downarrow\F)$. The isomorphism are the $\sds:\phi\to \phi'=\sigma'\circ\phi\circ\sigma^{-1}$ of \eqref{sdseq}. There is a natural inclusion of $\imath^+:\Iso(\F\downarrow\V)\to \Iso(\F\downarrow\F)$.

The other morphisms obtained by decomposing the source morphism into basic morphisms and then composing these basic morphisms using concatenation and tensor products to obtain the target morphism.
More precisely, consider $\phi:= \phi_0\ot \phi_1\odo\phi_n$ such that  $\psi=\phi_0\circ (\phi_1\odo \phi_n)\in (\F\downarrow\V)$ is well defined.
There will be one basic generating morphism $\phi\to\psi$ for such a pair denoted by $\gamma(\phi_0;\phi_1\kdk\phi_n)$.

A general morphism in $\F^+$ is a concatenation of tensor products,  generating morphisms and isomorphisms modulo the relations of a monoidal category, that is associativity, units, interchange and equivariance  under the action of isomorphisms given above.

\begin{prop}
$\FF^+=(\V^+,\F^+,\imath^+)$ is a Feynman category.
\end{prop}
\begin{proof}
Condition \eqref{objectcond} for $\FF^+$ is the condition \eqref{morphcond} for $\FF$. For condition \eqref{morphcond} fix any morphism $\Phi:\phi\to \psi$.
We first show the existence of the decomposition \eqref{feydecompeq}. By condition \eqref{morphcond} for $\FF$, there exist isomorphisms
$\sds:\phi\to \phi'=\bigotimes_{v\in V}\phi_v$ and  $(\tau\Downarrow\tau'):\psi\simeq \psi'=\bigotimes_{w\in W}\psi_w$.
Thus we can assume that both $\phi$ and $\psi$ are decomposable.
 The statement then follows from the Theorem \ref{forestthm} in Appendix \ref{graphpluspar}.
 A short version is that any iteration of morphisms is given by a flow chart with input the source
 of $\psi$ and the output, the target of $\psi$. The tensor product acts as
 disjoint union on the flow charts.
 Decomposing the target decomposes the flow charts by following the sources upwards. Such a chart is
 connected since otherwise the target would not lie in $\V$ and thus the decomposition is into connected
 components, with the compositions along  these connected components yielding the $\psi_v$.
The last axiom holds due to the axiom \eqref{smallcond} for $\FF$.
\end{proof}

\begin{lem}
 Strictifying $\F$, we have
$\V^+$ is equivalent to $\Iso(\imath^\otimes \downarrow \imath)$,
$\Iso(\F^+)$ is equivalent to $\Iso(\imath^\otimes \downarrow \imath^\otimes)
\simeq \Mor(\V)^\otimes$.
 Assuming strict associativity constraints,
 this is generated by commutativity constraints and morphisms of $\V$.
 $\Mor(\F^+)$ is generated by these isomorphisms and the morphisms
 $\phi=\phi_0\odo\phi_n\to \psi=\phi_0\circ (\phi_1\odo\phi_n)\in (\imath^\otimes\downarrow\imath)$.
\end{lem}

\begin{proof}
Note that by axiom (ii) $\Iso(\F\downarrow \F)$ is equivalent
to $(\Iso(\F\downarrow \V))^\otimes$, and by the definition above $\F^+$
 can be made strict by considering only objects in
 $(\F\downarrow\V)^\otimes\simeq(\imath^\otimes\downarrow \imath)^\otimes$
 together with their isomorphism, which are fixed by axiom (ii), and the
 generating morphisms
 $\gamma(\phi_0;\phi_1\kdk \phi_n):\phi=\phi_0\odo \phi_n\mapsto
 \psi=\phi_0\circ (\phi_1\odo \phi_n)\in (\imath^\otimes\downarrow\imath)$.
\end{proof}
The notation $\gamma$ was chosen to be reminiscent of operadic compositions. Note that in the
general case, there is a condition of composability.

\begin{cor}
\label{fplusupscor}
$\F^+\dashops_\C$ is equivalent to the category of strict monoidal functors $\D$ on the strictification, which
are uniquely determined by the data:
\begin{enumerate}
\item (Groupoid rep) A functor $\D$ from $\V^+=\Iso(\V^\otimes\downarrow \V)$ to $\C$ which is given by:
\begin{enumerate}
\item (Object data): $\D:\Obj(\V^\otimes \downarrow \V)\to \Obj(\C)$. That is an object $\D(\phi)$ of $\C$ for each basic morphism $\phi$.
\item (Iso data): Actions of the isomorphisms $\sds$.
That is a left action of isomorphisms of $\Mor(\V)^\otimes$
and a right action of $\\Mor(\V)$ on the $\D(\phi)$.
    \end{enumerate}
\item (Composition data) Morphisms
$\D(\gamma):\D(\phi_0)\otimes \bigotimes_{i=1}^n\D(\phi_i)\to \D(\psi)$
where $\psi=\phi_0\circ(\phi_1\odo \phi_n)\in (\F\downarrow \V)$.
\end{enumerate}
\end{cor}
Note that the groupoid data states that if $\phi'\simeq \phi$ via $\sds$ in $(\F\downarrow \F)$, then $\D(\sds):\D(\phi)\stackrel{\sim}{\to}\D(\phi')$.
In particular,  if $\sigma,\sigma'\in \Mor(\V)$  are in same isomorphism class, that is are maps between elements of $\V$ which are in the same connected component of $\V$, then $\D(\sigma)\simeq\D(\sigma')$.

The corollary allows us to compute the first examples, which are the start of a ladder of complexity.
\begin{prop} \mbox{}
\label{plusrepprop}
\begin{enumerate}
\item For $\FFtriv$, the $\ops$ are  equivalent to associative monoids, aka.\ algebras.
\item For $\FinSet^+$ the $\ops$ are  equivalent to operads and for $\Surj^+$ to operads without constants.
\item For $\FinSet^+_<$ the $\ops$ are  equivalent to non--symmetric  operads and for $\Surj^+_<$ non--symmetric to operads without constants.
\end{enumerate}
\end{prop}
For the reader unfamiliar with operads the latter two statements  can serve as a definition, see
\eqref{operadgammaeq} below.

\begin{proof} To calculate the groupoid data:
$\V=\triv$, $\V^+=\Iso(\V^\otimes \downarrow \V)=\Mor(\V)=id_*$. So the groupoid part of $\D$ is fixed by an object $A:=\D(id_*)$.
The composition data is morphism $\mu:\D(\gamma(id_*;id_*):\D(id_*\amalg id_*)=A\otimes A\to \D(id_*\circ id_*)=\D(id_*)=A$.
The associativity of $\mu$ follows from the associativity $(id_* \circ \id_*)\circ id_*=id_*\circ (id_*\circ id_*)$.

For $\FinSet^+$, $\V=\triv, \V^{\otimes}=\SS$ and $\V^+=\Iso(\SS\downarrow \{*\})$
has as objects the maps $\un\to \{*\}$ where $n$ may be zero, where $\underline{0}=\emptyset$. Since $\{*\}$ is a one-element set,
there is a unique map $\pi_n:\un\to \{*\}$ and hence the objects of $\V^+$ are  given by $\N_0$. As the action of isomorphisms on the target $\un$ is trivial, the isomorphisms are exactly the isomorphism of the source $\un$ and hence $Aut(\pi_n)=\SS_n$ and there are no isomorphisms between $\pi_n$ and $\pi_m$ for $n\neq m$. Thus $\V^+ \simeq \SS$, and the groupoid data  is a functor $\O:\SS\to \C$ that is an $\SS$--module which is a collection of objects $\O(n)=\O(\pi_n)$ together with an action of $\SS_n$ on $\O(n)$, were we reverted from the notation $\D$ to the generic $\O$ for elements of $\ops$, which also conforms with the usual operad notation. The composition data is given by the maps, which define an operad. Denoting $\O(\gamma(\pi_m;\pi_{n_1}\kdk\pi_{n_m}))=\gamma_{m;n_1\kdk n_m}$
\begin{equation}
\label{operadgammaeq}
\xymatrix{
\O(m)\otimes \O(n_1)\odo \O(n_m)\ar@-[rr]^{\gamma_{m;n_1,\dots, n_m})}\ar@{=}[d]&&\O(n)\ar@{=}[d]\\
\O(\pi_m\amalg \pi_{n_1}\amalg\cdots\amalg \pi_{n_m})&& \O(\pi_n)
}
\end{equation}
where $n=\sum n_i$. The associativity gives a condition on $\gamma$ as does the $\SS_n$ actions. These are spelled out in any text on operads, see e.g.\ \cite{MSS,woods}
or \cite[\S2.2.5]{HopfPart1}  for a formula using indexing.

If we only have surjections, the map $\pi_0$ is missing and hence there is no $\O(0)$, in other words, there are no constant terms.

If we have an order on the fibers of the maps, then the objects of $\V^+$ are $(\pi_n,<_{\un})$ with isomorphisms acting transitively on the orders,
that is we have as a groupoid the objects $\{\SS_n\}, n\in \N_0$ with $\SS_n$ acting on $\SS_n$ as automorphisms via the regular representation.
Its skeleton is simply $\N_0$ as a discrete category. A representative for each isomorphism class is given by $(\pi_n,<)$ where $<$ is the standard order on $\un$.
Thus the groupoid part is simply given by a sequence of objects $\O(n)$.
Composition is as above where $\gamma$ is fixed  above and it has to satisfy the condition of associativity.

\end{proof}
\begin{rmk}
\label{operadrmk}
One can ask about the pseudo--operad structure.
Using the subset with $\phi_i=\pi_n$, and all other $\phi_j=\pi_1$ we obtain maps
\begin{equation}
\gamma_{n;1\kdk,1,m,\kdk 1}:\O(n)\otimes \O(1)\odo \O(1)\otimes \O(m)\otimes \O(1) \odo \O(1)\to \O(n+m-1)
\end{equation}
To obtain the usual pseudo--operad maps $\circ_i$, we at this time lack an operadic unit.

If we have a unit $u:\unit\to \O(1)$ which is a unit for the operation $\g$:
\begin{eqnarray}
\label{leftuniteq}
\gamma_{1;n} \circ (u\otimes id_{\O(n)})&=&id_{\O(n)}\\
\label{rightuniteq}
\gamma_{m;1\kdk 1}\circ (id_{\O(m)}\otimes u^{\otimes m})&=&id_{\O(m)}
\end{eqnarray}
where we tacitly used the unit constraints.
Thus if there is a unit and
   then we can define
\begin{multline}
\label{circieq}
\circ_i:\O(n)\otimes \O(m)\simeq \O(n)\otimes \unit\odo \unit \otimes \O(m)\otimes
\unit \odo \unit \\
\stackrel{id_{\O(n)}\otimes u \odo u \otimes id_{\O(m)} u \odo u}{\longrightarrow}
\O(n)\otimes \O(1) \odo \O(1) \otimes \O(m)\otimes
\O(1) \odo \O(1)\\
\stackrel{\gamma_{n;1\kdk 1,m,1\kdk 1}}{\longrightarrow}  \O(n+m-1)
\end{multline}
and using \eqref{leftuniteq} and \eqref{rightuniteq}, we can recover $\g$ from \eqref{circieq} by using iterated $\circ_i$ operations and associativity.

However, if $\O$ is reduced, which means that $\O(1)=\unit$
the composition data factors trough unit constraints
then $\O$ is automatically unital. In the general case,
this is formalized by the definition below.

\end{rmk}

\begin{df}
\label{functortypedef}
A groupoid compatible pointing for a functor $\D\in \F^+\dashops$
is a collection of elements $u_\sigma:\unit_\C \to \D(\sigma)$, $\sigma\in \Mor(\V)$,
which satisfy $\D(\sigma)\circ u_{\sigma'}=u_{\sigma\circ\sigma'}$ and are compatible
with  groupoid action and  the composition data $\D(\gamma)$. I.e.\
the following diagrams commute:

\begin{equation}
\label{urightcompateq}
\xymatrix{
\D(\phi_0)\otimes \D(\sigma_1)\odo \D(\sigma_n)\ar[r]^{\D(\gamma)}&
\D(\phi_0\circ (\sigma_1\odo\sigma_n))\\
\D(\phi_0)\otimes \unit\odo \unit
\ar[u]_{id\otimes u_{\sigma_1}\odo u_{\sigma_n}}
&\ar[l]_{\text{unit constraints}}^\simeq
\D(\phi_0)
\ar[u]_{\D((\sigma^{-1}_1\odo\sigma^{-1}_n\Downarrow id))}^\simeq
}
\end{equation}
where the right morphisms is given by the groupoid data and
\begin{equation}
\label{uleftcompateq}
\xymatrix{
\D(\sigma_0)\otimes \D(\phi_1)\ar[r]^{\D(\gamma)}
&\D(\sigma_0\circ \phi_1)\\
\unit \otimes \D(\phi_1)\ar[u]_{u^{\sigma_0}\otimes id}
&
\D(\phi_1)\ar[l]_{\text{unit constraint}}^{\simeq}
\ar[u]_{\D((id\Downarrow \sigma_0))}^{\simeq}
}
\end{equation}

A functor $\D$ is called reduced if $\D(\sigma)\simeq\unit$ for all $\sigma\in \Mor(\V)$.

A functor $\D$ and a choice of groupoid compatible pointing, is called {\em groupoid compatibly pointed (gcp)} functor.

A functor $\D\in \F^+\dashops$ is a {\em hyper--functor} if it is reduced and gcp using the identification $\D(\sigma)\simeq\unit$.

\end{df}

\begin{rmk}
\label{unitalrmk}
\mbox{}
\begin{enumerate}

\item  Due to the groupoid data, to check that  $\D$ is reduced it suffices to check $\D(id_*)\simeq \unit$ for a set of representatives $*$ for the isomorphism classes of $\V$.

 \item \label{usigmapart}
Due to the compatibility with the action of the groupoid, the $u_\sigma$ are also already fixed by a choice of the $u_{id_*}$  where  $*$ runs through representative of  the isomorphism classes of objects of $\V$. Concretely, if $\sigma: *\to *'$ an isomorphism, then
$\D((id\Downarrow\sigma))\circ u_{id_*}=u_\sigma$.
\item From \eqref{uleftcompateq} and \eqref{urightcompateq} it follows that
 $\D(\gamma(\sigma;\sigma'))\circ (u_{\sigma}\otimes u_{\sigma'})=u_{\sigma\circ\sigma}$.

\end{enumerate}

\end{rmk}
\begin{ex}

For $\FFtriv$, we retrieve the motivating Example \ref{Assex}. The $(\FFtriv)^+=\ops$ are associative mo\-noids in $\C$. There is again only one isomorphism
in the skeleton of $\V$ namely $id_*$.
Gcp means that the monoids are unital and reduced means that $M\simeq\unit$.
Thus a hyper--functor is trivial.

For $\FinSet$, we retrieve Remark \ref{operadrmk}. The $\FinSet^+\dashops$ are operads,  the gcp respectively reduced $\FinSet^+\dashops$ are unital operads, respectively the reduced operads. The hyper functors are unital reduced operads.

\end{ex}

\begin{rmk}
For a gcp functor $\D$, we can define the analogue of the $\circ_i$ of \eqref{circieq}.
\end{rmk}
\begin{df} For a gcp functor $\D$, let $\phi:*_1\odo *_n\to Y$
and $\psi:*'_1\odo *'_n\to *_i$,
then define:

\begin{multline}
\label{circigeneq}
\circ_i:\D(\phi)\otimes \D(\psi)\simeq \D(\phi)\otimes \unit\odo \unit \otimes \D(\psi)\otimes
\unit \odo \unit \\
\stackrel{id_{\D(\phi)}\otimes u_{id_{*_1}} \odo u_{*_{i-1}} \otimes id_{\D(\psi)}
u_{*_{i+1}} \odo u_{*_n}}{\longrightarrow}
\D(\phi)\otimes \D(id_{*_1}) \odo \D(id_{*_{i-1}}) \otimes \D(\psi)\\
\otimes
\D(id_{*_{i+1}}) \odo \D(id_{*_n})\\
\stackrel{\gamma}{\to}  \D(\phi\circ id_{*_1}\odo id_{*_{i-1}}\otimes \phi)
\end{multline}
\end{df}

\begin{rmk}
The difference between this and the equation \eqref{operadgammaeq} is that there
are possibly many objects
in $\F$ and hence more identities. Using the general compatibility  conditions
\eqref{rightcompateq} and \eqref{leftcompateq}, we can recover the $\g$ from the $\circ_i$.
Thus the $\circ_i$ give another generating set which is ``local'' in the sense that it involves only two morphisms and
equivalently only one edge in the graphic description of Appendix \ref{graphpluspar}.
\end{rmk}

\subsubsection{Signs as hyper--functors}
\label{Kpar}

For a finite-dimensional $k$ vector space $V$ of dimension $n$, let $\det(V)$
be the graded vector space
$
\det(V) = \Sigma^{-n} \Lambda^nV ;
$
this is the one-dimensional top exterior power of $V$, concentrated in
degree $-n$. If $S$ is a finite set, let $\det(S)=\det(k^S)$. There is a natural action of $\Aut(S)$ on $\det(S)$. Choosing an order on $S$, thereby identifying the set $\Aut(S)$ with $\SS_n$ as an $\SS_n$-module $\Sigma^{-|S|}\sign(\SS_{|S|})$, where $\sign$ is the sign representation.

\begin{ex}[Signs: $\fK$]
\label{Kex}
For $\GG$  (see Appendix \ref{graphsec}, especially \S\ref{GGdefsec}), we define  $\fK\in\Agg^+\dashops$ as follows.

\begin{equation}
\fK(\phi)=\det(E_{ghost}(\phi))
\end{equation}
The composition is given by $\fK(\phi_0)\otimes \fK(\phi_1)\odo \fK(\phi_n)=\det(E_{ghost}(\phi)\otimes \det(E_{ghost}(\phi_1)) \odo \det( E_{ghost}(\phi_n))\to \det(E_{ghost(\phi_0)\amalg E_{ghost}(\phi_1)}\amalg \cdots \amalg E_{ghost}(\phi_n))= \det(E_{ghost}(\phi_0\circ (\phi_1\odo \phi_n))$ with the identification according to \eqref{icompeq}, see also Lemma \ref{ghostamalglem}.
Since an isomorphism $\sigma$ does not have any edges, $\fK(\sigma)=\unit=k$ and the composition is simply given by the unit constraints.
Hence $\fK$ is a hyper--functor.
This generalizes \cite{GKmodular}.

As the twist is a hyper--functor, for the compositions data it is enough to give the data of the $\circ_i$
which are simply
$\fK(\phi)\circ_i\fK(\psi)=\det(E_{ghost}(\phi)\amalg E_{ghost}(\psi))$ and check that these are appropriately associative, which is
straightforward.

The fact that this gives a hyper-functor basically boils down to the fact that the ghost graph uniquely determines the isomorphism class of a basic morphism in $\Agg$.
\end{ex}

\begin{ex}[Homology twist]
We similarly define ${\it Det}$ via ${\it Det}(\phi)=\det(H_1(\gh(\phi)))$.
\end{ex}
\subsubsection{$\V$--twists and suspension}
\label{suspensionpar}
There are special types of hyper--functors called $\V$--twists, see \cite[\S4.2.1]{feynman} generalizing
cobordism twist of \cite{GKmodular}. Consider $Pic(\C)$, that is the full subcategory of  tensor invertible elements in $\C$ and an $\fL\in \V\dashopcat_{Pic(\C)}$.

Given a hyper--functor $\D$ one defines the $\fL$ twist as
\begin{equation}
\label{Ktwisteq}
\D_{\fL}(\phi)=\fL(t\phi)^{-1}\ot \D(\phi)\ot \fL(s(\phi))
\end{equation}
where again $s,t$ are the source and target maps and the composition uses the morphism $\fL(X)^{-1}\ot \fL(X)\simeq \unit$.
\begin{ex}{Suspensions}
One of the most important twist sis given by using suspension.
These are defined for graph based Feynman categories. In particular for the operadic and modular operad categories.

There are two interesting versions. The first is the naive suspension. $\Sigma$
which takes values $\Sigma(*_S)=\Sigma \unit$.

The second in the genus marked case is $\fs$ given by $\fs(*_{S,g})=\Sigma^{-2(g-1)-|S|}\sign_{\SS_|S|}$.
Without the genus marking $g=0$.
\end{ex}

There is a fundamental relation
\begin{equation}
\K \simeq {\it Det}_{\fs\Sigma}
\end{equation}
This states that if $\gh(\phi)$ is contractible, the $\K$ is simply the suspension $\fs\Sigma$. This is precisely the odd structure for the bar complex, see \S \ref{baralgebrasec}.
But, {\em if the underlying graph has topology, the mere suspensions do not suffice.}

This is one of the basic mantras explained in detail in \cite{KWZ}.

\subsubsection{Gcp version and hyper version of $\FF^+$}

For the different types of functors in Definition \ref{functortypedef} there are Feynman categories
through which these functors factor.

For this one adjoins morphisms and mods out by relations.
The paradigmatic example is $\FinSet$ which is generated by $\surj$ and the morphism
$i:\emptyset \to \{*\}$, which satisfies the relation $\pi_{S_+}\circ (id_S\amalg i)=\pi_S$, where $\pi_S:S\to \{*\}$ is
the unique surjection, $S_+=S\amalg \{*\}$ and $id_S\amalg i:S=S\amalg \emptyset\to S$.

\begin{df}
\label{gcpdef}
We define $\FF^{+\it gcp}=(\V^{+\it gcp},\F^{+\it gcp},\imath^{+\it gcp})$
as follows: $\V^{+\it gcp}=\V^+, \Iso(\F^{+\it gcp})=\Iso(\F^+)$ and $\imath^{+\it gcp}=\imath^+$.
But for $\F^{+ \it gcp}$,
we first freely adjoin  a morphism
$i_\sigma:\unit_{\F^+}=id_{\unit_\F}\to \sigma$ for all $\sigma \in \V^+$, and then
 mod out by the relations
 \begin{equation}
 \label{gcpreleq}
\sds (i_\tau)=i_{ \sigma'\circ\tau\circ \sigma}
 \end{equation}
 implementing the compatibility with the groupoid structure of $\V^+$.
The compatibility with the generating morphisms is forced by modding out by the relations postulating that the following diagrams commute
\begin{equation}
\label{rightcompateq}
\xymatrix{
\phi_0\otimes \sigma_1\odo \sigma_n\ar[r]^{\gamma}&
\phi_0\circ (\sigma_1\odo\sigma_n)\\
\phi_0\otimes \unit_{\F^+}\odo \unit_{\F^+}
\ar[u]_{id\otimes i_{\sigma_1}\odo i_{\sigma_n}}
\ar[r]_{\text{unit constraints}}^\simeq
&\phi_0\ar[u]_{(\sigma^{-1}_1\odo\sigma^{-1}_n, id)}^\simeq
}
\end{equation}
where the right morphisms is given by the groupoid data and

\begin{equation}
\label{leftcompateq}
\xymatrix{
\sigma_0\otimes \phi_1\ar[r]^{\gamma}
&\sigma_0\circ \phi_1\\
\unit_{\F^+} \otimes \phi_1\ar[u]^{i_{\sigma_0}\otimes id}
&
\phi_1\ar[l]_{\text{unit constraint}}^{\simeq}
\ar[u]_{(id\Downarrow \sigma_0)}^{\simeq}
}
\end{equation}

$\FF^{\it hyp}=(\V^{\it hyp},\F^{\it hyp},\imath^{\it hyp})$ is defined as follows: $\F^{\it hyp}$ is the quotient of $\F^{+ \it gcp}$ by the relation that $i_\sigma$ is invertible.
$\V^{\it hyp}$ is the full sub--groupoid of $\V^+$ whose objects are non--isomorphisms and $\imath^{\it hyp}$ is the restriction of $i^+$.
\end{df}
\begin{prop}
\label{gcpfactorprop}
Both $\FF^{+ \it gcp}$ and $\FF^{\it hyp}$ are Feynman categories. Moreover any gcp functor factors through $\FF^{+ \it gcp}$ and any hyper--functor factors through $\FF^{\it hyp}$.
\end{prop}

\begin{proof}
The last statement is straight-forward. The first statement for
for $\FF^{+ \it gcp}$ this is again straight--forward, as we are only adding element--type morphisms and the relations preserve decomposability.
For $\FF^{\it hyp}$ this follows from a two--step argument:
First, inverting the $i_\sigma$, we see that in $\F^{\it hyp}$ the full subcategory spanned by the elements $\sigma\in \Mor(\V)$ and $\unit$ is equivalent to a discrete category. There is a unique morphism between any two objects which is an isomorphism: $\Hom_{\F^{\it hyp}}(\sigma,\sigma')=\{i_{\sigma'}\circ i_\sigma\}$, $\Hom(\unit,\sigma)=\{i_\sigma\}$, $\Hom(\sigma,\unit)=\{i^{-1}_\sigma\}$ and $\Hom(\unit,\unit)=id_\unit$.  Contracting this isomorphism class to $\unit$, we obtain an equivalent category that is clearly a Feynman category with vertices $\V^{\it hyp}$.

\end{proof}

\begin{rmk}
\label{irmk}
From \eqref{leftcompateq} and \eqref{rightcompateq} it follows that $\gamma(i_\sigma\otimes i_{\sigma'})=i_{\sigma\circ\sigma'}$.
Also, the analog of Remark \ref{unitalrmk} applies using $i_\sigma$ in lieu of the $u_\sigma$.
\end{rmk}

\subsection{Computations for the plus construction: realizing the first ladder}
In this section, we will apply the explicit graphical presentation for the plus construction of Appendix \ref{graphpluspar} and starting  from $\FFtriv$ construct the categories corresponding to various types of algebras in a first step and in a second step the Feynman categories for various forms of operads.

\subsubsection{From objects to monoids, via the plus construction}

Using Theorem \ref{forestthm}, we can compute $(\FFtriv)^+$ and $(\FFtriv)^\gcp$ to obtain the first rung of the fundamental ladder.

In $\Crl_{(\FFtriv)^+}$, see \S\ref{crlFdefpar}, the vertex color is necessarily $id_*$ as it is the only possible morphism and all the flag colors are $*$. The decorations $(\bs,\be)$ for morphisms are necessarily given by identities. There is only one objects in $\Crl_{(\FFtriv)^+}$ whose isomorphism group is trivial. Hence the objects of
$(\Crl_{(\FFtriv)^+})^\otimes$ are the $*_{id_n},$ with $id_n=:id_{*^{\otimes n}}=(id_*)^{\otimes n}$ whose automorphisms are   $\SS_n$,.
Identifying $n$ with $\id_*^{\otimes n}$
  automorphisms  $Iso(\Agg_{(\FFtriv)^+})\simeq (\Crl_{(\FFtriv)^+})^\otimes=\SS$

A general morphism in  $\Agg_{(\FFtriv)^+}$, see \S\ref{aggFdefpar}, is given by decorated forests whose trees are linear, i.e.\ they all have bi--valent vertices, since $id_*$ is the only possible decoration.
 The map $g_V$ is a surjection $\un\ta\um$. Let $|g_V^{-1}(i)|=n_i$.
  Writing a 2--regular, aka.\ linear, tree with $n$ vertices as $\tn$, we see that the morphisms in $\Agg_{(\FFtriv)^+}$
have underlying trees $\dottree n_1\cdots \dottree n_m$. The vertices in each $\dottree n_i$ are ordered from the root to the top,
thus giving an order on the fibers $g_V^{-1}(i)$, as each vertex has a unique outgoing flag.
The extra data identifies the root, that is the only outgoing flag, of $\dottree n_i$ with the target $\dottree$ corresponding to $i$ in $\um$.

A basic morphism, is thus given by one linear tree with an ordering of its vertices $(\tn,<)$. Here $<$ can alternatively be thought of as an identification of the vertices of the tree with the factors of $id$ in $(id_*)^{\otimes n}$, that is a bijection $\un\leftrightarrow\un$ aka.\ an order on $\un$.
The isomorphisms are given by permuting the (source) vertices. By pre--composing, they act as permutations on the vertices of the linear tree and hence transitively on the linear orders on the fibers.

Considering $(\FFtriv)^\gcp$, there is one added  morphism $i_{id_*}:\emptyset\to \dottree $.
It is convenient to introduce the notation $\dottree u$ the ghost-graph of the morphism $i_{id_*}$.
There are the relations
$(\pi_{2},1<2)\circ (id,1)\amalg i=(\pi_{2},1<2)\circ i \amalg(id,1)=
(\pi_2, 1<2)$. Note that $(\pi_2,2<1)=(\pi_2,1<2)\circ \tau_{12}$ and thus we also have the unit equation regardless of the position and order of the unit insertion.

An example is given in Figure \ref{linearfig}.

 \begin{figure}
   \includegraphics[width=.8\textwidth]{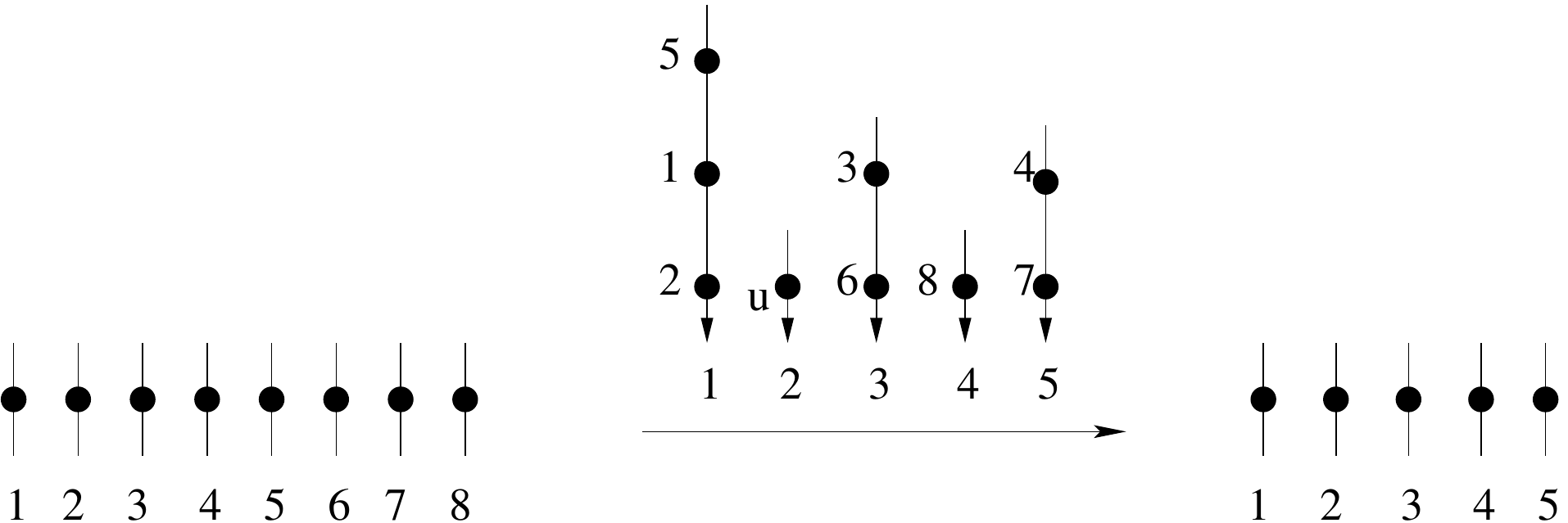}
   \caption{\label{linearfig} A morphism given by a marked linear forest. This yields the same morphism in $(\FFtriv)^\gcp=\FNCSet$ given in Figure \protect{\ref{plcorfig}}. The underlying forest is
   $\dottree 3 \, \dottree 0\, \dottree 2\, \dottree 1\, \dottree 2$, which fixes the surjection $\underline{8}\ta\underline{5}$.}
 \end{figure}

\begin{prop}
\label{algebraprop}
As Feynman categories
$(\FFtriv)^+\simeq \Surj^>$ and $(\FFtriv)^\gcp\simeq\FNCSet$.
\end{prop}

\begin{proof} The morphism $\ff$ of Feynman categories  consists of the functor $v$ given by $\dottree=\dottree id_*\to 1$ and the functor $f$
 given by $\dottree^{\otimes n}\to \un$ and  on basic morphisms $f(\tn,<)=(\pi_n,<), n\geq 1$ where the order of $\tn$ is given by the height of the vertex $i$ and the order on $\un$ is the order induced by this. In the case of $(\FFtriv)^\gcp$ we also have $f(\dottree\; 0=i_{id_*})=i:\emptyset \to \underline{1}$, as a generator with the appropriate relations.
\end{proof}

\begin{rmk}
This gives another natural interpretation of $\NCSet$ and $\Delta S$ complementary to \cite{Loday,Lodaycrossed,NCSetPira,NCSet}.
\end{rmk}

\subsubsection{Operads and non--Sigma operads via the plus construction: The second rung}
One can now apply the plus construction again that is compute $((\FFtriv)^+)^+=\FF\fS_>^+\simeq \operads^{\neg \Sigma}$. We can also obtain operads through the plus construction.
Let $\operads$ be the Feynman category for operads, $\operads^{\neg\Sigma}$ that for non--Sigma,  operads and let $\operads_{\it unital}$ and $\operads^{\neg\Sigma}_{\it unital}$ the ones for unital operads and non--Sigma operads. The unital operads are sometimes called May operads and the corresponding Feynman category was called $\FF_{May}$ in \cite{feynman}.
 The latter have an adjoined unit morphism $u:\empty\to \O(1)$ which satisfies the unit equations, see \cite[\S2.2]{feynman}, see also Table \ref{zootable}. We will also consider $\operads_0$ and $\operads^{\neg\Sigma}_0$ whose $\ops$ whose $\O(1)\simeq \unit$.

\begin{prop}
\label{operadprop}
We have the following identifications:

\begin{tabular}{lll}
$\FFinSet^+\simeq \operads$& $\FFinSet^\gcp\simeq \operads_{\it unital}$& $\FFinSet^\hyp=\operads_0$\\
$\FNCSet^+\simeq \operads^{\neg \Sigma}$&
$\FNCSet^\gcp \simeq \operads^{\neg \Sigma}_{\it unital}$&$\FNCSet^\hyp=\operads^{\neg-\Sigma}_0$
\end{tabular}

Restricting to surjections, one obtains the Feynman categories for the same representations, but without constants, i.e.\ $\O(0)$.
In particular, $\Surj^\hyp=\operads^{\it red}$ that is the Feynman category  whose $\ops$ are reduced operads, no $\O(0)$ and $\O(1)\simeq \unit$.
\end{prop}
\begin{proof}
Consider $\FFinSet^+$. There is only one flag color $\uone$. Up to isomorphism, the basic morphisms are
given by $\pi_n:\un\to \uone$, where $\pi_0=i:\emptyset\to \uone$. These give a smaller, but equivalent subcategory, known as the biased version.
These give rise to the corollas $*_{\pi_n}$, which are the isomorphism classes of objects in $\Crl_{\FFinSet^+}$. The automorphisms in $\Crl_{\FFinSet^+}$ of $*_{\pi_n}$ are $\SS_n$ as $\pi_n\circ\sigma=\pi_n$ and $\sigma$ acts trivially on the flag labeling up to equivalence, since \eqref{flagequiveq} holds for any permutation. Thus, replacing $\Crl_{\FFinSet^+}$ by a skeleton, we have $\Crl_{\FFinSet}^+\simeq \SS$.

The objects in $\Agg_{\FinSet^+}$ are then forests of $*_{\pi_n}$ or  more generally $*_{\pi_S}$, where $\pi_S:S\ta \uone$.
The only color isomorphism is $id_{\uone}$, thus the data $(\bs,\be)$ is trivial.
Thus, up to isomorphism,
the  basic morphisms in $\Agg_{\FFinSet}$ are exactly the graph morphisms $\phi:*_{\un_1} \amalg\dots \amalg *_{\un_m}\to *_{\un}$ whose underlying ghost graph is a tree. This implies that  $n=\sum_i n_i$  The arity of vertex decorated by $\pi_n$ is $n$ and hence knowing the arity fixes this decoration. The flag colors, the edge and tail decorations are all redundant as well. This means that we have obtained the skeletal version of $\operads$.

More generally allowing for the morphisms $\pi_S:S\ta \{r\}$, the vertices $*_{\pi_S}$ will be or arity $|S|$ and have incoming flag set in bijection with
$S$ and an outgoing flag corresponds to $\{r\}$. Hence we have $*_{\pi_S}$ is completely determined by the rooted corolla $*_{S\amalg \{r\},t}$.
And we obtain $\FFinSet^+=\operads$.

In $\FFinSet^\gcp$ there is an extra morphism $u:=i_{id_{\uone}}:\emptyset\to id_{\uone}$ which provides a unit-element morphism, \cite[2.2]{feynman}.
It is straightforward to check that for $\O\in \FFinSet^\gcp\dashops$, $\O(u)$ is a unit.

In the case of $\FNCSet$, consider
the objects of $\Crl_{\FNCSet^+}$. In the biased version, that is using only the sets $\un$, the objects are $*_{\pi_n,<_n}$, where $<_n$ is an order on $n$.
The permutations act transitively on the orders of the decorated corollas: $(p\Downarrow id)(*_{\pi_n,<_n})=*_{\pi_n,p(\circ<_n)}$, were $p(\circ<_n)$ is the permutation of \S\ref{crlFdefpar}. Thus there are no automorphisms and a skeletal version of $\Crl_{\FNCSet^+}$ is the discrete category $\N_0$ whose objects are the natural numbers with only identity morphisms.
This means that up to isomorphism, there is a unique ordered vertex of arity $n$ for each $n$ which can be represented by the planar, planed corolla $*_{\un\amalg \{r\},r}$.
Proceeding to $\Agg_{\FNCSet^+}$, we see that the planar, planted forest underlying the morphism uniquely determines the morphism up to isomorphism.
Thus, we have objects given by planar, planted corollas and basic morphisms given by graph morphisms whose ghost graph is a tree.
This tree becomes planar if one pulls back the orders from the orders on the source using the identification of the vertices of the graph and the source.
Finally, compatibility states that the  order of the leaves of the target corolla is that of the leaves of the tree.
This is the description of $\O^{\neg\Sigma}=\O_{\dec \Ass}$, see \cite{decorated} and Appendix \ref{connectionspar} below.
\end{proof}

\begin{lem}
\label{nscoverlem}
 Forgetting the order yield a forgetful functor $\mathfrak{p}:\operads^{\neg \Sigma}\to \operads$.
This morphism is a cover with decoration $p_!(\trivial)=\Assoc$.

Furthermore $\Assoc((*_{S\amalg \{r\}},r))=\{\text{all orders $<_S$ on $S$}\}$, with the usual composition:
$\gamma((*_{S \amalg \{r_0\}.,r_0,<_S};(*_{T_1\amalg \{r_1\}},r_n,<_n) \kdk (*_{T_1\amalg \{r_1\}},r_1,<_1))=*_{T\amalg \{r_0\},r_0,<_T}$ where $T=T_1\amalg \dots \amalg T_n$ with the order $<_T=<_1\amalg \dots \amalg <_n$.

This cover restricts to the non--unital and reduced $\O(1)$ and reduced cases.
\end{lem}

\begin{proof}
This is contained in \cite{decorated}. The key steps are: to define $\pp=(p_\V,p)$,  functor $p_\V$ is given simply by $(*_{S\amalg_\{r\}},r,<_S)\mapsto (*_{S\amalg \{r\}},r)$ on objects. On morphisms of $\operads^{\neg\Sigma}$, $p$ is given by forgetting the planar structure on the forests. This is a cover, as it is surjective on objects and morphisms and every morphism has a unique lift once the target is fixed. This latter is the case as the planar structure on the forest is completely fixed by the planar structure of the source corollas.
 On objects  computing  $p_!(*_{S\amalg \{r\}},r)$ yields the fiber that is the set $\{(*_{S\amalg\{r\}},r,<_S)\}$ or simply the set of orders on $S$. Composition of the linear orders is that according to the forests. I.e.\ lifting the planar structure of the source corollas yields a planar structure on each tree underlying a morphism and this in turn yields the given order of the leaves, which are the set $T$.
\end{proof}
\section{Modules and Enriched Feynman categories}
\label{enrichedpar}
\subsection{Enriched categories}
\label{enrichedcatpar}
To consider modules of algebras,  or more  generally if $\O\in\Fops_\C$ their modules will need Feynman  categories enriched in $\C$.
The general reference is \cite{kellybook}, to which we refer for full details. We will give the salient features here.

Recall that a category enriched $\F$ in a symmetric monoidal category $(\E,\otimes)$ has a class of objects $\Obj(\F)$ and morphisms  $\Hom(X,Y)\in \Obj(\E)$
together with a associative unital composition maps which are morphisms in $\E$: $\circ:\Hom(Y,Z)\ot_\E \Hom(X,Y)\to \Hom(X,Z)$ and units $id_X:\unit_\E\to \Hom(X,X)$.
Likewise for a monoidal structure, all structure morphisms are morphisms and objects $\E$, in particular:
\begin{equation}
\ot:\inthom(X,Y)\ot_\E \inthom(Z,W)\to \inthom(X\ot Z,Y\ot W)
\end{equation}
is a morphism in $\E$.
\begin{ex}
Consider the category $\underline{G}_k$ enriched over $\kVect$ which has one object $*$ and morphism set $k[G]$ with composition on basis elements
given by $\circ(g,h)=g\circ h$.
\end{ex}
\begin{ex}[Internal Hom]
A monoidal category can be enriched over itself. The standard example is $\kVect$, since $\inthom(V,W)$ again has the structure of a vector space.
More generally, $\E$ can be enriched over itself if it is closed. This means that it an internal hom,
$\inthom(X,Y)\in E$, and $\inthom(X\otimes Y,Z)\simeq \inthom(X, \inthom(Y,Z))$ functorially,  see \cite[\S1.6]{kellybook} for details.
\end{ex}

One can then consider functors between two categories $\C$ and $\D$ enriched over the same $\E$ in  straightforward formalism.
In particular,  \\
$\Hom_\C(X,Y)\to \Hom_\D(\O(X),\O(Y)):\phi\mapsto \O(\phi)$ is a morphism in $\E$.

\begin{ex}\label{kgmodex}
Functors $\O$ from $\underline{G}_k$ to $\kVect$ thought of as enriched over $\kVect$ are $k[G]$-modules. Let $M=\O(*)$  the fact that the map $\Hom(*,*)=k[G]\to \inthom(M,M)$ is a morphism in $\kVect$ means that the action is $k$-linear, i.e.\ the action is given by $\mu:k[G]\otimes_k M\to M$.
\end{ex}

\subsubsection{Freely enriched categories}
Given an enriched category one can define an underlying category by defining the underlying morphism via  $Hom_{\Set}(X,Y):=$ $\Hom_\E(\unit_\E,\inthom(X,Y))$. This is actually a 2--functor, see \cite{kellybook}, which has an adjoint, called free enrichment. We will use the notation $F_\E$ for the free enriched version of $\E$.
E.g.\ if $\E=\kVect$ then $Hom_{\F_\E}(X,Y)$ is the free vector space on $Hom_\F(X,Y)$. If $\E=\Top$ the   $Hom_{\F_\E}(X,Y)$ is  $Hom_\F(X,Y)$ with the discrete topology.

\subsubsection{Cartesian vs. Linear Enriched}
There are basically two types, Cartesian enriched and linearly enriched.
Cartesian enriched means that $\ot_\E$ is also a Cartesian product like in $\Top$\footnote{Fixing a convenient topological category.}. Linear means that one is at lead $\Ab$ enriched, and $\ot_\E$ is ``bi--linear''. Typical examples for $\E$ are $\kVect, dg\mdash\kVect$, etc..

There are basically no big modifications to Feynman categories in the Cartesian enriched case.
In the linear case, there are necessary modifications as the notion of a groupoid becomes unavailable. Note that $GL(V)\subset End(V)$ is a subspace, but not a linear subspace.

\subsection{Modifications in case of  enrichment}
In this sub-para\-graph, we will collect the modifications that are necessitated in the  enriched case, especially  in that of linear enrichment.
\subsubsection{Cartesian enriched Feynman Categories}
Generally in the enriched case axiom (ii) is be replaced by the rather technical axiom (ii').
\begin{itemize}
\item[(ii')] The pull-back of preserves $\imath^{\otimes \wedge}\colon [\F^{op},Set]\to [\V^{\otimes op},Set]$
{\em restricted to representable pre--sheaves} is monoidal.
\end{itemize}
The monoidal structure on pre--sheaves is given  by Day convolution $\day$, thus (ii') means that
\begin{multline}
\label{dayeq}
\imath^{\otimes \wedge}Hom_{\F}(\,\cdot\, , X\otimes Y)=
Hom_{\F}(\imath^{\otimes}\, \cdot\, ,X\otimes Y)=\\
\imath^{\otimes\wedge}Hom_{\F}( \,\cdot\,, X)\day \imath^{\otimes\wedge}  Hom_{\F}(\,\cdot\, , Y)
\end{multline}

Using the definition of the Day convolution the right hand side of (\ref{dayeq})  becomes the co--end condition:
\begin{multline}
\imath^{\otimes\wedge}Hom_{\F}( \,\cdot\,, X)\day\imath^{\otimes\wedge}
Hom_{\F}( \,\cdot\, , Y)
 =Hom_{\F}(\imath^{\otimes} \,\cdot\,, X)\day Hom_{\F}(\imath^{\otimes} \,\cdot\, , Y)\\
=\int^{Z,Z'}Hom_{\F}(\imath^{\otimes} Z, X)\times
Hom_{\F}(\imath^{\otimes} Z' , Y)\times Hom_{\V^{\otimes}}( \,\cdot\,,Z\otimes Z')
\end{multline}

The co--end formula expresses the ``bi--linearity'' of composition \cite{Auslander,MacLane}.

Just like condition \eqref{morphcond}, the smallness condition \eqref{smallcond} should be modified in the enriched case as (co)limits become so--called indexed (co)--limits, see \cite{kellybook}.
\begin{itemize}
\item[(iii')] For all $*\in \V$, the  indexing functors $\tilde \imath^{\otimes}(*):=Hom_{\F}(\imath^{\otimes} *, -)$ are essentially small.
\end{itemize}
The indexing functor takes care of the ``linearity'' of morphisms.

\begin{df} A Feynman category $\FF$ enriched
over a Cartesian $\mathcal E$ is a triple $(\asts,\clusters,\imath)$ of a category $\clusters$ enriched
over $\mathcal E$ and an enriched category $\asts$ which satisfy the enriched
version of the axioms of Definition \ref{feynmandef}. That is (i), (ii') and (iii') as given above
\end{df}

\subsubsection{Linear enrichment/Weak Feynman categories/Index enriched Feynman categories}
As there is no good notion of groupoid, in the linear case, the axiom (i) has to be modified to (i').
\begin{df}
\label{weakdef}
A weak Feynman category is a triple $({\mathcal W},\F,\imath)$, where both ${\mathcal W}$ and $\F$ are categories enriched over $\CalE$,
$\imath: \mathcal{W}\to \F$ is a functor enriched over $\E$, $\F$ is symmetric monoidal enriched over $\E$, and $\mathcal W$ symmetric monoidal tensored over $\CalE$ satisfying: (i') $\imath^{\otimes}$ is essentially surjective, and (ii') and (iii') as above.
\end{df}
This notion is closely related to Getzler's patterns, see \cite{Getzler,feynman,BKW}.

\begin{df} An indexed enriched Feynman category is a weak Feynman category $\FF=(\V,\F,\imath)$ indexed over a $\Set$ Feynman category
$\fB=(\V_\B,\B,\imath_\B)$, such that $\V=(\V_\B)_\E$.
\end{df}

\begin{as}
From now an, we assume that $\C=\E$ is enriched over itself and
has a co--product. For representations, one is interested in  Abelian categories $\C$, which is why we denote the co--product by $\oplus$.
\end{as}

\subsection{Enrichment functors}
Enrichment functors are the generalization of Remark \ref{enrichfunctorrmk}. We refer to Appendix \ref{twocatapp} for the two--categorical notions.

\begin{df} A weak enrichment functor for a Feynman category $\FF$ is a lax 2--functor $\F\to \underline{\E}$
with is strictly monoidal, see Remark \ref{enrichfunctorrmk}.
An enrichment functor is a weak enrichment functor that also satisfies $\D(\sigma)=\unit$.
\end{df}

\subsubsection{Indexed enriched Feynman categories over $\FF$ and $\FF^\hyp$}
The following is proved in \cite[Proposition 4.1.2, Theorem 1.4.1]{feynman} connecting the plus construction and enrichment.
\begin{thm} \label{hypthm}\mbox{}
\begin{enumerate}
\item \label{functorpart} There is a 1--1 correspondence  between indexed enriched Feynman categories over $\FF$ and enrichment functors.
\item \label{catpart} There is a 1--1 correspondence between enrichment functors and $\FF^{\hyp}\dashops_\E$.
\end{enumerate}
Denote the indexed enriched Feynman category of $\FF$ corresponding to $\D\in \FF^+\dashops_\E$ by $\FF_\D$, then
monadicity holds for the weak Feynman category  $\FF_\D$.
\end{thm}

The correspondence is represented by the formula  \eqref{enrichedhomeq} which is a generalization of \eqref{fiberhomeq}.

\begin{equation}
\label{enrichedhomeq}
\inthom_{\F_\D}(X,Y)=\bigoplus_{\phi\in \Hom_\F(X,Y)}\D(\phi)
\end{equation}
with composition
\begin{equation}
\D(\phi)\otimes\D(\psi)\to \D(\phi\circ\psi)
\end{equation}

\begin{rmk}
\label{proofrmk}

The first two parts the statement for $\E=\Set$ is contained in Remarks \ref{decomprmk} and \ref{enrichfunctorrmk}.
The general proof of these statements is similar. also follows from the definitions, in particular Definition \ref{enrichfunctdf} and Corollary  \ref{fplusupscor}. Note that the condition that
 composition with isomorphism is strict, i.e.\ it is given by unit constraints from \cite{feynman} follows from  Definition \ref{enrichfunctdf}  condition \eqref{idXcond}  together with the groupoid action
 via Remark \ref{irmk}. This is why we could remove this extra condition in the definition of an enrichment functor.
\end{rmk}

\subsubsection{Generalizing to $\FF^\gcp$}
The results and constructions are analogous to Example \ref{Assex}.

Relaxing the condition of an enrichment functor to a weak enrichment functor, we obtain the generalization of
Theorem \ref{hypthm} \eqref{functorpart}.
\begin{prop}\mbox{}
\label{functorequivprop}
 Weak enrichment functors are in 1--1 correspondence with $\FF^\gcp\dashops_\E$.
\end{prop}
\begin{proof}
This is straightforward as in Remark \ref{proofrmk} using Definitions \ref{functortypedef},  \ref{enrichfunctdf} and Proposition \ref{twofunctorprop}.
\end{proof}

There is a generalization of the results of Theorem \ref{hypthm} \eqref{catpart} to  $\FF^\gcp\dashops$.
 In the enriched case, we have to be careful about splitting.

\begin{df}
\label{splitdef}
We call $\D\in \F^{+}\dashops_\C$ split, if each for all $\sigma\in \Mor(\V)$.
\begin{enumerate}
\item  $\D(\sigma)=\D(\sigma)^\times \oplus \D^{\it red} (\sigma)$.
\item Any invertible $\phi\in \D(\sigma)$ with $\phi^{-1}\in \D(\sigma^{-1})$ is in $D(\sigma)^{\times}$.
\item \label{gpdcond} $\bigoplus_{\sigma\in \Mor(\V)}\D(\sigma)=\Gpd(\D)_\E$ that is the free enrichment of a set--groupoid $\Gpd(\D)$.

\end{enumerate}
A $\D\in \F^{gcp +}\dashops_\C$ is called split if it satisfies the conditions above and furthermore
The morphism $\D(i_\sigma):\unit\to \D(\sigma)$ is split. I.e. $\D(\sigma)^\times=\unit \oplus \overline{\D}(\sigma)^\times$, where
the first summand is $im(\D(i_\sigma))$.

A weak enrichment functor is split, if the corresponding functor $\D\in \FF^\gcp\dashops$ is split.
\end{df}

\begin{rmk} \mbox{}
\begin{enumerate}
\renewcommand{\theenumi}{\roman{enumi}}
\item  Condition (3) means
 that  $\Obj(\Gpd(\D))=Obj(\V)$, but $Mor(\Gpd(\D))=\amalg_{\sigma\in \Mor(\V)}\amalg_{\sigma_i\in I_\sigma}\sigma_i$
 and $\D(\sigma)^\times=\bigoplus_{\sigma_i\in I_\sigma}\unit$ and there is are composition morphisms
 $I_\sigma\times I_{\sigma'}\to I_{\sigma\sigma'}$ for composable $\sigma,\sigma'$, such that fixing an element in $I_\sigma$ or an element in $I_{\sigma'}$ the composition morphism is a bijection.
\item In the $\gcp$ case, all the
$\D(\sigma)^\times=\unit\oplus \bar\D(\sigma)^{\times}$ split with the
first component being $\D(i_\sigma)$. In the language above $I_\sigma$ is a pointed set $(I_\sigma,0)$ and there is an inclusion $\V\to \Gpd(\D)$, with the image of $\sigma$ being $0\in I_\sigma$. In particular, there is an involution,
$\bar{}:I_\sigma\to I_{\sigma^{-1}}$ for which for $\sigma:X\to Y$ the composition
$\unit_{\sigma_i^{-1}}\otimes \unit_{\sigma_i}\to \unit_{id_X}$ where $id_X$ is the base point of $I_{id_X}$ .
\item Any $\D\in \F^{\it hyp}\dashops_\C$ is split.
\end{enumerate}
\end{rmk}

\begin{as}
From now on, we will assume that all functors from $\F^+$ and $\F^{+gcp}$ are split.
\end{as}

In the case $\E=\Set$  split is simply given by $\D(\sigma)=\D(\sigma)^\times\amalg \D(\sigma)^{red}$

Given a split $\D\in \F^\gcp$, set $\V_\B:=\Gpd(\D)$ and let $\F_B$ be the trivially extended
 monoidal category along the projection $j:\V_\B\to \V$ given  by
$I_\sigma\mapsto \sigma$.
This means the $\Obj(\F_\B)=\Obj(\F)$. To give the morphisms, note that $\Mor(\F)$ is a $\V^\otimes\mdash\V^\otimes$ bi--module with the action $\sds$.
This action is extended to bi--module action of $\V_\B$ by $(\sigma_i\Downarrow\sigma'_j)(\phi)=\sds(\phi)$.

There is the natural  inclusion $i_\B:\V_\B\to \F_\B$. Set $\fB(\D)=(\V_\B,\F_\B,\imath_\B)$.

\begin{thm}
\label{weakthm}
For a split weak  enrichment functor $\D:\FF^+$ there is weak Feynman category $\FF_\D$ indexed over the Feynman category
$\fB(\D)$.
\end{thm}
\begin{proof}
The fact that $\FF_\D$ is indexed over $\FF_\D$ is clear. The fact that $\fB$ is a Feynman category follows in a straightforward fashion similar to \cite[Theorem 4.1.4]{feynman}
\end{proof}

\begin{cor}
A split weak enrichment functor $\D$ for $\FF$ lifts to an enrichment functor $\tilde \D$ over $\fB(\D)$.
\end{cor}

The values of $\tilde{\D}$ on morphisms are $\tilde{\D}(\sigma_i,\phi)=\unit_{\sigma_i}\otimes \D(\phi)\simeq \D(\phi)$ where $\sigma_i$ is the isomorphism corresponding to $i\in I_\sigma$, $\unit_{\sigma_i}$ is the corresponding component of $\Gpd(\D)$ and the morphism is given by pre--composing the morphism  $\D(\sigma) \ot \D(\phi)\to \D(\phi)$ with the inclusion of $\unit_{\sigma_i}\to \D(\sigma)$.

\subsubsection{$\FNCSet$ as an indexed enriched Feynman category over $\FFinSet$}
We know that $\FFinSet^\gcp=\operads$, so enrichment functors will be operads.
Let $\Assoc\in \operads\dashopcat_\Set$ be the associative operad as in Lemma \ref{nscoverlem},
then $\Assoc\in \operads^\hyp$, since $\Assoc(*_{\{s,t\},t})$ has only one element.
The following is not straightforward
\begin{lem}\label{assocenrichedlem}
$\FNCSet=\FFinSet_{\Assoc}$ is indexed enriched over $\FFinSet$. \qed
\end{lem}

\begin{rmk}
\label{assdecoenrmk}
We now have two description of $\operads^{\neg\Sigma}$. Using the Lemma  \ref{assocenrichedlem} above, Proposition \ref{operadprop} and Lemma \ref{nscoverlem}

\begin{equation}
(\FFinSet_{\Assoc})^\gcp=(\FinSet^\gcp)_{\dec \Assoc}
\end{equation}
This is part of a general statement, see \S\ref{plusdecpar} and \cite{pluspaper}.
\end{rmk}

\subsubsection{Enriching quivers}
\label{enrichedquiversec}
Quivers give a generalization of Example \ref{kgmodex}.
As an example consider of a simple quiver $Q:\bullet_1\to \bullet_2$. That is the category has two isomorphisms $\id_{\bullet_1}$ and $id_{\bullet_2}$ together with a morphism $\phi:\bullet_1\to \bullet_2$. This give rise to a Feynman category $\FF_Q$, where $\V$ is the discrete category with objects $\bullet_1,\bullet_2$ and $\F=Q^\otimes$.

\begin{prop} The weak enrichments of $\FF_Q$ are in 1--1 correspondence with $(A,B,\leftsub{A}{M}_B)$ of
 two unital split, in the sense of split enrichment functors, algebras and a bi--module.
\end{prop}

\begin{proof}
We can define a class of
 enrichment functors by giving two groups and a bi--module over them
  $\D$ will have values $\D(id_{\bullet_1})=A$, $\D(id_{\bullet_2})=B$ and $\D(\phi)=\leftsub{A}{M_B}$.
  The algebra structure comes from the compositions $id_{\bullet_i}\circ id_{\bullet_i}=id_{\bullet_i}$
 The bi--module structure comes from the compositions $\phi\circ id_{\bullet_1}=id_{\bullet_2}\circ \phi=\phi$.

\end{proof}

If the functors are weak or not depends on the algebras $A$ and $B$.

\begin{ex}
  A particular example over $\Set$ is the choice $A=GL(V), B=GL(W), Hom(V,W)$, where $V,W\in \kVect$.
  Notice that this is not enriched over $k$ as discussed above. It is also a weak indexing.
  With the base category $\fB$ having two objects with automorphisms $GL(V)$ and $GL(W)$ respectively and one morphism between the two objects, i.e.\ the trivial $A\mdash B$ module.
\end{ex}
\begin{rmk} Going over to finite graphs,
this type of example is tied to the quantum graph symmetries \cite{graphsym,KKWKsym}, where the enrichment now
takes values in $C^*$ algebras and has applications to material science, \cite{chapter}.

\end{rmk}

\subsubsection{Twists}
One reason one uses the categories $\FF_\D$ is to obtain the necessary sign twist for the bar and co--bar constructions.
In particular, the twists of \S\S\ref{Kpar} and \ref{suspensionpar} are important for the transforms in \S\ref{barpar}.

There is also a nice interpretation of twisting the triples for indexed Feynman categories, which we will not describe in detail here,  but refer to \cite[Proposition 4.1.7]{feynman}. We wish to note, that the $\V$--twists modify the triples in an isomorphic way, hence one obtains isomorphisms between $\F_\D\dashopcat$ and $\F_{\D_\fL}\dashopcat$, which is what one is used to in the algebra case, see \S\ref{baralgebrasec} and in general \cite{KWZ} for relevant examples.

\subsection{Modules} With this preparation, we can finally define modules for $\FF^\gcp\dashops$. This generalizes the definition of \cite{feynman}, where the modules were only defined for $\FF^\hyp\dashops$.

\begin{df} Given a $\D\in \FF^\gcp\dashops_E$, $\D$-modules in a monoidal category $\C$ enriched over $\D$ are $\FF_\D\dashops_\C$.
\end{df}

\subsubsection{Modules over an associative algebra}
\label{Assenrichedex}
We can now do the construction of Remark \ref{Assex}  full justice, that is we can consider modules over an algebra and not just a monoid.
The example also exhibits all the features above.

Let $A$ be an associative  algebra over $k$, that is the value of $\A\in \Surj^>\dashops_{\kVect}=(\FFtriv)^+$.
We consider the category $\F_\A$ whose objects are the natural numbers and whose morphisms are given by $\inthom_{\F_\A}(n,n)=A^{\otimes n}$ and $\inthom_{\F_\A}(n,m)=0$ for $n\neq m$ with permutation action and unit $\unit=k=A^{\otimes 0}$.
This is the free symmetric category on the category enriched over $k$, $\V_A$, which has one object $*$ and morphisms $Hom_{\V_\A}=A$.
Thus ${\it Fun}_\otimes(\F_\A,\E)={\it Fun}(\V_\A,\E)=A\mdash mod_\E$ that is the category of $A$--modules in $\E$. Here we assume that $\E$ is also enriched over $\kVect$ and the monoidal functors are over $\kVect$.

Now, if $\A$ is split as a functor in $(\FFtriv)^+\dashops$, then $A=A^\times\oplus A$ and any invertible element lies in $A^\times =G_k=k[G]$ for some group $G$.

If $A$ is unital that is $\A\in NCS\dashops_{\kVect}=(\FFtriv)^\gcp$, being pointed by the unit and
 $A^\times=\unit\oplus \bar{A}^\times$, which is the inclusion of $e\in k[G]$.
In this case, $\fB_\A=\mathfrak{V}_\A$. Again, hyper--means that $A$ is reduces, that is $\bar(A)^\times=0$ and $A^\times =\unit$.
In this case indeed $\fB=\FFtriv$ and $\F_\A$ is indexed enriched over $\FFtriv$.

\begin{rmk}
This suggests the study of group like elements of a Hopf algebra as a replacement for the condition of having an underlying set--groupoid.
This will be explored in the future.
\end{rmk}

\subsubsection{Enrichments of $FI$, $FI_G$ and $FI_d$}
\label{FIpar}
Consider a functor $\D\in FI^\gcp \mdash \ops_\E$.
There are two generating morphisms in $FI$: $id_*$ and $i:\emptyset\to 1$. Let $\D(id_*)=A=A^\times\oplus \bar A$ and $\D(i)=M$.
The morphism $i_{id_*}$ provides $\unit\to A^{\times}$.

\begin{prop}
The weak indexed enrichments of $FI$ are in 1--1 correspondence with pairs $(A,M)$ where $A$ is a unital algebra (monoids) and $M$ is an $A$--module.
\end{prop}

\begin{proof}
The composition $id_*\circ id_*=id_*$ provides the multiplication map $\mu:\D(id_*)\otimes \D(id_*)=A\ot A\to A=\D(id_*)$, the unit is provided by $\unit\to \D(id_*)=A$ from the $\gcp$ data,
while the composition $id_*\circ i=i$ provides the module map $\rho:\D(id_*)\ot \D(i)=A\otimes M\to M= \D(i)$.\end{proof}
\begin{rmk}\mbox{}
\begin{enumerate}
\item As $\Inj\subset\FinSet$ we have that $\Inj^\gcp\subset \FFinSet^\gcp=\operads_{\it unital}$ and hence $\Inj^\gcp\dashops$ are unital operads with only $\O(1)$ and $\O(0)$. It is well known that this pair is a pair of a unital algebra and a module over it.
\item Considering $\Inj^+\dashops$, one arrives at a non--unital algebra and a module over it.
\end{enumerate}
\end{rmk}

Two special cases over $\Set$ have been considered in \cite{FIG,FId}. Here we generalize these in two ways. There is a general solution
yielding the two construction as special cases and this can be performed in any enriched setting, that is also $k$--linearly.

Let us briefly review the two constructions. For the first the category $G$--maps is studied. Objects are finite
sets with a morphism between $R$ and $S$ given by a pair of maps: and injection $f:R\to S$ and a map $\rho:R\to G$. Composition is given by $(g,\sigma)\circ (f,\rho)=(g\circ f,\tau)$ with $\tau(x)=\sigma(f(x))\rho(x)$.

The category $FI_d$ is given as follows: again the objects are finite sets and morphisms given by pair $(f,m)$, where $f:R\to S$ is an injection, but now $m:S\setminus f(R)\to \{1,\dots,d\}$. For the composition $(f,m)\circ (g,n)=(f\circ g, p)$ with $R\stackrel{g}{\to}S\stackrel{f}{\to}T$, where $p$ is defined on
$T \setminus f\circ g(R)=f(T\setminus g(R))\amalg T\setminus g(S)$ as $m\circ f^{-1}\amalg n$.

The extra data is compatible with orders and defines the categories $OI_G$ and $OI_d$.
\begin{prop}
Restricting $A$ to be group $G$ and $M$ to be trivial, we recover the category $FI_G$ .
Restricting $A$ to be trivial and $M=\{1,\dots,d\}$ then we recover the category $FI_d$.

Similar results hold in the ordered cases $OI_G$ and $OI_d$.
\end{prop}
\begin{proof}
In general, decomposing $S$ as $\amalg_{s\in S}*$, according to the decomposition \eqref{feydecompeq} a morphism in $Hom_{FI_\D}(R,S)$ is given by a tensor product of pairs $(id_*,g),a\in A)$ and $(i,m\in M)$,
 with a factor of $id_*$ for each element in the image of $R$ and a factor of $i$ for each element not in the image.
If $A$ is $G$ and $M$ is trivial, then an injection $R\to S$ is given by a tensor product of pairs $(id_*,g),g\in G)$ and $(i,1)$.
The data of the $(id_*,g)$ is equivalent to the data of a map $R\to g$ which assigns $g$ to the factor of $id_*$ corresponding to $r$.
Composition gives $(id_*,g)\circ (id_*,h)=(id_*,gh)$, where $(id<h)$ is the factor that corresponds to $f(x)$ if $(id_*,x)$ corresponds to $x$.
 and $(id_*,g)\circ (i,1)=(i,1)$, since the action is trivial $\rho(g)1=1$. Thus the two categories coincide.

If $A$ is trivial and $M=\{1,\dots, d\}$ then a morphism $R\to S$ is given by tensor factors of $(id_*,1)$ and $(i,j)$ with $j\in \{1,\dots,d\}$.
This data is encoded in the injection $f:R\to S$ and a morphism $S\setminus f(R)$. Composition now reads $(id_*,1)\circ (i,j)=(i,j)$ which is the first part of the formula for $p$. The second part is simply the insertion of factors of $i$ for the elements missed by $g$.

The proofs for the ordered cases is analogous.
\end{proof}
Note that the case of $A=G$ is one in which $A^\red=\emptyset$. An in this case the indexing,  $A^\times=G$ is not reduced. The (weak) Feynman category is actually indexed over itself.

\subsection{Decoration}
In the case of decoration in the linear case, we need the modification that is spelled out in \cite[\S 2.2]{decorated}.

\begin{df} For a fixed choice of   $\jmath:\Iso(\F) \to V^{\otimes}$ realizing the equivalence of condition \eqref{objectcond}:

The objects of $\Fdeco$ are  tuples $(X, a_{v_1},\dots,a_{v_{|X|}})$, where
for $\jmath(X)=\bigotimes_{v\in I}\ast_v=\ast_{v_1}\odo \ast_{v_{|X|}}$, $a_{v_i}\in \O(\ast_{v_i})$.

The morphism of $\Fdeco$ are given by the subset
$$Hom_{\Fdeco}((X, a_{w_1},\dots,a_{w_{|X|}}),(Y,b_{v_1},\dots,b_{v_{|Y|}}))\subset Hom_{\F}(X,Y)$$
of those morphisms $\phi:X\to Y$, such that if $\bigotimes_v \phi_v$ is the decomposition of $\phi$ according to
the diagram \eqref{feydecompeq}, with $\phi_v:X_v=\bigotimes_{w\in I_v}\imath(*_w)\to \imath(\ast_v)$,  then $\O(\phi_v)(\bigotimes_{w\in I_v}(a_w))=b_v$.

The  monoidal structure is given by $$(X, a_{w_1},\dots,a_{w_{|X|}})\otimes (Y,b_{v_1},\dots,b_{v_{|Y|}}))=
(X\otimes Y, a_{w_1},\dots,a_{w_{|X|}},b_{v_1},\dots,b_{v_{|Y|}})$$ and the commutativity constraints are given by those of $\F$ on the first component and the respective permutations on the others.
\end{df}
The result can also be made into enriched category,  if $\F$ is tensored over $\E$.

\begin{df}Assuming the conditions above,
the category $\Fdeco$ as an enriched category over $\CalE$ has objects $X\otimes \O(X)$ and formally the same set of morphisms $\F$--$Hom_{\Fdeco}(X\otimes \O(X),Y\otimes \O(Y))=Hom_{\F}(X,Y)$. A morphisms $\phi$ via tensoring  becomes the morphism $\phi\otimes \O(\phi)$ in  $\C$.
Its symmetric monoidal structure is given by $(X\otimes \O(X))\otimes_{\Fdeco} (Y\otimes \O(Y))=(X\otimes_\F Y)\otimes \O(X\otimes Y)$, with composition of morphisms and symmetries given by the isomorphism
$(X\otimes Y)\otimes \O(X\otimes Y)\simeq (X\otimes \O(X))\otimes (Y\otimes \O(Y))$ in $\C$ provided by the strong symmetric monoidal structure of $\O$.
 $\Vdeco$ is likewise defined by objects $(V \otimes \O(\imath(V)))$ with $V\in \V$ and the morphisms of $\V$. The inclusion is given by
 $\imath (V\otimes \O(\imath (V)))=(\imath(V)\otimes \O(\imath(V)))$.
\end{df}
Similar modification allow for the decoration of enriched Feynman categories with functors to $\C$ which is also enriched over $\E$.

\subsubsection{The relation between decoration and plus construction}
\label{plusdecpar}
With the notion of enrichment for decoration, we can state a Proposition generalizing Remark \ref{assdecoenrmk} that will be proven in \cite{pluspaper}
\begin{prop}
  In general $(\FF^\gcp)_{\dec \O}=(\FF_\O)^\gcp$.
\end{prop}

\section{Bar, co--bar, Feynman Transforms, \& Master Equations}
\label{barpar}
In analogy with (co)--algebras, there are three transforms we will consider for $\fops$: the bar--, the cobar transform and the Feynman transform aka.\ dual transform.
The bar--cobar transforms yield a pair of adjoint functors.
These transforms serve a dual purpose. One is to give resolutions, the other is to give deformations through master--equations, which generalize the Maurer--Cartan equation.

\subsection{Motivating example: Algebras}
\label{baralgebrasec}
\subsubsection{Bar transform}
If $(A,\eps,d_A,|\cdot |)$ is an augmented associative algebra dg, then the bar transform is the dg--co--algebra given  by the free  co--algebra $BA=T\Sigma^{-1}\bar A$, where $\bar A=ker(\eps)$ together with differential from algebra structure. The usual notation for an element in $BA$ is $a_0|a_1|\dots|a_n$ it will have degree $\sum_{i=0}^n (|a_i|+1)=\sum_i |a_i|+n+1$. In this notation the co--product is
\begin{eqnarray}
\Delta(a_0|\dots|a_n)&=&1\otimes a_0|\dots|a_n+ a_0|\dots|a_n\otimes 1 +\bar\Delta(a_0|\dots|a_n)\nn\\
\bar\Delta(a_0|\dots|a_n)&=&\sum_{i=1}^{n} a_0|\dots|a_{i-1} \otimes a_{i}|\dots |a_n
\end{eqnarray}
Using the differentials
\begin{eqnarray}
\del_i(a_0|a_1|\dots|a_n)&=&a_0|\dots |a_{i-1}a_{i}|\dots |a_n \nn\\\
\del_B (a_0|a_1|\dots|a_n)&=&\sum_{i=1}^n (-1)^{\sum_{j=1}^{i-1}(|a_j|+1}) \del_i (a_0|a_1|\dots|a_n)
\end{eqnarray}
 The total differential on $BA$ is $d_B+d_A$.

\subsubsection{Odd version}
Notice that looking at $\Sigma BA$, the degree of an element $a_0|\dots|a_n$ is $n$ which is the number of bars.
As a mnemonic one can put $deg(|)=1$.
In the shifted version the co--product reduced co--product $\bar \Delta$ has degree $-1$ as there is one bar less.
This type of odd structure lies at the heart of the story for deformations. It actually gives the right degree to the Hochschild complex,
see \cite{KWZ} for a detailed exposition.

\subsubsection{Cobar and Bar/cobar}
Likewise let $(C,\eta,|\cdot|)$ be an associative co--augmented connected dg--co--algebra, $\bar C=coker(\eta)$.  The co--bar transform is the dg--algebra
 $\Omega C:=Free_{alg}(\Sigma^{-1}\bar C)$ together with a differential coming from co--algebra structure dual to the structures above.
The bar--cobar transform
$\Omega BA$ is a  resolution of $A$.

\subsubsection{Dual and Feynman transforms}
For the dual or Feynman transform consider  a finite--dimensional algebra $A$ or a graded algebra with finite dimensional pieces and let $\check A$ be its (graded) dual co--algebra.
Then the dual or Feynman transform of $A$ is $FA:=\Omega \check A$ together with a differential from multiplication.  Now,  the double Feynman transform $FFA$ a resolution.

There is particular interest in the differential for deformations. In particular one considers Maurer-Cartan elements which satisfy the equation $dm+\frac{1}{2}[m,m]=0$.

\subsection{Transforms for Feynman categories}
As before ,one can ask the question of how much of the structure of these transforms can be pulled back to the Feynman category side. The answer is: ``Pretty much all of it''. We shall not discuss all the details which  can be found in \cite{feynman}, but will give an overview following\cite{matrix}.

We start with a general overview and a discussion of the necessary structures.
The result of the transforms $BA$ or $FA$ is actually an odd version of a (co)--free co--algebra or an odd algebra with a (co)differential.

Algebras are a general case of elements of $\Fops$ and hence the transforms will be defined for inputs $\O\in \fopcatc$.
The co--algebra output means that one is going to the opposite category $\C^{op}$ for the target category in the output of the transforms.
The construction will be  free constructions, which, however, also have the extra structure of an additional (co)differential.
 The ``oddness'' is necessary for the signs that are needed in order for the differentials to square to zero. In general, this means that one needs an odd version $\FF^{odd}$ of the Feynman category $\FF$. In order to define this odd version, one needs to make assumptions on the Feynman category and fix a presentation.
 The  transform will transform an $\op$ into a new $\op$ for the odd version of the Feynman category $\FF^{odd}$ either in $\C^{op}$ or $\C$.
 In the graphical case, this is achieved by the twist by $\K$, see Example \ref{Kex}. The equation \eqref{Ktwisteq} then neatly explains the relation between odd version and the shifts. The twist by $\K$ means that each edge gets degree $1$, which is exactly the convention that $deg(|)=1$ in the bar construction; see   Figure \ref{barfig}.

Thus the resulting Feynman category is actually a category of chain complexes in a category enriched over $\Ab$.
Furthermore, for the (co)-differential to work, we have to have signs. These are exactly what is provided by the odd versions. In order to be able to define the transforms, one has to fix an odd version $\FF^{odd}$ of $\FF$, see \cite[\S5.2.3,\S5.2.6]{feynman} for full details. This is analogous to the suspension in the usual bar transforms. In fact, the following is more natural, see \cite{feynman, KWZ}. The degree is $1$ for each bar and in the graph case the edges get degree $1$. The basic example are graphical Feynman categories, for which the odd version is given by the twist with $\K$. In general, one needs an ordered presentation.

\begin{figure}
    \centering
    \includegraphics[width=0.5\textwidth]{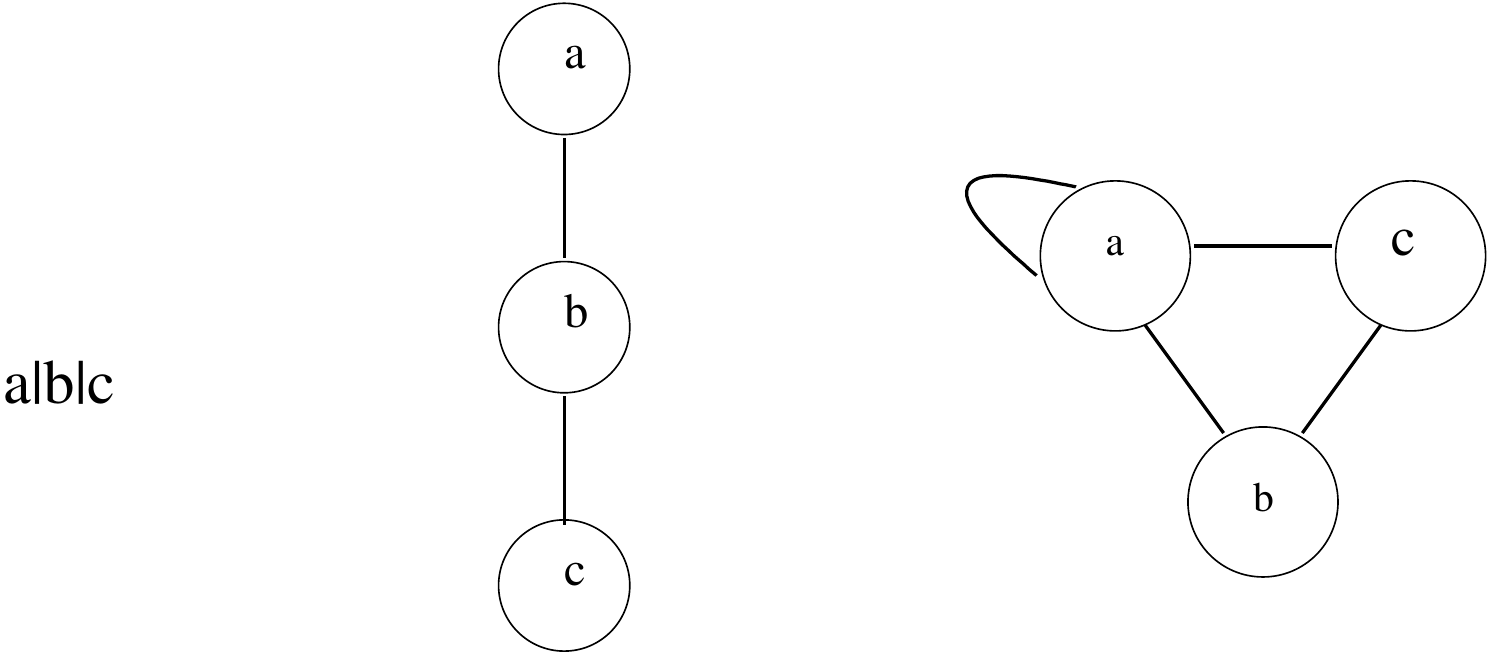}
    \caption{The sign mnemonics for the bar construction, traditional version with the symbols $|$ of degree $1$, the equivalent linear tree with edges of degree $1$,  and a more general graph with edges of degree $1$. Notice that in the linear rooted  case  there is a natural order of edges, this ceases to be the case for more general graphs}
    \label{barfig}
\end{figure}

The transforms are of interest in themselves, but one common application is that the bar-cobar transform as well as the double Feynman transform give a ``free'' resolution. In general, of course, ``free'' means co-fibrant. For this kind of statement one needs a Quillen model structure, which is provided in \S\ref{modelpar}.

The Feynman transform is  quasi--free, that is for  $\O\in \fopsc$, $F\O=\imath^{odd}_!(\imath^*{\O})\in \F^{odd}\mdash\ops$ is free, if one ignores the differential. The dg structure is compatible with $F\O$ precisely if it satisfies a Master Equation, which is fixed by the choice of $\FF^{odd}$.

\subsubsection{Presentations } As mentioned, in order to define the transforms,  we have to give what is called an ordered presentation \cite{feynman}.
Rather then giving the technical conditions, we will consider the graph case  and show these structures in this case.

\subsubsection{Basic example $\GG$}
In $\GG$ the presentation comes from the following set of morphisms $\Phi$
\begin{enumerate}
\item There are 4 types of basic morphisms: Isomorphisms, simple edge contractions, simple loop contractions and mergers. Call this set $\Phi$.
\item
 These morphisms generate all one--comma generators upon iteration. Furthermore, isomorphisms act transitively on the other classes.  The relations on the generators are  given by commutative diagrams.
 \item
 The relations are  quadratic for edge contractions as are the relations involving isomorphisms.
 Finally there is a non--homogenous relation coming from  a simple merger and a loop contraction being equal to an edge contraction.

\item  We can therefore assign degrees as $0$ for isomorphisms and mergers, $1$ for edge or loop contractions and split $\Phi$ as $\Phi^0\amalg \Phi^1$. This gives a degree to any morphism.
\end{enumerate}

Up to isomorphism any morphism of degree $n$ can be written in $n!$ ways up to morphisms of degree $0$. These are the enumerations of the edges of the ghost graph.

There is also a standard order in which isomorphisms come before mergers which come before edge contractions, cf.\ \cite{BorMan,feynman}.
This gives an ordered presentation.

In general, an ordered presentation is a set of generators $\Phi$ and extra data such as the subsets $\Phi^0$ and $\Phi^1$; we refer to \cite{feynman} for details.

\subsubsection{Differential} Given a
$d_{\Phi^1}=\sum_{[\phi_1]\in \Phi^1/\sim} \phi_1\circ$ defines an endomorphism on the Abelian group generated by the isomorphism classes morphisms.
The non--defined terms are set to zero. $\Phi^1$ is called resolving if this is a differential.

In the graph case, this amounts to the fact that for any composition of edge contractions $\phi_e \circ \phi_{e'}$, there is precisely another pair of edge contractions
$\phi_{e''} \circ \phi_{e'''}$ which contracts the edges in the opposite order.

This differential will induce differentials for the transforms, which we call by the same name. We again refer to \cite{feynman} for details.

\subsubsection{Setup} $\FF$ be a Feynman category enriched over $\Ab$ and with an ordered presentation and let $\FF^{odd}$ be its corresponding odd version. Furthermore let $\Phi^1$ be a resolving subset of one-comma generators and let $\mathcal{C}$ be an additive category, i.e.\ satisfying the analogous conditions above. In order to give the definition, we need a bit of preparation.
Since $\V$ is a groupoid, we have that $\V\simeq \V^{op}$. Thus, given a functor $\Phi:\V\to \C$, using the equivalence we get a functor from $\V^{op}$ to $\C$ which we denote by $\Phi^{op}$. Since the bar/cobar/Feynman transform adds a differential, the natural target category from $\fops$ is not $\C$, but complexes in $\C$, which we denote by $Kom(\C)$. Thus any $\O$ may have an internal differential $d_\O$.

\subsubsection
{ The bar construction} This is the functor
\begin{equation*}
\Bar \colon \fops_{Kom(\C)}\to \foddops_{Kom(\C^{op})}
\end{equation*}
\begin{equation*}
\Bar(\O):=\imath_{\FF^{odd} \; *}(\imath_{\FF}^*(\O))^{op}
\end{equation*}
together with the differential $d_{\mathcal{O}^{op}}+d_{\Phi^1}$.

\subsubsection{  The cobar construction} This is the functor
\begin{equation*}
\Cobar \colon \foddops_{Kom(\C^{op})}\to \fops_{Kom(\C)}
\end{equation*}
\begin{equation*}
\Cobar(\mathcal{O}):=\imath_{\FF \; *}(\imath^\ast_{\FF^{odd}}(\mathcal{O}))^{op}
\end{equation*}
together with the co-differential $d_{\mathcal{O}^{op}}+d_{\Phi^1}$.

\subsubsection{Feynman transform}
 Assume there is a duality equivalence $\vee\colon \CalC\to \CalC^{op}$.
The Feynman transform is a pair of functors, both denoted $\FT$,
\begin{equation*}
\FT\colon \fops_{Kom(\C)}  \leftrightarrows \foddops_{Kom(\C)}\colon \FT
\end{equation*}
defined by
\begin{equation*}
 \FT(\O):=\begin{cases} \vee\circ \Bar(\O) & \text{ if } \O \in \fops_{Kom(\C)} \\ \vee\circ \Cobar(\O) & \text{ if } \O \in \foddops_{Kom(\C)}
\end{cases}
\end{equation*}

\begin{prop} {\rm \cite[Lemma 7.4.2]{feynman}} The bar and cobar construction form an adjunction.
\begin{equation*}
\adj{\Cobar}{\foddops_{Kom(\mathcal{C}^{op})}}{\fops_{Kom(\mathcal{C})}}{\Bar}
\end{equation*}
\end{prop}

The quadratic relations in the graph examples are a feature that can be generalized to the notion of {\em cubical} Feynman categories. The name reflects the fact that in the graph example the $n!$ ways to decompose a morphism whose ghost graph is connected and has $n$ edges into simple edge contractions correspond to the edge paths of $I^n$ going from $(0,\dots,0)$ to $(1,\dots, 1)$. Each  edge flip in the path represent one of the quadratic relations and furthermore the $\SS_n$ action on the coordinates is transitive on the paths, with transposition acting as edge flips. We will not give the full detailed definition, but note that $\GG$ is cubical and refer to \cite[Definition 7.2.1]{feynman}
for the technical details.

This is a convenient generality in which to proceed.

\begin{thm} {\rm \cite[Theorem 7.4.3]{feynman}} Let $\FF$ be a cubical Feynman category and $\mathcal{O}\in \fops_{Kom(\C)}$.  Then the co--unit $\Cobar\Bar(\mathcal{O})\to\mathcal{O}$ of the above adjunction is a levelwise quasi-isomorphism.
\end{thm}
\begin{rmk}
In the case of $\mathcal{C}=dgVect$, the Feynman transform can be intertwined with the aforementioned push-forward and pull-back operations to produce new operations on the categories $\mathcal{F}-\mathcal{O}ps_\mathcal{C}$.  A lifting (up to homotopy) of these new operations to $\mathcal{C}=Vect$ is given in \cite{Ward}.  In particular this result shows how the Feynman transform of a push-forward (resp. pull-back) may be calculated as the push-forward (resp. pull-back) of a Feynman Transform.  One could thus assert that the study of the Feynman transform belongs to the realm of Feynman categories as a whole and not just to the representations of a particular Feynman category.
\end{rmk}

\subsection{Master equations}
\label{masterpar}
In \cite{KWZ}, we identified the common background  of master equations that had appeared throughout the literature for operad--like objects and extended them to all graphs examples. An even more extensive theorem for Feynman categories can also be given.

The Feynman transform is quasi--free. An algebra over $F\O$ is dg--if and only if it satisfies the relevant Master Equation. First, we have the tabular theorem from \cite{KWZ} for the usual suspects.

\begin{thm}\label{methm}(\cite{Bar},\cite{MerkVal},\cite{wheeledprops},\cite{KWZ})  Let $\mathcal{O}\in \fopsc$ and $\mathcal{P}\in \foddops_\mathcal{C}$ for an $\F$ represented in Table $\ref{MEtable}$.  Then there is a bijective correspondence:
\begin{equation*}
\Hom(\FT(\mathcal{P}),\mathcal{O})\cong {\rm ME}(\lim_\V(\mathcal{P} \tensor\mathcal{O}))
\end{equation*}

\end{thm}
Here {\rm ME} is the set of solutions of the appropriate master equation set up in each instance.

\begin{table}[htb] \centering
\begin{tabular}{p{2.2cm}||l|l}
Name of  $\fopsc$& Algebraic Structure of $F\O$& Master Equation ({\rm ME})\\ \hline\hline
operad,\cite{GJ}& odd pre-Lie & $d(-)+ -\circ- =0$  \\ \hline
cyclic operad \cite{GKcyclic}  & odd Lie & $d(-)+ \frac{1}{2}[-,-] =0$  \\ \hline
modular operad \cite{GKmodular}& odd Lie + $\Delta$& $d(-)+ \frac{1}{2}[-,-]+\Delta(-) =0$  \\ \hline
properad  \cite{Vallette}& odd pre-Lie & $d(-)+ -\circ- =0$  \\ \hline
wheeled properad \cite{wheeledprops} & odd pre-Lie + $\Delta$ & $d(-)+ -\circ- +\Delta(-) =0$  \\ \hline
wheeled prop \cite{KWZ} & dgBV & $d(-)+ \frac{1}{2}[-,-] +\Delta(-) =0$  \\
\end{tabular}
\caption{
\label{MEtable} Collection of Master Equations for operad--type examples}
\end{table}

 With Feynman categories this tabular theorem can be compactly written and generalized.  The first step is the realization that the differential specifies a natural operation, in the above sense, for each arity $n$. Furthermore, in the Master Equation there is one term form each generator of $\Phi^1$ up to isomorphism.

  The natural operation which lives on a space associated to an $\mathcal{Q}\in \fopcat$ is denoted $\Psi_{\mathcal{Q},n}$ and is formally defined as follows:

\begin{df}\cite[\S7.5]{feynman}
For a Feynman category $\FF$ admitting the Feynman transform and for $\mathcal{Q}\in\fopsc$ we define the formal master equation of $\FF$ with respect to $\mathcal{Q}$ to be the completed co--chain $\Psi_{\mathcal{Q}}:= \prod \Psi_{\mathcal{Q},n}$.  If there is an $N$ such that $\Psi_{\mathcal{Q},n}=0$ for $n>N$, then we define the master equation of $\FF$ with respect to $\mathcal{Q}$ to be the finite sum:
\begin{equation*}
d_{\mathcal{Q}}+\ds\sum_{n}\Psi_{\mathcal{Q},n} = 0
\end{equation*}
We say $\alpha\in \lim_\V(\mathcal{Q})$ is a solution to the master equation if $d_\mathcal{Q}(\alpha)+\sum_{n}\Psi_{\mathcal{Q},n}(\alpha^{\tensor n}) = 0$, and we denote the set of such solutions as ${\rm ME}(\lim_\V(\mathcal{Q}))$.
\end{df}
Here the first term is the internal differential and the term for $n=1$ is the differential corresponding to $d_{\Phi^1}$, where $\Phi^1$ is the subset of odd generators.

\begin{thm}{\rm \cite[Theorem 5.7.3]{feynman}} Let $\mathcal{O}\in \fopsc$ and $\mathcal{P}\in \foddops_\mathcal{C}$ for an $\F$ admitting a Feynman transform and master equation.  Then there is a bijective correspondence:
\begin{equation*}
\Hom(\FT(\mathcal{P}),\mathcal{O})\cong {\rm ME}(\lim_\V(\mathcal{P} \tensor\mathcal{O}))
\end{equation*}
\end{thm}

\section{W-construction and cubical structures}
\label{wpar}
\subsection{Setup}
In this section we start with a cubical Feynman category $\FF$.

\subsubsection{The category   $w(\FF,Y)$, for  $Y \in \F$}

\mbox{}\\

{\sc Objects:}
The objects are the set $\coprod_n C_n(X,Y)\times [0,1]^n$, where $C_n(X,Y)$ are chains of morphisms from $X$ to $Y$ with  $n$ degree $\geq1$ maps
modulo contraction of isomorphisms.

  An object in $w(\FF,Y)$ will be represented (uniquely up to contraction of isomorphisms) by a diagram
\begin{equation*}
X\xrightarrow[f_1]{t_1} X_1\xrightarrow[f_2]{t_2} X_2\to\dots\to X_{n-1}\xrightarrow[f_n]{t_n} Y
\end{equation*}
where each morphism is of positive degree and where $t_1,\dots,t_n$ represents a point in $[0,1]^n$.  These numbers will be called weights.  Note that in this labeling scheme isomorphisms are always unweighted.

{\sc Morphisms:}
\begin{enumerate}
\item  Levelwise commuting isomorphisms which fix $Y$, i.e.:
\begin{equation*}
\xymatrix{X \ar[r] \ar[d]^{\cong} & X_1 \ar[d]^{\cong} \ar[r] & X_2 \ar[d]^{\cong} \ar[r] & \dots \ar[r] & X_n \ar[d]^{\cong} \ar[r] & Y  \\ X^{\prime} \ar[r] & X^{\prime}_1\ar[r] & X^{\prime}_2\ar[r] &\dots \ar[r] & X^{\prime}_n \ar[ur] & }
\end{equation*}
\item  Simultaneous $\SS_n$ action.
\item  Truncation of $0$ weights: morphisms of the form $(X_1\stackrel{0}\to X_2\to\dots\to Y)\mapsto (X_2\to\dots\to Y)$.
\item  Decomposition of identical weights:  morphisms of the form $(\dots \to X_i\stackrel{t}\to X_{i+2} \to \dots) \mapsto (\dots\to X_i\stackrel{t}\to X_{i+1}\stackrel{t}\to X_{i+2}\to\dots)$ for each (composition preserving) decomposition of a morphism of degree $\geq 2$ into two morphisms each of degree $\geq 1$.
\end{enumerate}

\begin{definition}
Let $\mathcal{P}\in\fopst$.  For $Y \in ob(\F)$ we define
\begin{equation*}
W(\mathcal{P})(Y):= \colim_{w(\FF,Y)}\mathcal{P}\circ s (-)
\end{equation*}
\end{definition}

\begin{thm} {\rm \cite[Theorem 8.6.9]{feynman}} Let $\FF$ be a simple Feynman category and let $\mathcal{P}\in\fopst$ be $\rho$-co--fibrant.  Then $W(\mathcal{P})$ is a co--fibrant replacement for $\mathcal{P}$ with respect to the above model structure on $\fopst$.
\end{thm}

Here ``simple'' is a technical condition satisfied by all graph examples.

\subsection{W-construction yields Associahedra}
To see that this indeed yields the usual W-construction, we turn to the well known case of associahedra and its cubical decomposition, see \cite{BoardmanVogt,MSS}.
The relevant Feynman category is the Feynman sub--category of $\operads$ were the vertices are restricted to be at most trivalent rooted planar corollas. This is a sub-Feynman category of $\operads^{pl}=\operads_{\dec Ass}$

Let $*_{n_+}$ be the corolla whose flags are $\{0,\dots,n\}$ where $0$ is the root with the natural linear order on the set of flags as an object of $\operads^{pl}$.

\begin{prop} Let $\trivial$ be the trivial operad to $\Set$.
$W\trivial(*_{n_+})=K_{n}$
\end{prop}
\begin{proof}
A basic morphism is given by an edge contraction. Any morphism can be decomposed into edge contractions $\phi=\phi_{e_1}\circ \cdots\circ \phi_{e_n}$ which up to isomorphism can be taken to be pure.
That means that we get a cube of dimension $n$ for each class of morphism of degree $n$. Such a class is uniquely determined by its ghost tree $\gh(\phi)$, which is a planted planar tree. The maximal degree of a morphisms with the given target $*_{n_+}$ has a source given by $n-1$ trivalent corollas. Any other morphism is obtained by shortening a maximal chain, which correspond to the boundaries $wt(e)\to 0$ by definition.

The compositions two or more morphisms of degree $1$ which are identified in the relations of $w$ with diagonals, not  noted  In the pictures. \end{proof}

\begin{rmk}
The first two cases are given by Figure \ref{K4decomp}, which is taken from \cite{KSchw}. Note that the outer boundary are the pieces, where $wt(e)\to 1$ and the inner boundaries are the ones where $wt(e)\to 0$, which results in the  shortening of the chain of morphisms and hence the contraction of the edge whose weight goes to $0$.
\end{rmk}

\begin{figure}[h]
    \centering
    \includegraphics[width=.8\textwidth]{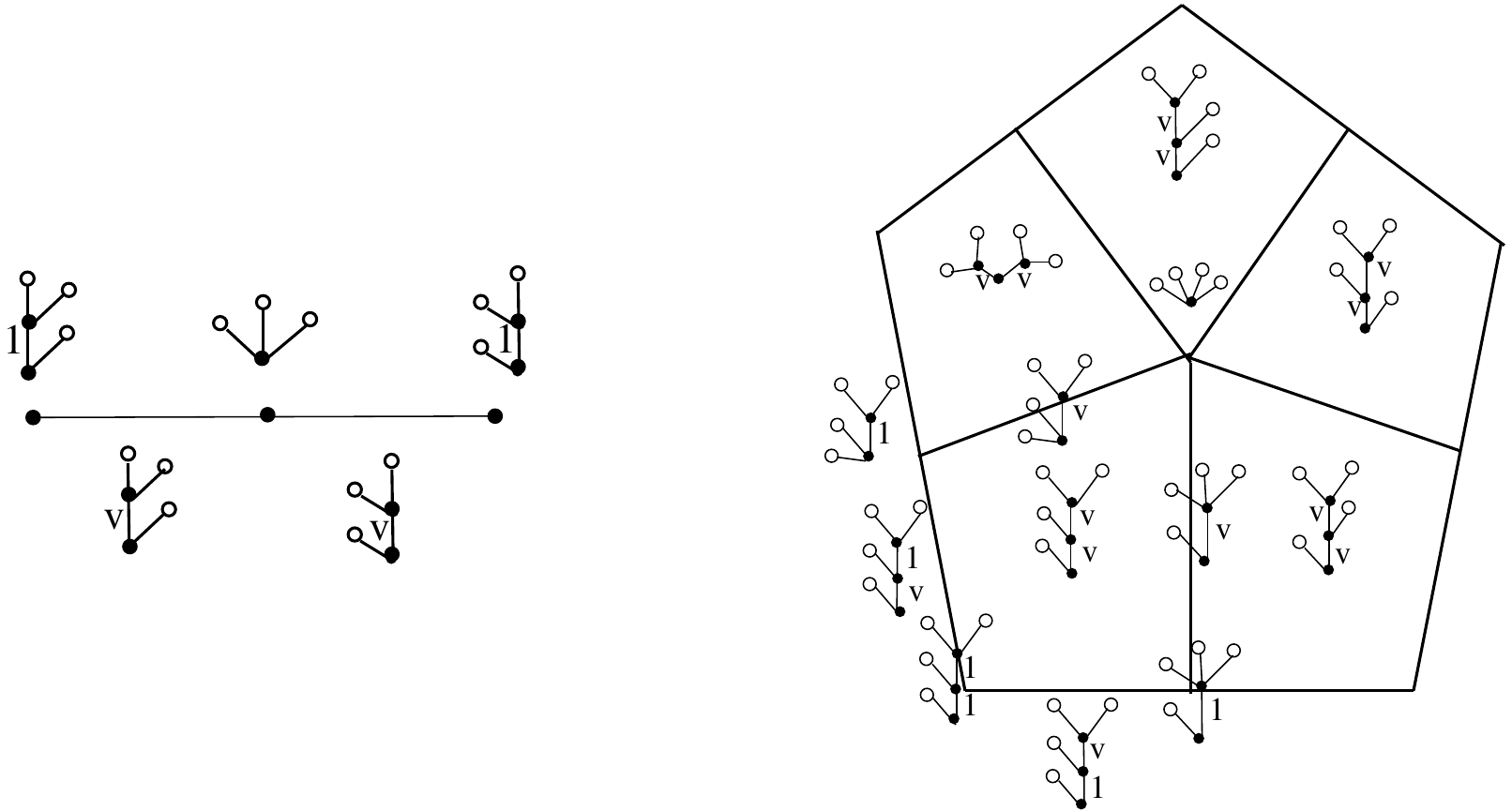}
    \caption{The cubical decomposition for $K_3$ and $K_4$, $v$ indicated a variable height.}
    \label{K4decomp}
\end{figure}

The $\opcat$ structure is the usual operad structure found by Stasheff \cite{Stasheff}. Namely, if $\ccirc{0}{i}: *_{n+}\amalg *_{m_+} \to *_{(m+n-1)_+}$ then
$W\trivial(\phi_{e_i})$ glues the two flags $0$ and $i$ of any two points in $W\trivial(*_{n_+})$ and $W\trivial(*_{m_+})$ and marks the new edge by $1$, that is we obtain the face of $K_{n+m}$ which is $K_n\circ _i K_m=K_n\times K_m$

\begin{rmk}
It is intriguing that this cubical decomposition also turns up in the stability conditions for $A_n$ quivers, which was pointed out to us by K.\ Igusa. It is furthermore worth pointing out that the central charge function results in
particular in an embedding of $K_4$ into $\C$. Such embeddings were sought after since \cite{Stasheff}, cf.\ e.g.\ \cite{MSS} for various truncations and constructions.
\end{rmk}

\section{Outlook}
\label{outlookpar}

\subsection{W-construction and moduli spaces.}
\label{modulispacepar}
Using the language set up in Appendix \ref{connectionspar}, especially \eqref{modulardiag}, in collaboration with C.\ Berger we prove \cite{Ddec}:
\begin{thm}{\rm \cite{Ddec}}
\begin{enumerate}
\item The derived push--forward $k'_!W{\final}(\ast_{g,s,S_1,\dots, S_b})$ is homotopy equivalent to $M_{g,s,S_1,\dots S_b}$.\\[-4mm]
\item The derived push--forward $k_!W{\O_{cA}}(\ast_{g,S})=\amalg_{b,s,S= S_1\amalg \dots \amalg S_b}$ is homotopy equivalent to $M_{g,s,S_1,\dots S_b}$. \\[-4mm]
\item The derived push--forward $j_!W{\final}$ is the cubical complex appearing in the Cutkosky rules in \cite{BlochKreimer, Kreimercut} and Outer Space
in \cite{Vogtout}.
\end{enumerate}
\end{thm}
This generalizes and puts the results of Igusa \cite{CullerVogtmann,Igusa} into this simple Feynman context.
It also makes some of the constructions of Costello \cite{Costelloribbon,Costelloenvelope} more transparent and rigorous.
The last statement is actually brand new and gives a framework for the cubical categories that have appeared in physics and topology.
\begin{thm}{\rm \cite{Ddec}}
$W(\final)(*_{g,s, S_1,\dots, S_b})$ is the cone over the combinatorial aka.\ Penner--Kontsevich compactification $\overline M_{g,s, S_1,\dots, S_b}^{KP}$.
\end{thm}

\subsection{Further connections to representation theory}
\label{furtherreppar}
Generalizations of some of the work presented here is already in the works.
The plus construction shall be generalized for arbitrary monoidal categories in \cite{pluspaper}. The new Feynman category  $\FinSet_<$ can be generalized to fibers with cyclic orders. This should be related via the plus construction  yield the Feynman category for cyclic operads. It is conceivable that if one bases everything on categories with a duality there is a plus construction which yields graphs.

A further point of study will be the connections to cluster algebras. Here there are already several strands especially in the form \cite{Igusamodulated}. The first is the appearance of the cubical decomposition of the Associahedra introduced above. Here the cubical cells are   related to wall crossings.  We expect a similar story for other polytopes. Especially for cyclohedra, which appear in the context of the little discs operad in the form of the cactus operad and Deligne's conjecture for $A_\infty$ algebras. These polytopes are at the vertices of the  diagrams of  \cite[Example 4.0.2]{Igusamodulated} and the cyclic versions in the Feynman theory are possibly related to  \cite{Igusacyclic,Igusacyclic2}, the appearance of cacti being a common theme.

Cluster transformations and cluster varieties yield a method to glue these local complexes together to global data. This is parallel to the results for moduli spaces or Cutkosky rules where the cubical complexes have face transitions according to shrinking edges of different graphs. The prerogative is to
weave these strands together. In this realm, we also expect to encounter 2--Segal spaces of \cite{DyKap1,Dykap2} which basically categorify the pentagon relation, which also appears in mathematical physics and number theory as the identity for quantum di--logarithms.

The enrichment of quivers \S\ref{enrichedquiversec} via \cite{graphsym} provides a new link to mutations and is possibly related to \cite{HananyCoulomb}.

Another intriguing aspect is given by moduli spaces and arcs, which also naturally appear in Feynman categories \cite{feynman,Ddec2} , operadic theory \cite{KLP} as well as in the theory of cluster algebras, see \cite{opper2018geometric,Kontsevicharcs}.

\appendix

 \section{Graph Glossary and Graphical Feynman categories}
\label{graphsec}

\subsection{The category of graphs}

Interesting examples of Feynman categories used in operad--like  theories are indexed over a Feynman category built from graphs.
It is important to note that although we will first introduce a category of graphs $\Graphs$, the relevant
Feynman category is given by a full subcategory $\Agg$  whose
objects are disjoint unions or aggregates of corollas. The corollas themselves
play the role of $\asts$.

Before giving more examples in terms of graphs it will be useful to recall some terminology.
A very useful presentation is given in \cite{BorMan} which we follow here.

\subsubsection{Abstract graphs}
An abstract graph $\G$ is a quadruple $(V_{\G},F_{\G},i_{\G},\del_{\G})$
of a finite set of vertices $V_{\G}$, a finite
set of half edges or flags $F_{\G}$,
an involution on flags $i_{\G}\colon F_{\G}\to F_{\G}; i_{\Gamma}^2=id$ and
a map $\del_{\G} \colon F_{\G}\to V_{\G}$.
We will omit the subscript $\G$ if no confusion arises.

Since the map $i$ is an involution, it has orbits of order one or two.
We will call the flags in an orbit of order one {\em tails} and denote the set of tails by $T_{\Gamma}$. The flags in an orbit of order two will
be called internal flags and this set of tails will be denoted by $F_\G^{\rm int}$. Thus $F_\G=T_g\amalg F_\G^{\rm int}$
We will call an orbit of order two an {\em edge} and denote the set of edges by $E_{\Gamma}$. The flags of
an edge are its elements.
The function $\del$ gives the vertex a flag is incident to.
It is clear that the set of vertices and edges form a 1-dimensional CW complex.
The realization of a graph is the realization of this CW complex.

A graph is (simply) connected if and only if its realization is.
Notice that the graphs do not need to be connected. Lone vertices, that
is, vertices with no incident flags, are also possible.

We also allow the empty graph $\egr$, that is, the unique graph with $V=\varnothing$.
 It will serve as the monoidal unit.
\begin{ex}
A graph with one vertex and no edges is called a {\em corolla}.
Such a graph only has tails.
For any set $S$ the corolla $\crl_{p,S}$ is the unique graph with $V=\{p\}$ a singleton and $F=S$.

We fix the short hand notation $*_S$ for the corolla with $V=\{\ast\}$ and $F=S$.
\end{ex}

Given a vertex $v$ of a graph, we set
$F_v=\del^{-1}(v)$ and call it {\em the
flags incident to $v$}. This set naturally gives rise to a corolla.
The {\em tails} at $v$ is the subset of tails of $F_v$.

As remarked above, $F_{v}$ defines a corolla $\crl_{v}=\crl_{\{v\},F_v}$.

\begin{rmk}
The way things are set up, we are talking about (finite) sets, so
changing the sets even by bijection changes the graphs.
\end{rmk}

\begin{rmk}
\label{disjointrmk}
 As the graphs do not need to be connected, given two graphs $\G$ and $\G'$
we can form their disjoint union:\\
$$\G\sqcup\G'=(F_{\Gamma}\sqcup F_{\Gamma'},V_{\Gamma}\sqcup V_{\G'},
i_{\Gamma}\sqcup i_{\G'}, \del_{\G}\sqcup \del_{\G'})$$

One actually needs to be a bit careful about how disjoint unions are defined.
Although one tends to think that the disjoint union $X\sqcup Y$
is strictly symmetric, this is not the case. This becomes apparent
if $X\cap Y\neq\emptyset$. Of course
there is a bijection $X\sqcup Y \stackrel{1-1}{\longleftrightarrow}Y\sqcup X$.
Thus the categories here are symmetric monoidal, but not strict symmetric monoidal.
This is important, since we  consider functors into other not necessarily strict monoidal categories.

Using MacLane's theorem it is however possible to make a technical construction that makes the monoidal
structure (on both sides) into a strict symmetric  monoidal structure
\end{rmk}

\begin{ex}
An {\em aggregate of corollas} or aggregate for short is a finite disjoint union of corollas, that is,
a graph with no edges.

Notice that if one looks at $X=\bigsqcup_{v\in I} \ast_{S_v}$ for some finite index set $I$ and some finite sets of flags $S_v$,
 then the set of flags is automatically the disjoint
union of the sets $S_v$. We will just say
just say $s\in F_X$ if $s$ is in some $S_v$.
\end{ex}

\subsubsection{Grafting of graphs into vertices}
Fix a graph $\G=(V,F,\del,i)$, a collection of graphs $\G_v=(F(v),V(v),\del(v),i(v))$
indexed by $v\in V$ and a set of bijections $\beta_v:F_v \to T_{\G_v}$ that is a bijection of   the flags at $v$ with the tails of $\G_v$.
Set $\beta=\amalg_{v\in V} \beta_v$.

We define the graph $\amalg_{v\in V}\G_v\circ_\beta\G$ to be the graph obtained from grafting $\G_v$ into the vertex $v$. As a graph the vertices, flags and boundary are those of $\amalg_{v\in V}\G_v$, that is  $\amalg_{v\in V}V(v)$ the flags are $\amalg_{v\in V} F(v)$ with the boundary map $\amalg_{v\in V} \del(v)$. But, the involution given by  $\amalg_{v\in V} i(v)$ on the internal edges of  $\amalg_v \G_v$ is continued to the tails $\amalg_{v \in V} T_{\G_v}$ by $\beta \circ i \circ \beta^{-1}$.
This equivalent to the statement that the edges of the graph $\amalg_{v\in V}\G_v\circ_\beta\G$ are the set $\amalg_{v\in V}E_{\G_v}\amalg \beta(E_\G)$.

\subsubsection{Category structure: Morphisms of Graphs}

\begin{df}\cite{BorMan,feynman}\label{grmor}
Given two graphs $\G$ and $\G'$, consider a triple $(\phi^F,\phi_V, i_{\phi})$
where
\begin{itemize}
\item [(i)]
$\phi^F\colon F_{\G'}\hookrightarrow  F_{\G}$ is an injection,
\item [(ii)] $\phi_V\colon V_{\G}\twoheadrightarrow V_{\G'}$ and $i_{\phi}$ is a surjection and
\item [(iii)] $i_{\phi}$
is a fixed point free involution on the tails of $\G$ not in the image of $\phi^F$.
\end{itemize}

One calls the edges and flags that are
not in the image of $\phi$ the contracted edges
and flags. The orbits of $i_{\phi}$ are called \emph{ghost edges} and denoted by $E_{ghost}(\phi)$. The ghost edges are uniquely determined by and uniquely determine $i_\phi$.

Such a triple is {\it a morphism of graphs} $\phi\colon \G\to \G'$ if

\begin{enumerate}
\item The involutions are compatible:
\begin{enumerate}
\item An edge of $\G$
is either a subset of the image of $\phi^F$ or not contained in it.
\item
If an edge is in the image of $\phi^F$ then its pre--image
is also an edge.
\end{enumerate}
\item $\phi^F$ and $\phi_V$ are compatible with the maps $\del$:
 \begin{enumerate}
\item Compatibility with $\del$ on the image of  $\phi^F$:\\
\quad If $f=\phi^F(f')$ then
$\phi_V(\del f)=\del f'$
\item Compatibility with $\del$ on the complement of the image of  $\phi^F$:\\
 The two vertices of a ghost edge in $\G$ map to
the same vertex in $\G'$ under $\phi_V$.
\end{enumerate}

 \end{enumerate}

If the image of an edge under $\phi^F$ is not an edge,
we say that $\phi$ grafts the
two flags.

The composition $\phi'\circ \phi\colon\G\to \G''$
of two morphisms $\phi\colon\G\to \G'$ and $\phi'\colon\G'\to \G''$
is defined to be  $(\phi^F\circ \phi^{\prime F},\phi'_V\circ \phi_V,i_{\phi'\circ\phi})$
where $i_{\phi'\circ\phi}$ is defined by its orbits viz.\ the ghost edges.
This means that $E_{ghost}(\phi\circ\phi')=E_{ghost}(\phi)\amalg \phi^F(E_{ghost})(\phi')$.

More explicitly, let $F_{\phi}=F'\setminus \phi^{F}(F')$, $F'_{\phi'}=F'\setminus \phi^{\prime F}(F'')$
and $F_{\phi'\circ\phi}=F\setminus \phi^F\circ \phi^{\prime F}(F'')$, then $F_{\phi'\circ \phi}=F_\phi \amalg \phi^F(F'_{\phi'})$ and
\begin{equation}
\label{icompeq}
i_{\phi'\circ\phi}=i_\phi\amalg \phi^{F}\circ i_{\phi'}\circ \phi^{F -1}: F_{\phi'\circ\phi}\to F_{\phi'\circ\phi}
\end{equation}

\end{df}

The following definition from \cite{feynman} is essential.
\begin{df}
\label{ghostdef}
The underlying ghost graph of
a morphism of graphs $\phi\colon\G\to \G'$ is the graph
 $\gh(\phi)=(V_\Gamma,F_{\Gamma},\hat \imath_{\phi})$ where $\hat \imath_{\phi}$ is $i_{\phi}$ on the complement of
$\phi^F(\Gamma')$ and identity on the image of  flags of  $\Gamma'$ under $\phi^F$.
Or, alternatively, the edges of $\gh(\phi)$ are  the ghost edges of $\phi$  that is $E_{ghost}(\phi)$.
\end{df}
\begin{lem}
\label{ghostamalglem}  Using the usual notation,  the following hold for the ghost graphs:
\begin{enumerate}
\item   Let $\phi:\G\to \G'$ then $\gh(\phi)=\amalg_{v'\in V'}\gh_{v'}$ and $\phi^F|_{F'_{v'}}:F'_{v'}\to T_{\gh_{v'}}$ is a bijection.
\item  $\gh(\phi\amalg \psi)=\gh(\phi)\amalg\gh(\psi)$.
\item For a composition of morphisms $\G\stackrel{\phi}{\to}\G'\stackrel{\phi}{\to}\G''$: $\gh(\phi')\circ_\beta \phi= \gh(\phi)\circ_{\phi^F} \gh(\phi)$
\end{enumerate}
\end{lem}
\begin{proof}
For the first statement, we define $\gh(v')(\phi)$ as follows: The vertices
are $\phi_V^{-1}(v')$, the flags are $\amalg_{v\in \phi_V^{-1}(v')}F_{v}$. The morphisms $\del_v$ and $i_v$ are the restriction of $\del_\G$ and $\hat i_\phi$, which are well defined.
For $\del_v$ this is guaranteed by condition (2) and  for
$i_v$ we have to check that if a flag  is in $F_{\phi_V^{-1}}(v')$ then so is $i_\phi(f)$ which is guaranteed by condition (2) (b). Now
 $\phi^F$ is a bijection onto its image, which are precisely the tails of the ghost graph. This restricts to the $\gh_v$.
 The second statement follows from the first.
 The last statement is clear for the vertices and the flags. We have to check that the involution coincide. These are determined by the ghost edges. The ghost edges of the composition are the disjoint union of the ghost edges of $\phi$ and the image of those of $\phi'$ which is guaranteed by \eqref{icompeq}.
\end{proof}
\begin{df}
We let $\Graphs$ be the category whose objects are abstract graphs and
whose morphisms are the morphisms described in Definition \ref{grmor}.
We consider it to be a monoidal category with monoidal product $\sqcup$ (see
Remark \ref{disjointrmk}).
\end{df}

\subsection{The Feynman category $\GG$}
\label{GGdefsec}
Let $\Agg$ be the full subcategory of $\Graphs$ whose objects are aggregates of corollas.
Let $\Crl$ be the sub--groupoid of $\Agg$ whose objects are corollas and whose morphisms are the isomorphisms between them and denote
the inclusion by $\imath$.

\begin{lem}
\label{decomlem}
Given a morphism $\phi\colon X\to Y$ where $X=\bigsqcup_{w\in V_X} \ast_w$ and $Y=\bigsqcup_{v\in V_Y}\ast_v$
are two aggregates, we can decompose $\phi=\bigsqcup \phi_v$ with $\phi_v\colon X_v\to \ast_v$ where $X_v$ is the sub--aggregate $\bigsqcup_{\phi_V(w)=v}*_w$, and $\bigsqcup_v X_v=X$.

Furthermore, $\gh(\phi)=\amalg_{v\in V}\gh(\phi_v)$ and  $\phi_{v}^{\prime F}: F_{v}\to T_{\gh(\phi_v)}$ is a bijection.

\end{lem}

\begin{proof}
Explicitly $(\phi_v)_V$ is the restriction of $\phi_V$ to $V_{X_v}$ and   $\phi_v^F$ is the restriction of
 $\phi^F$ to $(\phi^{F})^{-1}(F_{X_v}\cap \phi^F(F_Y))$. This map is still injective.
Finally $i_{\phi_v}$ is the restriction of $i_{\phi}$ to $F_{X_v}\setminus \phi^F(F_Y)$.
These restrictions are possible due to the condition (2) above. The penultimate statement follows from Lemma \ref{ghostamalglem}.
The tails of the ghost graph are precisely the  elements in the image of $\phi^F$.

\end{proof}

\begin{prop} The triple $\GG=(\Crl,\Agg,\imath)$ is a strictly strict Feynman category.
\end{prop}
\begin{proof}
The proof follows from the Lemma above. Condition (i) is clear on the object level. By the lemma an isomorphism factors as a disjoint union of isomorphisms. It also shows that condition (ii) holds. Note that this condition implies that decompositions are unique up to unique isomorphism. Indeed any decomposition will be given by a permutation of the decomposition above and isomorphisms of the vertices this will uniquely determine the diagram.
Finally, (iii), follows from the category of finite sets is essentially small.
Strictness is clear from the definition.
\end{proof}
\subsection{Extra structures}
\subsubsection{Dictionary}
This section is intended as a reference section.
Recall that  an order of a finite set $S$ is a bijection $S\to \{1,\dots ,|S|\}$.
Thus the group $\SS_{|S|}=\Aut\{1,\dots,n\}$ acts on all orders.
  An  orientation of  a finite set $S$ is an
equivalence class of orders, where two orders are equivalent if they are
obtained from each other by an even permutation.
With this Table \ref{dicttable} provides a dictionary for standard  graph terminology.
\begin{table}
\begin{tabular}{l|l}
A tree& is
a connected, simply connected graph.\\

{A directed  graph $\G$}& is a graph together with  a map $F_{\G}\to \{in,out\}$\\
&such that the two flags of each edge are mapped\\
&to different values. \\

A rooted tree& is a directed tree such that each vertex  has exactly\\& one ``out'' flag.\\
A {ribbon or fat graph} &is a graph together with a cyclic order on each of \\&the
sets $F_{v}$.  \\
A planar graph& is a ribbon graph that can be embedded
into the\\
& plane such that the induced cyclic orders of the \\
&sets $F_v$ from the orientation of the plane  \\
&coincide with the chosen cyclic orders.\\

A planted planar tree&is a rooted planar tree\\
% together with a \\
%&linear order
%on the set of flags incident to the root.\\
An oriented graph& is a graph with an orientation on the set of its edges.\\
An ordered graph& is a graph with an order on the set of its edges.\\
A $\gamma$ labelled graph&is a graph together with a  map $\gamma:V_{\Gamma}\to \N_0$.\\
A b/w graph&is a graph $\G$ with a map $V_{\G}\to \{black,white\}$.\\
A bipartite graph& is a b/w graph whose edges connect only \\
&black to white vertices. \\
A $c$ colored graph& for a set $c$ is a graph $\G$ together with a map $F_{\G}\to c$\\
&s.t.\ each edge has flags of the same color.\\
A connected 1--PI graph& is a connected graph that stays connected, \\
&when one severs any edge.\\
A 1--PI graph&is a graph whose every component is 1--PI.
\end{tabular}
\caption{\label{dicttable}Graph Dictionary}
\end{table}

\subsubsection{Remarks and language}\mbox{}
\begin{enumerate}

\item  In a directed graph one speaks about the ``in'' and the ``out''
edges, flags or tails at a vertex. For the edges this means the one flag of the edges
 is an ``in'' flag at the vertex. In pictorial versions the direction
is indicated by an arrow. A flag is an ``in'' flag if the arrow points to the vertex.

\item
As usual there are edge paths on a graph and the natural notion
of an oriented edge path. An edge path is a (oriented) cycle if it starts and
stops at the same vertex and all the edges are pairwise distinct. It is called simple if
each vertex on the cycle has exactly one incoming flag and one outgoing flag belonging to the cycle.
An oriented simple cycle will be called a {\em wheel}.
An edge whose two vertices coincide is called a {\em (small) loop}.

\item There is a notion of a the genus of a  graph, which is the minimal dimension
of the surface it can be embedded on. A ribbon graph is planar if this genus is $0$.
\item
 For any graph, its Euler characteristic is given by
$$\chi(\Gamma)=b_0(\Gamma)-b_1(\G)=|V_{\Gamma}|-|E_{\Gamma}|; $$
where $b_0,b_1$ are the Betti numbers of the (realization of) $\Gamma$.
Given a $\gamma$ labelled graph, we define
the total $\gamma$ as
\begin{equation}
\label{gammaeq}
\gamma(\Gamma)=1-\chi(\Gamma)+\sum_{v \text{ vertex of $\Gamma$}}\gamma(v)
\end{equation}

If $\G$ is {\em connected}, that is $b_0(\G)=1$ then a $\gamma$ labeled graph is traditionally called a genus labeled graph and
\begin{equation}\label{genuseq}
\gamma(\Gamma)=\sum_{v\in V_{\Gamma}}\gamma(v)+b_1(\Gamma)
\end{equation}
is called the genus of $\G$.
This  is  actually not the genus of the underlying graph, but the genus of
a connected Riemann surface with possible double points whose dual graph is the genus labelled graph.

A genus labelled graph is called {\em stable} if each vertex with genus labeling $0$ has at least 3 flags and
each vertex with genus label $1$ has at leas one edge.
\item A planted planar tree induces a linear order on all sets $F_v$, by declaring the first
flag to be the unique outgoing one.
Moreover, there is a natural order on the edges, vertices and flags
given by its planar embedding.

\item A rooted tree is usually taken to be a tree with a marked vertex. Note that necessarily a rooted tree as described above has exactly  one ``out'' tail. The unique vertex whose ``out'' flag is not a part of an edge is the root vertex.
The usual picture is obtained by deleting this unique ``out'' tail.
\end{enumerate}

\subsubsection{Category of directed/ordered/oriented graphs.}
\label{ordsec}
\begin{enumerate}

\item
Define the category of directed graphs $\Graphs^{dir}$ to be the category whose
objects are directed graphs. Morphisms are morphisms $\phi$ of the underlying graphs,
which additionally satisfy that $\phi^F$ preserves orientation of the flags and the $i_{\phi}$ also
only has orbits consisting of one ``in'' and one ``out'' flag, that is the ghost graph is also directed.

\item

The category of edge ordered graphs $\Graphs^{or}$ has as objects graphs with an order on the edges.
A morphism is a morphism together  with an order $ord$  on  all of the edges of the ghost graph.

The composition of orders on the ghost edges is as follows.
$(\phi,ord)\circ \bigsqcup_{v\in V}(\phi_v,ord_v):=(\phi \circ \bigsqcup_{v\in V}\phi_v,ord\circ\bigsqcup_{v\in V}ord_v)$ where the order on the set of all ghost edges, that is $E_{ghost}(\phi)\sqcup \bigsqcup_vE_{ghost}(\phi_v)$,
is given by first enumerating  the elements of $E_{ghost}(\phi_v)$ in the
order $ord_v$ where the order of the sets  $E(\phi_v)$ is given by
the order on $V$, i.e. given by the explicit  ordering of the tensor product in  $Y=\bigsqcup_v \ast_v$.\footnote{Now we are working with ordered
tensor products. Alternatively one can just index the outer order
by the set $V$ by using \cite{TannakaDel}}
 and then enumerating the edges of $E_{ghost}(\phi)$ in their order $ord$.

\item The oriented version $\Graphs^{or}$ is then obtained by passing from orders to equivalence classes.

\end{enumerate}

\subsubsection{Category of planar aggregates and tree morphisms}
\label{planarsec}
Although it is hard to write down a consistent theory of planar graphs with planar morphisms, if not impossible,
there does exist a planar version of special subcategory of $\Graphs$.

We let $\Crl^{pl}$ have as objects planar corollas --- which simply means that there is a cyclic order on
the flags --- and as morphisms isomorphisms of these, that is isomorphisms of graphs, which preserve the cyclic order.
 The automorphisms of a corolla $\ast_S$ are then isomorphic to $C_{|S|}$, the cyclic group of order $|S|$.
Let $\CCyclic^{pl}$ be the full subcategory of aggregates of planar
corollas whose morphisms are morphisms of the underlying
 corollas, for which the ghost graphs in their planar structure
 induced by the source  is compatible
 with the planar structure on the target via $\phi^F$. For this we use the fact that the tails of a planar tree have a cyclic order.

Let $\Crl^{pl,dir}$ be directed planar corollas with one output and let $\operads^{pl}$ be the  subcategory of $\Agg^{pl,dir}$ of aggregates of corollas of the type just mentioned, whose morphisms are morphisms of the underlying directed corollas such that their associated ghost graphs are compatible with the planar structures as above.

In general, one needs to use so--called almost ribbon graphs, see e.g.\ \cite{decorated}
or \cite[Appendix A1]{postnikov}, and in \S\ref{connectionspar}.

\subsection{Insertion}
Given graphs, $\G$,$\G'$, a vertex $v\in V_{\G}$ and an isomorphism $\phi$: $F_v\mapsto T_{\G'}$
 we define $\G\circ_v\Gamma'$ to be the graph obtained by deleting $v$ and identifying the flags of $v$ with
the tails of $\G'$ via $\phi$. Notice that if $\G$ and $\G'$ are ghost graphs
of a morphism then it is just the composition of ghost graphs, with the morphisms at the other vertices being the identity.

\subsection{Operad-types and their graphical Feynman categories}
\label{operadtypepar}

There is a substantial list of examples that are generated by decorating the graphs of $\GG$ and restricting to certain subcategories,
see Table \ref{zootable}.
The decorations are actually decorations in the technical sense of \S\ref{decopar}.
Examples of the needed decorations are listed in Table
\ref{decotable}.

\begin{table}\begin{tabular}{llll}
$\FF$&Feynman category for&condition on ghost graphs $\gh_v$ for basic \\
&&morphisms and additional decoration&\\
\hline
$\operads$&(pseudo)--operads&rooted trees\\
%$\operads_{May}$&May operads&rooted trees with levels\\
%$\operads^{odd}$&odd operads&rooted trees  + orientation of set of edges& odd pre-Lie\\
$\operads^{\neg\Sigma}$&non-Sigma operads &planar rooted trees\\% & all $\circ_i$ operations\\
$\operads_{\it mult}$&operads with mult.&b/w rooted trees.\\
$\CCyclic$&cyclic operads&trees& \\
$\CCyclic^{\neg\Sigma}$&non--Sigma cyclic operads&planar trees& \\

%$\CCyclic^{odd}$&odd cyclic operads &trees + orientation of set of edges& odd Lie&\\
%&&++ orientation of the set of edges\\
$\GG$&unmarked nc modular operads& graphs \\
$\GG^{\it ctd}$&unmarked  modular operads&connected graphs \\
$\modular$&modular operads&connected + genus marking \\
$\modular^{\it nc}$&nc modular operads &genus marking \\
%$\modular^{odd}$&$\K$--modular&connected + orientation on set of edges & odd dg Lie \\
%&&+ genus marking&\\
%$\modular^{nc,odd}$&nc $\K$-modular& orientation on set of edges & BV\\
%&&+ genus marking&\\
$\dioperads$&dioperads&connected directed graphs w/o directed\\
&&loops or parallel edges\\
$\props$&PROPs&directed graphs w/o directed loops\\
$\properads$&properads&connected directed graphs \\
&&w/o directed loops\\
$\dioperads^{\circlearrowleft }$&wheeled dioperads&directed graphs w/o parallel edges \\
%$\dioperads^{\circlearrowleft odd}$&odd wheeled dioperads&directed graphs w/o parallel edges &BV\\
%&&+ orientations of edges&\\
$\props^{\circlearrowleft,ctd}$& wheeled properads&connected directed graphs \\
%$\props^{\circlearrowleft,ctd, odd}$&odd wheeled properads&connected directed graphs w/o parallel edges &odd Lie admissible\\
%&&+ orientations of edges&+extra differential\\
$\props^{\circlearrowleft}$& wheeled props &directed graphs\\
%$\props^{\circlearrowleft, odd}$&odd wheeled props &directed graphs w/o parallel edges &BV\\
%&&+ orientations of edges&\\
$\FF^{\it 1\mdash PI}$&1--PI algebras&1--PI connected graphs.
\end{tabular}

\caption{\label{zootable} List of Feynman categories with conditions and decorations on the graphs, yielding the zoo of examples}
\end{table}

\subsubsection{Flag labeling, colors, direction and roots as a decoration}
Recall that $*_S$ is the one vertex graph with flags labelled by $S$ and these are the objects of $\V=\Crl$ for $\GG$.
For any set $X$ introduce the following $\GG$-$\oper$: $X(*_S)=X^{S}$. The compositions are simply given by restricting to the target flags.

Now let the set $X$ have an involution $\bar{}:X\to X$.
Then a natural subcategory $\FF_{\dec X}^{dir}$ of $\GG_{\dec X}$ is given by the wide subcategory, whose morphisms additionally satisfy that only flags marked by
elements $x$ and $\bar x$ are glued and then contracted; viz\ $\imath_\phi$ only pairs flags of marked $x$ with edges marked by $\bar x$.  That is the underlying ghost graph has edges whose two flags are labelled accordingly.
 In the notation of  graphs:  $X(f)=\overline{\imath_{\phi}(f)}$.

 If $X$ is pointed by $x_0$, there is the subcategory of $\GG_{\dec X}$ whose objects are those generated by $*_S$ with exactly one flag labelled by $x_0$
 and where the restriction on graphs is that for the underlying graph additionally, each edge has one flag labelled by $x_0$.

Now if $X=\Z/2\Z=\{0,1\}$ with the involution $\bar 0=1$,  we can call $0$ ``out'' and $1$ ``in''.  As a result, we obtain the category of directed graphs $\GG_{\dec\Z/2Z}$.
Furthermore, if $0$ is the distinguished element, we get the rooted version. This explains the relevant
examples  Table \ref{decotable}.
More generally, in quantum field theory the involution sends a field to its anti--field and this is what decorates the lines or propagators in a Feynman graph.

Finally, if the involution is trivial, then we obtain the colored version, where ghost edges have flags of the same color.
\subsubsection{Other decorating operads and connecting the bootstraps}
\label{connectionspar}

The other decoration operads are
 \begin{enumerate}

 \item $\N\in \GG^{\it ctd}\dashopcat_{\Set}$
 given by  $\N(*_S)=\N$ with the composition given by addition for edge contractions and $n\mapsto n+1$ for loop contractions.
 \item $\Assoc \in \operads\dashopcat_{\Set}$, as defined in Lemma \ref{nscoverlem}.
 $\Assoc(*_S)=\{$orders on the set $S\}$

 \item $\CycAssoc\in \CCyclic\dashops_{\Set}$: $\CycAssoc(*_S)=
 \{$cyclic orders on the set $S\}$ here composition is given by
 splicing in the cyclic order, see e.g.\ \cite{woods}.
\end{enumerate}
\begin{table}
\begin{tabular}{llll}
$\FFdeco$&Feynman category for&decorating $\O$&restriction\\
\hline
$\FF^{\it dir}$&directed version&$(\Z/2\Z$, $\bar{0}=1)$&\\
%&edges contain one input\\
%&&& and one output flag\\
$\FF^{\it rooted}$&root&$(\Z/2\Z$, $\bar{0}=1)$&vertices have one\\
&&&  output flag.\\
$\FF^{\it genus}$&genus marked&$\N$&\\
$\FF^{\it C\mdash col}$&colored version&$(C,\bar{c}=c)$&\\
$\operads^{\neg\Sigma}$&non-Sigma-operads&$\Assoc$&\\
$\CCyclic^{\neg\Sigma}$&non-Sigma-cyclic operads&$\CycAssoc$&\\
$\modular^{\neg\Sigma}$&non--Signa-modular&$\ModAssoc$&\\
%$\CCyclic^{dihed}$&dihedral&$\Dihed$&\\
%$\modular^{dihed}$&dihedral modular&$\ModDihed$&
\end{tabular}
\caption{\label{decotable}List of decorated Feynman categories with decorating $\O$ and possible restriction. $\FF$ stands for an example based on $\GG$ in the list or more generally indexed over $\GG$ (see \cite{feynman}). }
\end{table}

Using these decorations, we have a diagram of Feynman categories.

\begin{thm} The following equalities can serve as a natural definition of the right hand side or as a theorem  identifying the right hand side as a decorated Feynman category
\begin{enumerate}
\item $\operads_{\dec \Assoc}=opereads^{\neg\Sigma}$.
\item The morphism $\fri:\operads\to \CCyclic$, given by forgetting the root, is an indexing, but neither connected nor a cover.
 \item  $i^*(\CycAssoc)=\Assoc$.
 \item $\GG^{ctd}_{dec \, j_!(\final)}=\modular$. By general theory $\frj$ factors as $\frj=\forget\circ \frk$ with $\frk:\CCyclic\to \modular$, where $\frk$ is the connected morphism defined by assigning the marking $g=0$ to a corolla.

\item  $\CCyclic_{dec\, \O_{\CycAssoc}}=\CCyclic^{\neg \Sigma}$, the Feynman category whose $\ops$ are non--sigma cyclic operads.

\item $\modular_{dec\, k_!(\O_{\CycAssoc})}= \modular^{\neg \Sigma}$, the Feynman category whose $\ops$ are non--sigma modular operads as defined in \cite{Marklnonsigma,decorated}. Here, $k_!(\CycAssoc)$ is the modular envelope of $\CycAssoc$.
\end{enumerate}
These fit into the diagram \eqref{modulardiag}, where the upper squares are the one of the type \eqref{decosquareeq} and the triangle is also such a square by using the identification $\FF_{dec \final}=\FF$.
\end{thm}

\begin{equation}
\label{modulardiag}
\xymatrix{
\operads^{\neg \Sigma}\ar[rr]^{\fri'}\ar[d]^{p(\Ass)}&&
\CCyclic^{\neg\Sigma} \ar[rr]^{\frk'} \ar[d]^{p(\CycAssoc)} && \modular^{\neg\Sigma}\ar[d]^{p(\frk_!(\CycAssoc))} \\
\operads\ar[rr]_\fri&&
\CCyclic \ar[rr]_\frk\ar[drr]_\frj && \modular\ar[d]^{p(\frj_!(\final))}\\
&&
&&\GG^{ctd}}
\end{equation}

{\sc Details:}
By calculating the push--forwards, one obtains that the basic objects of $\Fe^{mod}$ are  of genus marked corollas $*_{g,S}$ with $g\in \mathbb{N}_0$, while the basic objects of $\Fe^{\neg \Sigma \, mod}$
are  marked corollas $\ast_{g,s,S_1,\dots, S_b}$, where $s,g\in \mathbb{N}_0$ and $S_1,\dots,S_b$ are non--empty sets which each have a cyclic order.

\vskip 2mm
\noindent As a geometric mnemonic $\ast_{g,s,S_1,\dots, S_b}$ represents an oriented {\em topological surface} of genus $g$ with $s$ internal punctures, $b$ boundaries and marked points $S_i$ on the boundary $i$.

These facts are now just a neat calculation using a Kan extension, with the only inputs being $j$ and $cAss$. This is another example of a radical reduction of complicated concepts to  more basic structures.

\section{Graph description of $\FF^+$, $\FF^{+ \it gcp}$ and $\FF^{\it hyp}$}
\label{graphpluspar}
In this section, we will give a graphical version for $\FF^+,\FF^\gcp$ and $\FF^\hyp$. This is a variation of the category for operads.  There is a discrete version, see \cite{pluspaper}, which uses the fact that for a strictly strict $\FF$, $\Obj(\F\downarrow \V)$ are a $\Obj(\V)$ colored operad and hence there is a decoration of the Feynman category for $\Obj(V)$--colored operads.  To obtain the correct behaviour for the isomorphism, we have to consider the
``$\V$--colored operad   $Iso(\F\downarrow \V)$'' and regard the corresponding decorated Feynman category. This is what is captured below.

\subsection{Combinatorial graph based description of $\FF^+$ in the strictified case.}
\label{combpluspar}
In this section, we consider $\FF$ to be strictly strict, that is $Iso(\F)=\V^\otimes$ where the latter is the strict free symmetric monoidal category.

\subsubsection{Planar planted corollas}
Recall that  a corolla is a graph with one vertex and no edges. In the in the notation of \cite{BorMan,feynman}, see Appendix \ref{graphsec}, a corolla with vertex $v$ and set of flags $S$ is given by  $v_S=(\{v\},S,\del:S\to \{v\},id)$. It is planted planar, if $S$ has a linear order. We say the smallest element is the root flag, and denote is by $s_0$. This gives $S$ a natural structure of pointed set $(S,s_0)$ and we denote the planar planted corolla by $(v_S,s_0,<)$.
The order is equivalent to a map  $\lab:S\leftrightarrow \{0,\kdk,n\}$ with $\lab(s_0)=0$ or in other words a bijection of pointed sets.

\subsubsection{The groupoid $\pCrl$}
\label{crlFdefpar}
An $\V$--colored $(\F\downarrow \V)$ decorated corolla is a morally rooted corolla whose lone vertex is $\phi \in \Obj(\F\downarrow \V)$ and whose leaves
$s(\phi)=*_1\odo *_n$ and whose root is $t(\phi)$, where $s,t$ are the source and target maps. This is the view taken in \cite{feynman}.

To use graphical notation as introduced above with all the bells and whistles, consider $(v_S,s_0,<, \clr_F,\dec_V )$, that is a planar planted corolla together with a map $\clr_F$ that decorates the flags by objects of $\V: \clr_F:S\to \Obj(\V)$, and a decoration of the vertex $\dec_V:\{v\}\to \Obj(\imath^\otimes\downarrow \imath)$.
is a compatible decoration of the lone vertex by a morphisms that is a choice $\dec(v)=\phi_v \in \Obj(\F\downarrow \V)$, such that
$\phi_v: \clr_F(\lab^{-1}(1))\odo \clr_F(\lab^{-1}(n))\to \clr_F(\lab^{-1}(0))$.
Two  corollas are equivalent if   $(v_S,s_0,<, \clr_F,\dec_V )\sim (v_S,s_0,<', \clr_F,\dec_V )$  they provide the same $s(\phi)$.
\begin{equation}
\label{flagequiveq}
\clr_F(\lab^{-1}(1))\odo \clr_F(\lab^{-1}(n))=\clr_F(\lab'^{-1}(1))\odo \clr_F(\lab'^{-1}(n))
\end{equation}

An $\V$--colored $(\F\downarrow \V)$ decorated corolla is  an equivalence class of planar planted decorated corollas. $[(v_S,s_0,<, \clr_F,\dec_V )]$.

Each morphisms $\phi:*_1\odo *_n\to *_0$  thus naturally determines an $\FF$--colored corolla $[(v,\{0,1,\to n\},0, \clr_F(n)=n, \dec_V(v)=\phi)]$, which will simply be denoted by $*_\phi$. This carries over to the notation used in \cite{feynman} using the formalism of \cite{TannakaDel}: $\phi:\bigotimes_{s\in S}*_s\to *_t$.

Let $\pCrl$ which is short for $\V\mdash\Crl_{(\F\downarrow \V)}$ be the groupoid whose {\rm objects} are  $\V$--colored $(\F\downarrow \V)$ decorated corollas.

The {\em morphisms} of the groupoid $\pCrl$ are isomorphisms of corollas compatible with the decorations. That is a tuple $(f,\bs)$ where $f$
a graph morphism $v_S\to v'_{S'}$ given by: the only possible map $f_V:\{v\}\to \{v'\}$, a bijection $f^F:S\leftrightarrow S'$ and involution $\imath_f=\emptyset$ which means no ghost edges. Note that $f^F$ induces a permutation: $p=\lab\circ (f^F)\circ \lab^{\prime -1}$. Let $*_i=\clr_F(\lab^{-1}(i)), *'_i=\clr_F(\lab'^{-1}(i))$.
Then $\bs=(\sigma_0,\sigma_1\kdk\sigma_n)$ is an ordered tuple of isomorphisms
$\sigma_i:*'_i \stackrel{\sim}{\to}*_{p(i)}$.
Now let $C_p: *_1\odo *_n\to *_{p(1)}\odo *_{p(n) }$ be the permutation commutativity constraint.

{\em Compatibility} means that $\phi'_{v'}=\sigma_0\circ \phi \circ C_p^{-1}\circ (\sigma_1\odo \sigma_n)$, i.e.\ the following diagram commutes.
\begin{equation}
\xymatrix
{\ast_1\odo \ast_n\ar[r]^{C_p}_\sim\ar[d]^{\phi_v}&\ast_{p(1)}\odo \ast_{p(n)}\ar[d]^{\phi_v\circ C_{p^{-1}}}&\ast'_1\odo \ast'_n\ar[l]_{\sigma_1\odo \sigma_n}^\sim
\ar[d]^{\phi_{v'}}\\
%{\sigma_0\circ\phi_v\circ C_{p}^{-1}\circ(\sigma_1\odo\sigma_n)^{-1}}\\
\ast_0\ar@{=}[r]&\ast_0\ar[r]^{\sim}_{\sigma_0}&\ast'_0
}
\end{equation}
It is straightforward to check that this is compatible with the equivalence relation $\sim$ and hence independent of representative.

Note that these are precisely the isomorphisms allowed in $(\F\downarrow\V)$.

\subsubsection{The  monoidal category  $\pAgg$}
\label{aggFdefpar}
Again, $\pAgg$ is short for $\V\mdash\Agg_{(\F\downarrow \V)}$.

Let $\pAgg$ be the category whose underlying objects are $\\Obj(\pAgg)=\\Obj(\pCrl^{\otimes})$
and whose  morphisms are given by morphisms of decorated aggregates of corollas.
That is {\em compatible} tuples $(g,\bs,\be)$, where
\begin{enumerate}
\item $g$ is an morphism of the underlying rooted corollas $g=(g_V,g^F,\imath_g)$. This means that $\gh(g)$ is a rooted forest and
every ghost edge has exactly one rooted flag that it is naturally directed $e= (f=r,\imath(f))$ with $r$ in the set of roots.

\item  If $F'$ is the set of flags of the target, then
$\bs=\{\sigma_f':f'\in F'\}$ is  a collection of isomorphisms $\sigma_{f'}:\clr(f')\stackrel{\sim}{\to}\clr(g^F(f'))$ and
\item  $\be=\{\sigma_e\},e\in E_{\gh(g)}$ is a collection of  isomorphisms $\sigma_e:\clr(f)\to \clr(\imath(f))$ for each edge directed as $e=(f,\imath_g(f))$
\end{enumerate}
This again clearly compatible with $\sim$.

An example of such a decorated tree is given in Figure \ref{Fplusgpdfig}.
\begin{figure}
  \includegraphics[width=.9\textwidth]{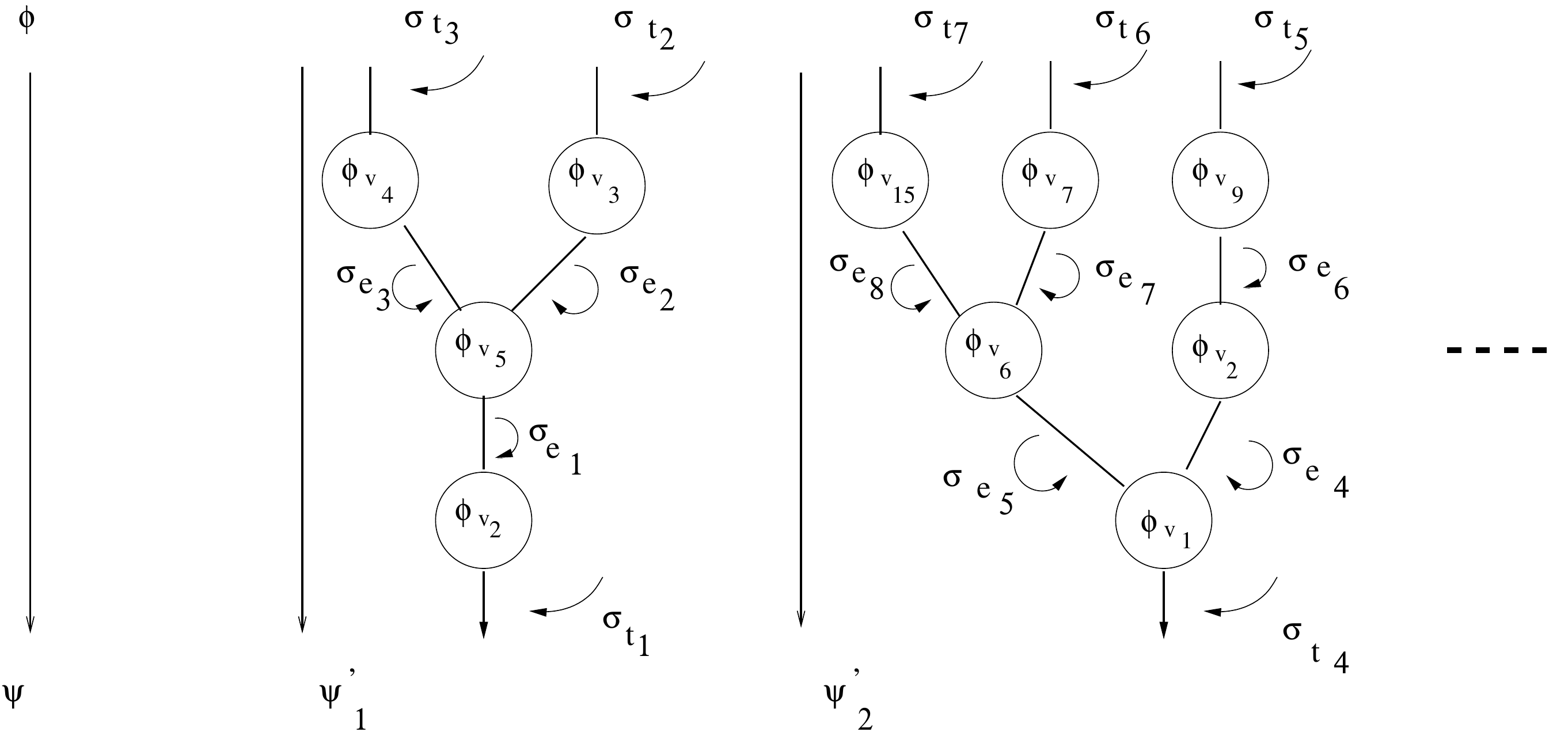}
  \caption{\label{Fplusgpdfig} A decorated rooted forrest. The $\sigma_e\in \be$ are written next to the internal edges, while the $\sigma_t\in \bs$ are written next to the tails. These trees give the flow charts to the $\psi'_i$ and the forrest is a map $\phi=\phi_1\odo \phi_n\to \psi=\psi_1\odo \psi_m$ where the $\psi_j$ are given by \protect{\eqref{psieq}}}
\end{figure}

To explain the compatibility condition, let $\t\subset \gh(g)$ be a connected component, viz.\ a rooted tree. Each such tree  corresponds to a vertex of $w_j$ the target via the map $g_v$, viz.\ all the vertices the tree are precisely the set $f_V^{-1}(w_j)$ and thus the components are can be enumerated as $\tau_j$.
The short version  of {\em compatibility} is that using $g$ and $\be$ each decorated $\tau_j$ represents a flow chart yielding a morphism $\psi_j$ and up to an isomorphism determined by $\bs$:
$\psi_1\odo \psi_m\simeq \psi$.
To define the {\em flow chart},  we first notice that since there is a root, each vertex has a height which is the distance from the root.
Next, we notice that picking a representative of the vertices, the tree actually is planar as there is a linear order at each vertex. This gives a linear order to all the leaves.
We proceed by induction. To start, assume that  $\tau_j$ has  only one vertex, the root vertex $v_i$ with $g_V(v_i)=w_j$.
Let $n=ar(v_i)$ be the arity, i.e.\ the number of incoming flags, of $v_i$, let
the incoming flags be colored by $*_{i_1}\kdk *_{j_n}$, the outgoing (root) flag be colored by $*_{j_0}$
 and let  $\phi_i$ be the color of $v_j$ that is $\phi_i:*_{j_1}\odo *_{i_n}\to *_{j_0}$.  We set $\psi'_j=\phi_j$.

We can now proceed by induction. Assume that we have defined a morphism $\psi'_j(\tau_j)$ for every decorated tree $\tau_j$ of maximal height $n$. Let $\tau_j$ be a tree of maximal height $n+1$. Consider the leaf vertices of $\tau_j$. Since $\tau_j$ is planar, they are ordered. There are two types of vertices. The first is at height $n+1$ and the second is at lower height.
Let $v_i$ be such a leaf vertex then the incoming flags are colored by $*_{i_1}\kdk *_{i_{ar(v_i)}}$ if $v_i$ is not at maximal height we set $\psi'_i=id_{*_{i_1}}\odo id_{*_{i_{ar(v_i)}}}$

If $v_i$ is of maximal height  and $\phi_i$, the color of $v_i$,
maps $*_{i_1}\odo *_{i_{ar(v_i)}}$ to $*_{i_0}$.
There is a (ghost) edge $e(i_0,j_k)$ which connects $v_i$ to some $v_j$. Then $\sigma_e: *_{i_0}\stackrel{\sim}{\to} *_{v_k}$ and we set $\psi_i'=\sigma_e\circ \phi_i$.
Let $\{v_i\}\,i=1\dots l$ be the set of leaf vertices. Let $\tau'$ be the tree obtained from $\tau'$ by cutting the outgoing edges of the leaf vertices of maximal height. Then we set
$\psi'(\tau):=(\psi'_1\odo \psi'_l)\circ \psi'(\tau')$.
This defines a morphism $\psi'(\tau_j)$. The source of the morphism is given  by the color of the leaf flags. Let $*_{l_1}\kdk *_{l_L}$ be the color of the leaf flags and $*_{j_0}$ be the color of the root of $\tau_j$ then $\psi'(\tau_j):*_{l_1}\odo *_{l_L}\to*_{j_0} $.

The map $g^F$ then identifies the flags of the vertices $w_j$ with the tails, that is the leaves and the root, of  $\tau_j$ and the decoration gives further isomorphisms.
The decoration  is {\em compatible} if using these isomorphisms $\psi'(\tau_j)$ is isomorphic to $\psi_j$, the color of $w_j$. More precisely, let $f_{j_1}\kdk f_{j_L}$ be the incoming flags of $w_j$
and $f_{j_0}$ be the root flag let $\sigma_{k}:\clr_F(f_{j_k})\to \clr_F(g^F(f_{j_k}))$ be the isomorphisms provided by $\bs$. Just as before, $g^F$ induces a permutation $p$ on the image and a corresponding commutativity constraint $C_p$.
Then the compatibility equation read

\begin{equation}
\label{psieq}
\psi_j=\sigma_0^{-1}\circ \psi'_j\circ C_P^{-1}\circ(\sigma_1\odo \sigma_L)
\end{equation}
In terms of morphisms such a forrest gives a morphism of vertex colors $\phi_1\odo \phi_n\to \psi_1\odo \psi_m$.
This construction is compatible with $\sim$, and thus does not depend on the choice of particular order $<$ in an equivalence class.

The monoidal structure on morphisms is given by disjoint union of decorated forests.
Composition of morphisms is given by composition of the underlying graph morphisms and decorations.
The composition of the graph morphisms has the effect of inserting  trees $\tau_j$ into the vertices $w_j$ of the ghost forest of the second morphism in a composition
$\amalg v_k\stackrel{g}{\to} \amalg w_j \stackrel{h} {\to} \amalg u_l$.
One tree per vertex, see \cite{feynman}, where the incoming flags of $w_j$ are identified with the leaves of $\tau_j$ and the outgoing flag of $w_j$ with the root of $\tau_j$.
The decoration of this ``blown-up'' tree is given as follows. For $\bs$ this is simply the concatenation of isomorphisms. I.e.\ if $h^F(f)=f'$ and $g^F(f')=f''$ then
the isomorphism is $\sigma_{f'}\circ \sigma_{f''}:\clr_F(f'')\to \clr_F(f)$.
For $\be$, for the edges of the $\tau_j$, the isomorphisms remain. For the image of the edges that connect the $w_j$  one again composes the isomorphisms.
Let $e'=(r',f')$ be a directed edge of $\gh(h)$. And let $g^F(f')=f, g^F(r')=r$ then the edge $e=(r,f)$ is an edge of $\gh(h\circ g)$ and the isomorphism is given by
$\sigma_e=\sigma_f \circ\sigma_{e'} \circ \sigma_r:\clr(r)\to \clr(f)$.

It is straightforward to check that this is associative. Let $\imath$ be the natural inclusion of $\pCrl\to \pAgg$.

\begin{thm}
 \label{forestthm}
 The triple $\pO=(\pCrl,\pAgg,\imath)$ is a Feynman category equivalent to $\FF^+$.
\end{thm}

\begin{proof}
Condition \eqref{objectcond} is clear on the object level by definition. For the isomorphisms, we notice that a morphism in $\pAgg$ is an isomorphisms if only if the forest only has trees of height zero, that is each tree is simply a corolla. In the corolla case, the morphisms in $\pAgg$
co\-incides with the morphism in
$\pCrl$ and hence  $\Iso(\pAgg)\simeq \pCrl^{\otimes}$.

For the condition $\eqref{morphcond}$ we first notice that each morphism in $\Agg^{\rm rt,<}_\FF$ is a tensor product of morphisms whose underlying graph is a tree, and a these are precisely the morphisms $\Obj(\Agg^{\rm rt,<}_\FF\downarrow\Crl^{\rm rt,<}_\FF)$. Thus the condition follows on the object level.
For the morphism level, it is clear that any other decomposition up to isomorphism is given by  permutation of the trees of the forest and an
isomorphism of the $\psi'(\tau)_j$.

Condition \eqref{smallcond} holds, since it holds in $\FF$ and the morphisms in the slice categories are the morphisms are given by decorated trees and the trees as well as the decorations satisfy \eqref{smallcond}, due to the said condition for $\operads$ and for $\FF$.

Thus the triple $\pO$ is a Feynman category. It remains to prove that this equivalent to $\FF^+$. For this it is enough to assume that $\FF$ is strict.

First,  we show that $\V^+$ is equivalent to $\pCrl$. The morphisms for $\V^+$ due to \eqref{morphcond} for $\FF$
 are given two generators, permutations of the source factors and isomorphisms of the source factors and the target. The former correspond to corollas whose decoration $\bs$ is given by identities, but with non--trivial $p$. The latter correspond to non--trivial isomorphisms in  $\bs$, but trivial $p$. These are also exactly the generators of
$\pCrl$. The relations are simply from composition of permutations and isomorphisms in both cases.
Formally the functor is given by sending $\phi:*_1\odo *_n\to *_0$ to $*_\phi$\footnote{For the fastidious reader a canonical choice of the one element set $\{v\}$, would be $\{\phi\}$, viz. using $\phi$ as an atom. Similarly  naturally the set of flags
 is simply the set $\{*_1\kdk  *_n\}$.}.
An iso-morphism $\sds:\phi\to \phi'$ in $\V^+$ to $(f,\bs,\sigma')$ where if $\sigma=\sigma_1\odo\sigma_n\circ p,p\in\SS_n$,
$f=(id,p^{-1}\amalg id,i_\emptyset)$, with $i_\emptyset$ the empty morphism, and $\bs=(\sigma_1\kdk\sigma_n)$.

To show the equivalence for $\pAgg$, we proceed in a similar manner. Having already matched the isomorphisms, we
define the value of a functor on an elementary morphism in $\FF^+$ to be to the morphism $f$ whose ghost tree is a two level tree whose source is the aggregate corresponding to the $\phi_i,i=0,\to n$ and whose target is the decorated corolla corresponding to $\psi=\phi_0\circ(\phi_1\odo \phi_n)$, where $f_V$ is fixed by this data, $f^F$ is the inclusion and $\imath$ is given by connecting the roots of vertices 1\kdk n to the flags of the vertex $0$.
The decorations $\bs$ and $\be$ given by identities, . Furthermore the two monoidal structures also agree in that taking tensor products in $\F^+$ corresponds to $\amalg$ in the $\pAgg$. We extend the definition to all maps of $\FF^+$ by functoriality. This is possible since the relations are preserved. Composing isomorphisms composes the permutations and the isomorphisms of the sources and targets on both sides. Composing elementary morphisms corresponds to gluing in two--level trees into vertices. And composing isomorphisms with elementary morphisms is also the same on both sides.
Finally, we see that there is a functor going backwards, which is the given by the induction process above. At each step, we have a collapse of a the leaf vertices with those below them, which is a combination of the image of an elementary morphism and an isomorphism. It is straightforward to check that this allows to recover any morphism in $\FF^+$ by showing that every morphism of $\FF^+$ is encoded in a unique decorated forest. Notice that the associativity of composition manifests itself in
the fact that one can ``collapse" all edges in any order. Here collapse means that this is a concatenation with an elementary morphism and isomorphisms leading to a blow--up of the collapsed vertex.
\end{proof}

\subsubsection{Leveling}
In the description of the flow chart, we secretly leveled up the trees by introducing unit morphisms for the leaf vertices not of maximal height.
This can be formalized in the tree/forest picture by introducing a new type of bi-valent non-decorated vertex. We will call this a black vertex.
This allows one to level--up the tree to a level tree, see Figure \ref{levelfig}.The leveling-up  replaces the flags of the vertices that are not of maximal height by a string of black vertices to get a level tree of the height of the original tree. This  is what we used above to define the flow chart.  This is not a unique choice.
There are many choices of inserting black vertices,  but these do not alter the result of the flow chart. Thus,
one has to introduce the equivalence relation that two trees with black and white vertices are equivalent if they agree after deleting all black vertices.
The original trees  become  canonical representatives as does the leveled--up tree.

For the composition of flow charts, one composes the underlying
  trees via insertion and then levels up again. It is in this sense that the morphisms of $\F^+$ are level trees/forests \cite[\S3.6]{feynman}.

\begin{figure}
  \includegraphics[height=2.5cm]{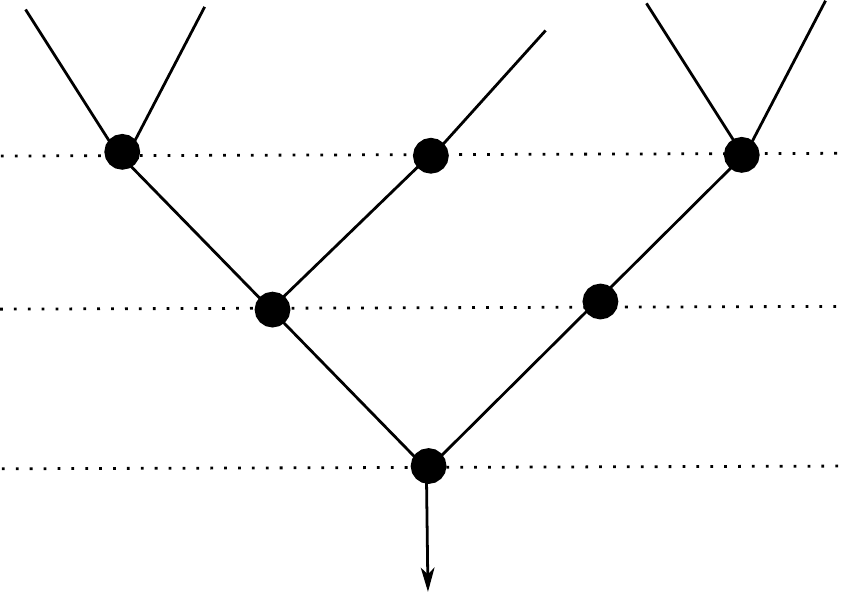} \qquad
  \includegraphics[height=2.7cm]{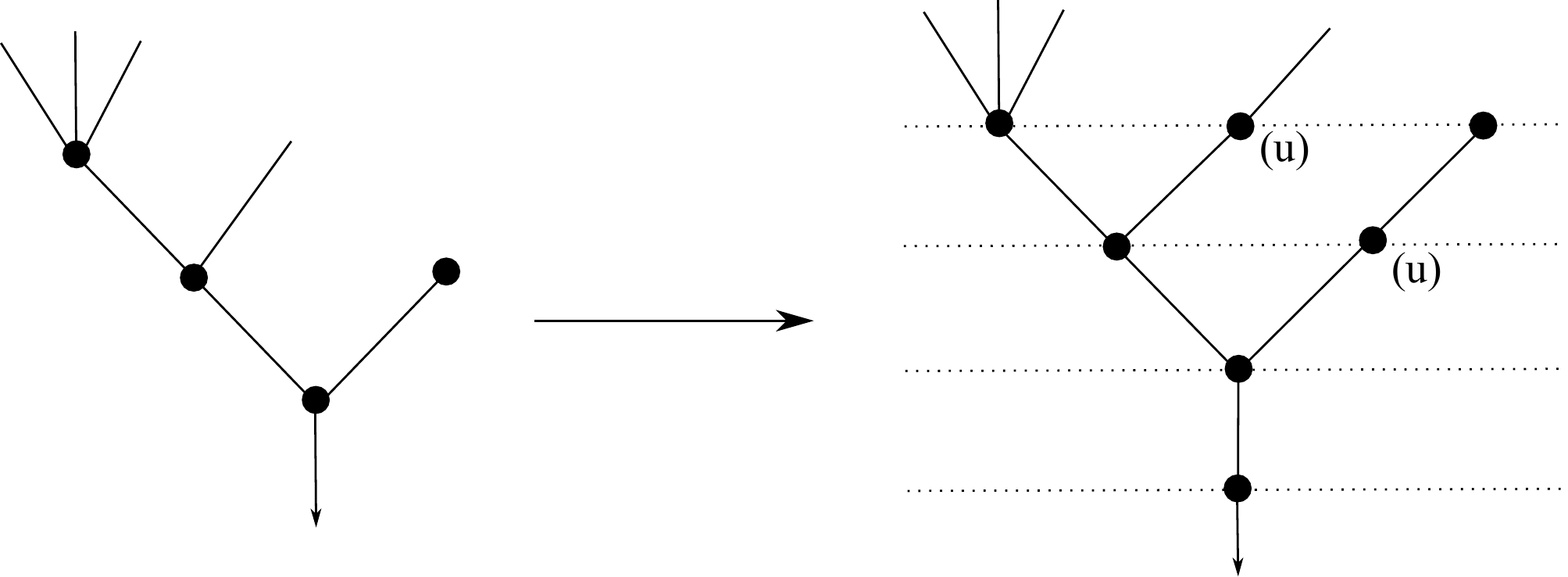}
  \caption{\label{levelfig} A tree with levels and the leveling of a tree. The marking $(u)$ indicates the presence of a unit. }
\end{figure}

\subsubsection{The decorations $(\be,\bs)$ as b/w trees with marked black vertices}
One can think of  $\be$  as a tail  as follows: Using that $T(\gh(g))=g^G(F')$, set $\clr_T: T\to \Mor(\V)$ $f\mapsto\sigma_{(g^F)^{-1}(f)}$. This is  used in Figure \ref{Fplusgpdfig}.
The datum $\bs$ is directly a directed edge decoration. Replacing each edge and each tail by a black two valenced vertex, we can put the decoration on the vertex and
thus encode the morphisms completely combinatorially as $\dottree \; \sigma$, where the roots are now marked by $\dottree \sigma_0^{-1}$.
See Figure \ref{Fplubwfig}.

\begin{figure}
  \includegraphics[width=.8\textwidth]{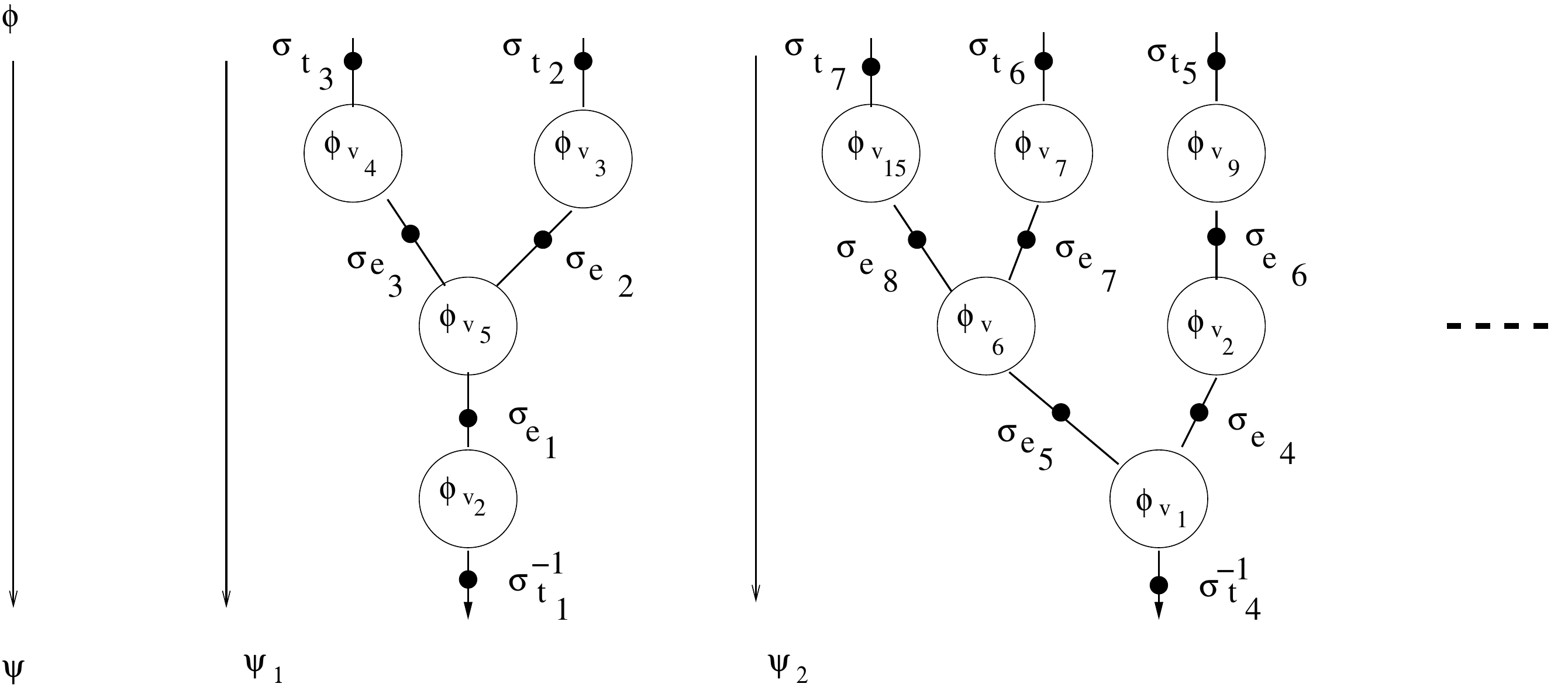}
  \caption{\label{Fplubwfig} The morphism of Figure \protect{\ref{Fplusgpdfig}} as a b/w forest with isomorphism decorations on the black vertices.}
\end{figure}

In particular the morphisms $\sds$ now are encoded as marked black vertices, See Figure \ref{sdscorfig}.

\begin{figure}
  \includegraphics[height=2cm]{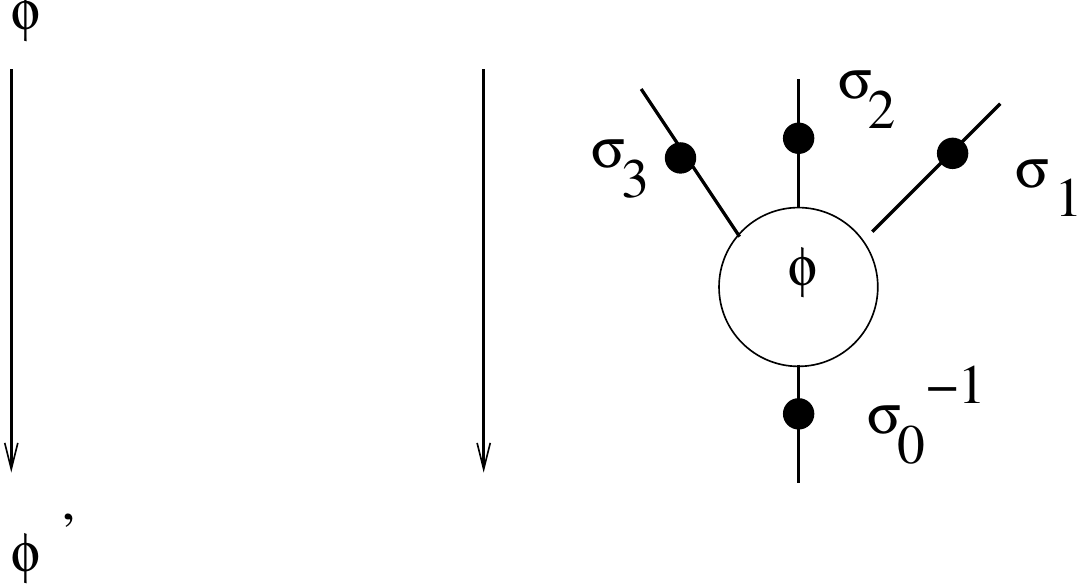}
  \caption{\label{sdscorfig} The morphism $D(\sds):\phi\to \phi'$ with $\sigma=\sigma_1\ot\sigma_2\ot\sigma_3$ and $\sigma'=\sigma_0$ as a b/w corolla with isomorphism decorations on the black vertices.}
\end{figure}

\subsection{The graphical gcp and hyper versions $\FF^{gcp +}$, $\FF^{\it hyp}$}
In the gcp version of the plus construction there are the extra morphisms $i_\sigma: \unit\to *_\sigma$.
We will write these morphisms as $\unit \stackrel{\dottree \; \sigma}{\to} *_\sigma$. Pre--composition with such maps, changes the color of the vertex
$*_\sigma$ to black in the above b/w picture. The well--definedness of the flow-chart is then guaranteed by \eqref{rightcompateq} and  \eqref{leftcompateq}
as the black vertices implement the morphisms $\sds$.

In the case of $\FF^{\it hyp}$, one can either omit the vertices $*_\sigma$, or by abuse of notation  regard black vertices $\dottree \; \sigma$ in lieu of $*_\sigma$.
In the first case, the morphisms are given by decorations, where none of the (non-black) vertices is decorated by an isomorphism.
In the second case, one has to be careful that the $\dottree \; \sigma$ are not quite vertices as they do not belong to $\V^{\it hyp}$, but rather denote an object
isomorphic to $\unit$.Making the vertex black is a useful mnemonic of this.
The abuse of notation is that we identify the target of the isomorphism $\dottree \; \sigma$ with the morphism itself.

\section{Double categories, 2--categories and monoidal categories}
\label{twocatapp}
In this appendix, we show that many of the constructions become natural in the language of 2--categories and double categories as founded in \cite{EhresmannDouble,Benabou} with the further developments in \cite{BrownHolonomy, BrownMosa}, see also \cite{Fioredouble}.
\subsection{2--categories and double categories}
Before going into the definition of 2--categories, we give a natural example:
\begin{ex}[The 2-category of categories]
\label{2catex}
We can set up a category of categories, whose objects are categories and whose morphisms are functors. There is another structure here though, namely there are natural transformation, which are  morphisms between functors. That is morphisms of morphisms or 2--morphisms. These satisfy natural compatibility conditions which can be encoded into a 2--category.
\end{ex}

We will need a slightly more general notion. This is based on the fact that a category can be given by the source, target and identity maps $s,t,id$ with $(s\circ id)(X)=(t\circ id)(X)=X$ and an associative unital composition $\circ$

%$\Obj(\C)
%\begin{tabular}[c]{c} 
%$\stackrel{s}{\leftarrow}$\\[-1mm]
%$\stackrel{id}{\rightarrow}$\\[-1mm]
%$\stackrel{t}{\leftarrow}$
%\end{tabular}
%\Mor(\C)$
\begin{equation}
\label{catdiageq}
\begin{tabular}[c]{c} 
$\Obj(\C)$\\
$s{\uparrow}\downarrow {\it id}\uparrow t$\\
$\Mor(\C)$
\end{tabular}
\quad 
\circ: \Mor(\C) \bisub{s}{\times}{t}\Mor(\C)\to \Mor(\C)
\end{equation}
\begin{df} 
A double category is given by a diagram \eqref{catdiageq} in categories. That is a category of objects $\CObj$ a category of morphisms $\CMor$ together with functors $s,t,id$ and $\circ$ an associative unital composition.

\end{df}

The objects of $\CObj$ are called objects, the morphisms of $\CObj$ are called  vertical 1--morphisms, and their composition,  vertical composition, is denoted by $\circ_v$.
 The objects of $\CMor$ go by the name of horizontal  1--mor\-phisms, with horizontal composition $\circ_h$ given by the functor $\circ$. The morphisms of $\CMor$ are referred to as 2--cells or 2--morphisms.
 The latter have both a horizontal and a vertical composition, the vertical composition $\circ_v$ is the composition in the category $\CMor$, while the horizontal composition is 
   given by the functor $\circ$. These two compositions satisfy the interchange equations.
  
The usual diagrams  for objects, vertical, horizontal and 2-morphisms and their composition are given by 
\begin{equation}
\label{doublecateq}
\xymatrix{
X\ar[r]^\phi\ar[d]_\sigma \ar@{}[dr]|{\Downarrow \alpha}&Y\ar[d]^{\sigma'}\\ 
X'\ar[r]^{\phi'}&Y'
}
\qquad 
\xymatrix{
X\ar[r]^\phi\ar[d]_\sigma \ar@{}[dr]|{\Downarrow}^{\alpha} &Y\ar[d]^{\sigma'}\\ 
X'\ar[r]^{\phi'}\ar[d]_\tau \ar@{}[dr]|{\Downarrow}^{\beta} &Y\ar[d]^{\tau'}\\ 
X''\ar[r]^{\phi''}&Y''
}
\qquad
\xymatrix{
X\ar[r]^\phi\ar[d]_\sigma \ar@{}[dr]|{\Downarrow}^{\alpha_1} &Y\ar[d]^{\sigma'}\ar[r]^{\psi} \ar@{}[dr]|{\Downarrow}^{\alpha_2} &Z\ar[d]^{\sigma''}\\ 
X'\ar[r]^{\phi'}&Y\ar[r]^{\psi'} &Z\\ 
}
\end{equation}
The interchange relation reads:
\begin{equation}
\label{doubleinterchangeeq}
(\alpha_1\circ_h\alpha_2)\circ_v(\beta_1\circ_h\beta_2)=(\alpha_1\circ_v\beta_1)\circ_h(\alpha_2\circ_v\beta_2)
\end{equation}
which means the two ways of composing the diagram \eqref{doublecatintdiageq}, first horizontal then vertical or  first vertical then horizontal, yield the same result
\begin{equation}
\label{doublecatintdiageq}
\xymatrix{
X\ar[r]^\phi\ar[d]_\sigma \ar@{}[dr]|{\Downarrow}^{\alpha_1} &Y\ar[d]^{\sigma'}\ar[r]^{\psi} \ar@{}[dr]|{\Downarrow}^{\alpha_2} &Z\ar[d]^{\sigma''}\\ 
X'\ar[r]^{\phi'}\ar[d]_\tau \ar@{}[dr]|{\Downarrow}^{\beta_1} &Y\ar[d]^{\tau'}\ar[r]^{\psi'} \ar@{}[dr]|{\Downarrow}^{\beta_2} &Z\ar[d]^{\tau''}\\ 
X''\ar[r]^{\phi''}&Y'\ar[r]^{\psi''}&Z''
}
\end{equation}

There are also certain weakenings of the axioms for double categories and their functors, which for instance relax the condition of associativity and units either up to isomorphism or to simply having morphisms. 
In general, retaining an equality adds the adjective ``strict'', allowing for an isomorphism instead of an equality is indicated by the attribute ``strong'' and only postulating a morphism from one side of the equality to the other goes by the designations ``lax'' or ``op--lax'' depending on the direction.

\begin{df} A strict functor between two double categories $\C$ and $\D$ is a pair of functors $(F,G)$: $F:\CObj(\C)\to\CObj(\D)$ and $G:\CMor(\C)\to \CMor(\D)$ which is compatible with the source, target, unit and composition functors.
A horizontally lax functor means that  the conditions are relaxed for the underlying  categories of morphisms.
\end{df}

\begin{rmk} Another nice way of encoding a double category $\D$ is by the four sets $D_0$ of objects, $D_1^H$ of horizontal 1--morphism, $D_1^V$ of vertical 1--morphisms, and $D_2$ of 2--morphisms, which form four categories:
the horizontal edge category $\D_H=\CObj$ given by $D_0$ and $D_1^H$, the vertical edge category $\D_V$ given by $D_0$ and $D_1^V$, the horizontal category of morphisms $\D_{\Mor}^H=\CMor$ with objects $D_1^H$ and morphisms $D_2$ and finally $\D_{\Mor}^V$ with the same morphisms,  but objects $D_1^V$,
see \cite{BrownMosa} for more details.
\end{rmk}

\begin{ex} Given a category $\C$, one can define a double category $\square\C$ with $\CObj(\square\C)=\C$ and $\CMor(\C)$ having objects $\Mor(\C)$ with 2--morphisms given by commutative diagrams.
That is, there is precisely one 2--morphism $(\psi\Downarrow \psi'):\phi\to \phi'$ for any two morphisms $\psi,\psi'$ with $\psi'\circ\phi=\phi'\circ \psi$. 

There is a sub--double--category $\IV\!\C$ of $\square\C$ given by restricting the vertical morphisms to be isomorphisms. 
Explicitly, $\CObj(\IV\C)=\Iso(\C)$ and $\CMor(\C)=\Iso(\C\downarrow \C)$, that is $D_0=\Obj(\C)$, $D_1^V=\Mor(Iso(\C))$, $D_1^H=\Mor(\C)$ and $D_2(\C)=\Mor(\Iso(\C\downarrow \C))$ with  horizontal composition given by  $(\sds)(\phi)\circ (\sigma'\Downarrow \sigma'')(\psi)=(\sigma\Downarrow \sigma'')(\phi\circ\psi)$.

$\IV\C$ is a double category both of whose underlying categories are groupoids, but its horizontal morphisms are not necessarily isomorphisms.
\end{ex}

\begin{ex}[Feynman categories and double categories]
\label{sdsex}
A Feynman category $\FF$ naturally yields the double--category $\IV\F$.
The 2--morphisms are the $(\sigma \Downarrow \sigma'):\phi\to \phi'$, see
\eqref{sdseq}. 

 Note that this is very natural as the condition \eqref{objectcond} and \eqref{morphcond} can be rephrased as: there is an equivalence of double categories
$(\V^\ot,(\imath^\ot \downarrow \imath)^\ot)$ and $\IV\!\F$, where the horizontal composition in $(\imath^\otimes\downarrow\imath)^\otimes$ is  the one naturally induced by the composition of morphisms in $\F$.

\end{ex}

\begin{df}
Following \cite{EhresmannDouble} we define a 2--category to be a double category whose category of objects is discrete, i.e.\ only has identity morphisms.
\end{df}
Example \ref{2catex} is the 2--category of categories whose horizontal 1--mor\-phisms are functors and whose 2--morphisms are natural transformations.

\begin{rmk}
Alternatively, one can simply omit $D_1^V$ from the list and retain only the three sets $D_0,D_1=D_1^H,D_2$ with their structural data.
In terms of diagrams, one shrinks the vertical sides, which are by definition identity maps:

\begin{equation}
\label{2cateq}
\raisebox{8mm}{\xymatrix{
X\ar[r]^\phi\ar@{=}[d] \ar@{}[dr]|{\Downarrow \alpha}&Y\ar@{=}[d]\\ 
X\ar[r]^{\phi'}&Y
 }}
\leadsto
\diagram
X\rtwocell^{\phi}_{\phi'}{\alpha}&Y
\enddiagram
%\xymatrix{
%X\ar@/^/[r]^\phi\ar@/_/[r]_{\phi'}&Y
%}
\end{equation}
\end{rmk}

\begin{ex} [Monoidal categories as 2--categories]
Just as a group $G$ defines a category $\underline{G}$ with one object, any strict monoidal category defines a 2--category.
Any strict monoidal category $(\C,\otimes,\unit)$ is a 2--category with one object $\underline{\C}$. This is $\Obj(\underline{\C})=\{*\}$,
$1\mdash\Mor(*,*)=\Obj(\C)$ with composition $\circ$ given by $\otimes$. That is $X\circ Y:=X\otimes Y$. The identity 1--morphism $id_*$ is $\unit\in \Obj(\C)=1\mdash\Mor(*,*)$, where one uses strictness.  Associativity also holds due to strictness.  $2\mdash\Mor(X,Y)=Mor_\C(X,Y)$ and $\circ_v=\circ$. The units in $2\mdash\Mor(X,X)$ are the $id_X$.
The horizontal composition is given by $\otimes$ again: For $\phi\in Hom_C(X,Y), \phi'\in \Hom(X',Y')$: $\phi'\circ_h\psi=\phi\otimes \psi'$. Due to strictness this is associative and $id_\unit=id_{id_*}$  is the unit for $\circ_h$. Finally, the interchange relation holds \eqref{interchangeeq}.

If the monoidal category is not strict, one has an example of a weak 2--category. In particular, it will be horizontally strong. The associativity  for the composition $\ot$ is not strict and given by associators, which satisfy the pentagon axiom. 
Likewise, the  unit $\unit_\E$  is not strict and its strong unit property for composition composition is given by the unit constraints. 

In the same vein, one has lax and op--lax monoidal functors as examples of a weakening of the conditions on the functor level.
\end{ex}

%\begin{df} A strict 2--functor between two 2--categories  $\C$, $\D$ is given by
%\begin{enumerate}
%\item A map of objects $F:\Obj(\C)\to \Obj(\D)$.
%\item A collection of functors $F_{X,X'}:\Hom(X,X')\to \Hom(F(X),F(X'))$.
%\end{enumerate}
%such that each $F_{X,X'}$ strictly preserves identity objects and they commute with horizontal composition in $\C$ and $\D$.
%
%A lax 2--functor has natural transformation for the horizontal composition  $F_{X,Y}(f)\circ_h F_{Y,Z}(g)\to F_{X,Z}(f\circ_h g)$ and a 2--morphism $id_{F(X)}\to F_{X,X}(id_{X})$ which is compatible with the unit constraints.
%\end{df}

\begin{rmk} There are several relationships between double and 2--ca\-te\-gories.
Being supplied the data $D_0,D_1,D_2$, there are three natural double categories one can construct:
\begin{enumerate}
\item  The horizontal double category,  given by $D_0,D_1^V=D_0,D_1^H=D_1,D_2$ with the natural structure maps. Here the elements of $D_1^V=D_0$ are viewed as  identity maps,
and the 2--morphisms are expanded into squares reversing the shrinking of \eqref{2cateq}.
\item The vertical double category, given by $D_0,D_1^V=D_1,D_1^H=D_0,D_2$. Where the identifications are as above, just switching the roles of horizontal and vertical.
\item The edge--symmetric double category is given by $D_0,D_1^V=D_1,D_1^H=D_1,D_2$, 
where the two--morphisms for a square of the type \eqref{doublecateq} given by a 2--morphism $\alpha:\sigma'\circ\phi\to \phi'\circ\sigma$. \end{enumerate}
\end{rmk}
Vice--versa, a double category has an underlying horizontal  respectively vertical 2-category given restricting the horizontal or vertical morphisms to be identities.

\begin{ex} [A monoidal category as a double category] 
\label{mondoubex} Given a monoidal category $\E$, we define $D(\E)$ to be the horizontal realization of $\underline{\E}$. 
\end{ex}

\begin{df}
\label{enrichfunctdf} 
An enrichment functor for a Feynman category $\FF$   with values in a monoidal category $\E$ is 
a  horizontally lax functor of double categories $(F,\D)$ from  $\IV\F$,  see   Example \ref{sdsex} to 
$D(\E)$, see Example \ref{mondoubex}.
\end{df}
\begin{rmk}

Note that there is only one possible component fuctor $F$, which is the trivial functor $F=\trivial$: $\trivial(X)=*$ and $\trivial(\sigma)=\unit_\E$.
Thus the data for an enrichment functor is: 
\begin{enumerate}
\item A functor 
\begin{equation}
\label{groupoidgendataeq}
\D: \Iso(\F\downarrow \F)\to \E
\end{equation}
On objects, we have $\D(\phi)\in \E$ and for morphisms
$\D(\sds):\D(\phi)\to \D(\phi')$ are isomorphisms.
\item For each pair of composable morphisms $\phi_0,\phi_1$ a natural morphism 
\begin{equation}
\label{composablegeneq}
\D(\phi_0)\otimes \D(\phi_1)\to \D(\phi_0\circ \phi_1)
\end{equation}
\item \label{idXcond}
An element in $\D(id_X)$ for each $X$. That is a morphism $id_{\D(X)}=id_*=\unit\to \D(id_X)$ which is a unit for the maps above.

\end{enumerate}
\end{rmk}
\begin{df}
We say that  an enrichment functor is lax--monoidal, if in addition one has the following data
\begin{enumerate}
\setcounter{enumi}{3}
\item  On  horizontal 1--morphisms maps:
\begin{equation}
\label{monoidalgenobeq}
\D(\phi)\otimes \D(\psi)\to\D(\phi\otimes \psi)\\
\end{equation}
\item On 2--morphisms:
\begin{equation}
\label{monoidalgenmoreq}
\D(\sds)\otimes \D((\tau\Downarrow\tau'))\to\D((\sigma\otimes \tau\Downarrow\sigma'\otimes \tau'))
\end{equation}
\item A unit morphism:
\begin{equation}
 \label{unitcondeq} 
\unit \to \D(id_\unit)
\end{equation}
\end{enumerate}
Such that the constraints are associative (i.e.\ satisfy the pentagon identity), the interchange relation is functorially preserved and the unit constraints are transformed into each other.
On vertical 1--morphisms the morphisms $\trivial(\sigma)\ot\trivial(\sigma')=\unit_\E\ot\unit_\E\to \trivial(\sigma\circ\sigma')=\unit_\E$ are given by the unit constraints.

We say the functor is  strongly monoidal if the morphisms  in the last three equations are isomorphisms, and 
 strictly monoidal if it has equalities in the last equations \eqref{monoidalgenobeq}, \eqref{monoidalgenmoreq} and \eqref{unitcondeq}.
\end{df}

\begin{prop}
\label{twofunctorprop}
For a Feynman category the data of an enrichment functor that is strict monoidal  is up to equivalence, determined by
\begin{enumerate}
\item (Groupoid data) A functor
\begin{equation}\Iso(\F\downarrow \V)\to \E
\end{equation}
\item (Composition data) For $\phi_0\in (\F\downarrow\V)$ and $\phi_1=\phi_{1,1}\odo \phi_{1,n}$, with $\phi_{1,i}\in (\F\downarrow \V)$ that are composable, set $\phi:=\phi_0\circ \phi_1$. Morphisms
\begin{equation}
\D(\phi_0)\otimes \D(\phi_{1,1})\odo \D(\phi_{1,n})\to \D(\phi).
\end{equation}
\item (Unit data) For each object $*_v\in \V$ an element
\begin{equation}
\unit\to \D(id_{*_v})
\end{equation}
which is a unit for the composition data.
\end{enumerate}
\end{prop}

\begin{proof}
Due to the condition \eqref{objectcond} for a Feynman category  $\Iso(\F\downarrow \F)\simeq \Iso(\F\downarrow\V)^\otimes$. Since $\D$ is strict monoidal $\D$ is fixed up to equivalence on $\Iso(\F\downarrow\V)$.
Again, since $\D$ is monoidal, we can use condition \eqref{morphcond} to reduce to the case where $\phi_0\in (\imath^\otimes\downarrow\imath)$ and $\phi_1\in (\imath^\otimes\downarrow\imath^\otimes)$.
Finally, since for and $X$, $id_X$ is isomorphic to $\bigotimes_{v\in V}id_{*_v}$ for some decomposition of $X\simeq \bigotimes_{v\in V}*_v$. We see that up to equivalence the unit data is fixed on the $id_{*_v}$.
\end{proof}

\begin{prop}[Ground monoid/ring]
In the lax monoidal case,  we have that $R:=\D(\id_{\unit_\F})$ is a unital  monoid and if $\D$ takes values in a linear category it is a ring.
Moreover all the $Hom_{\F_\D}(X,Y)$ become $R\mdash R$--modules, and the category is enriched in $R\mdash R$--modules.
\end{prop}

\begin{proof}
The multiplicative structure is given by $\D(id_\unit)\otimes \D(id_\unit)\to \D(id_\unit)$ corresponding to the composition
$id_\unit\circ id_\unit=id_\unit$. The $R\mdash R$--module structure is given by the left and unit constraints.  For the right action:
$\D(\phi)\otimes \D(id_\unit)\to \D(\phi \otimes id_\unit)\to \D(\phi)$, where the first map is given by the lax monoidal structure and the second by the unit constraint in $\F$.  This provides the morphisms $\D(\phi)\ot R \to \D(\phi)$. The unit comes from the structure map $\unit\to \D(id_*)$.

\end{proof}

\subsection{Holonomy, connections,  gcp and hyper functors}
\label{holonomypar}
\subsubsection{Holonomy and connections}
We briefly recall the pertinent elements  from \cite{BrownHolonomy,BrownMosa,Fioredouble} and sketch how to adapt and apply them to the setting of Feynman categories.
A homolony for a double category $\D$ is a functor $\bar{\hphantom{m}}:\D_V\to \D_H$  which is identity on the objects, viz.\  $D_0$.
A left respectively right connection for a holonomy is assignment $\lrcorner$ respectively $\ulcorner$ from $D_1^V\to D_2$
\begin{equation}
\xymatrix{
X\ar[r]^{\bar\sigma}\ar[d]_\sigma \ar@{}[dr]|{\Downarrow  \, \lrcorner(\sigma)}&Y\ar[d]^{id}\\ 
Y\ar[r]^{id}&Y
}
\xymatrix{
X\ar[r]^{id}\ar[d]_{id} \ar@{}[dr]|{\ulcorner(\sigma)\, \Downarrow }&X\ar[d]^{\sigma}\\ 
X\ar[r]^{\bar\sigma}&Y
}
\end{equation}
which satisfy a natural compatibility, see \cite{BrownHolonomy}. A connection pair is a  holonomy together with left and a right connection for it.  

A category with holonomy and a connection pair is called a {\em category with connection}. 
A functor $(F,G)$ of double categories with connections is a functor which preserves this extra structure. 
It weakly preserves the holonomy if there are natural morphisms $F(\sigma)\to G(\bar\sigma)$. Such a functor is strong, if the morphisms are isomorphisms. Likewise, one can relax the condition on the connection; see \S\ref{aplfeypar} for a concrete application.

\begin{ex}
Note that in the double category $\square\C$, we can choose $\bar{\hphantom{m}}$ to be the identity functor. This restricts to a functor $\Iso(\C)\to \C$ which is a holonomy for $\IV\C$,
and further restricts to a holonomy for the horizontal double category of $\C$.
Since there is precisely one two morphisms for each square in both cases, there are unique connections given by $\lrcorner(\sigma)=(\sigma\Downarrow id):\bar\sigma=\sigma\to id$ and $\ulcorner(\sigma)=(\id\Downarrow \sigma):id\to \bar\sigma=\sigma$. This related to the fact that $\square\C$ has a canonical thin structure. In the horizontal double category case, $\sigma$ is an identity.

\subsubsection{Applications to Feynman categories}
\label{aplfeypar}
In particular, there is a canonical connection on the double category of a Feynman category $\IV\F$  and the horizontal double category of $\underline{\E}$.

An enrichment functor weakly preserving the connection has natural morphisms $\trivial(\sigma)=\unit\to \D(\bar\sigma=\sigma)$, i.e.\ $\D$ is groupoid compatibly pointed.
It preserves the connection if $\unit=\D(\sigma)$ and it is strong if $\D$ is a hyper functor. The conditions for the weakness of a gcp functor and the strongness of the hyper-functor for the connections are the diagrams \eqref{urightcompateq} and \eqref{uleftcompateq}. In particular, $\D(\ulcorner(\sigma)):\unit\to \D(\sigma)$ yields the groupoid pointing and $\D(\lrcorner(\sigma):\D(\sigma)\to \unit$ gives the splitting as in Definition \ref{splitdef}.
\end{ex}

\section{Model structures}
\label{modelpar}
In this section we discuss Quillen model structures for $\fopcat_\C$. It turns out that these model structures can be defined if $\C$ satisfies certain conditions and if this is the case work for all $\FF$, e.g.\ all the previous examples.

\subsection{Model structure}
\begin{thm}{\rm \cite[Theorem 8.2.1]{feynman}}\label{modelthm}  Let $\FF$ be a Feynman category and let $\mathcal{C}$ be a cofibrantly generated model category and a closed symmetric monoidal category having the following additional properties:
\begin{enumerate}
\item  All objects of $\mathcal{C}$ are small.
\item  $\mathcal{C}$ has a symmetric monoidal fibrant replacement functor.
\item  $\mathcal{C}$ has $\tensor$-coherent path objects for fibrant objects.
\end{enumerate}
Then $\fopsc$ is a model category where a morphism $\phi\colon \mathcal{O}\to\mathcal{Q}$ of $\F\-ops$ is a weak equivalence (resp. fibration) if and only if $\phi\colon \mathcal{O}(v)\to\mathcal{Q}(v)$ is a weak equivalence (resp. fibration) in $\mathcal{C}$ for every $v\in \V$.
\end{thm}

\subsubsection{Examples}
\begin{enumerate}
\item Simplicial sets. (Straight from Theorem \ref{modelthm})
\item ${\it dg}\Vect_k$ for ${\it char}(k)=0$ (Straight from Theorem \ref{modelthm})
\item $\Top$ (More work, see below.)
\end{enumerate}

\subsubsection{Remark}
Condition (i) is not satisfied for $\Top$ and so we cannot directly apply the theorem. In \cite{feynman} this point was first cleared up by following \cite{Fresse} and using the fact that all objects in $Top$ are small with respect to topological inclusions.

\begin{thm} {\rm \cite[Theorem 8.2.13]{feynman}} 
Let $\mathcal{C}$ be the category of topological spaces with the Quillen model structure.  The category $\fopsc$ has the structure of a cofibrantly generated model category in which the forgetful functor to $\vseq$ creates fibrations and weak equivalences.
\end{thm}
\subsection{Quillen adjunctions from morphisms of Feynman categories}
\subsubsection{Adjunction from morphisms}
We assume $\mathcal{C}$ is a closed symmetric monoidal and model category satisfying the assumptions of Theorem \ref{modelthm}.  Let $\fr{E}$ and $\FF$ be Feynman categories and let $\ff\colon \fr{E}\to\FF$ be a morphism between them.  This morphism induces an adjunction
\begin{equation*}
f_!\colon\eops \leftrightarrows \fopsc\colon f^*
\end{equation*}

\begin{lemma}\label{qalem} Suppose $f^*$ restricted to $\vfmods\to \vemods$ preserves fibrations and acyclic fibrations, then the adjunction  $(f_!, f^*)$ is a Quillen adjunction.
\end{lemma}

\subsection{Cofibrant replacement}

\begin{thm}  The Feynman transform  of a non-negatively graded dg $\F$-$\oper$ is co-fibrant.
\label{cofcor}

The double Feynman transform of a non-negatively graded dg-$\F$-$\oper$ in a cubical Feynman category is a co-fibrant replacement.
\end{thm}

\bibliography{ausbib}

\begin{thebibliography}{GCKT20b}

\bibitem[Aus74]{Auslander}
Maurice Auslander.
\newblock Representation theory of {A}rtin algebras. {I}, {II}.
\newblock {\em Comm. Algebra}, 1:177--268; ibid. 1 (1974), 269--310, 1974.

\bibitem[Bar07]{Bar}
Serguei Barannikov.
\newblock Modular operads and {B}atalin-{V}ilkovisky geometry.
\newblock {\em Int. Math. Res. Not. IMRN}, (19):Art. ID rnm075, 31, 2007.

\bibitem[Bat98]{Bathigher}
M.~A. Batanin.
\newblock Monoidal globular categories as a natural environment for the theory
  of weak {$n$}-categories.
\newblock {\em Adv. Math.}, 136(1):39--103, 1998.

\bibitem[Bat17]{Batnoper}
M.~A. Batanin.
\newblock An operadic proof of {B}aez-{D}olan stabilization hypothesis.
\newblock {\em Proc. Amer. Math. Soc.}, 145(7):2785--2798, 2017.

\bibitem[Bau98]{BauesHopf}
Hans-Joachim Baues.
\newblock The cobar construction as a {H}opf algebra.
\newblock {\em Invent. Math.}, 132(3):467--489, 1998.

\bibitem[BB17]{PolynomialMonadplus}
M.~A. Batanin and C.~Berger.
\newblock Homotopy theory for algebras over polynomial monads.
\newblock {\em Theory Appl. Categ.}, 32:Paper No. 6, 148--253, 2017.

\bibitem[BD98]{BaezDolanPlus}
John~C. Baez and James Dolan.
\newblock Higher-dimensional algebra. {III}. {$n$}-categories and the algebra
  of opetopes.
\newblock {\em Adv. Math.}, 135(2):145--206, 1998.

\bibitem[B{\'e}n67]{Benabou}
Jean B{\'e}nabou.
\newblock Introduction to bicategories.
\newblock In {\em Reports of the {M}idwest {C}ategory {S}eminar}, pages 1--77.
  Springer, Berlin, 1967.

\bibitem[BFSV03]{VogtiteratedMonoidal}
C.~Balteanu, Z.~Fiedorowicz, R.~Schw\"{a}nzl, and R.~Vogt.
\newblock Iterated monoidal categories.
\newblock {\em Adv. Math.}, 176(2):277--349, 2003.

\bibitem[BK15]{BlochKreimer}
Spencer Bloch and Dirk Kreimer.
\newblock {C}utkosky {R}ules and {O}uter {S}pace.
\newblock {\em arXiv:1512.01705 preprint}, 2015.

\bibitem[BK17]{BergerKaufmann}
Clemens Berger and Ralph~M. Kaufmann.
\newblock Comprehensive factorisation systems.
\newblock {\em Tbilisi Math. J.}, 10(3):255--277, 2017.

\bibitem[BK20a]{Ddec}
Clemens Berger and Ralph~M. Kaufmann.
\newblock {D}erived {F}eynman functors and modular envelopes.
\newblock {\em in preparation}, 2020.

\bibitem[BK20b]{Ddec2}
Clemens Berger and Ralph~M. Kaufmann.
\newblock {D}erived {F}eynman functors and modular envelopes: chain aspects.
\newblock {\em in preparation}, 2020.

\bibitem[BKM]{BaKaMa}
M.~A. Batanin, Ralph~M. Kaufmann, and Martin Markl.
\newblock {C}omparing {F}eynman and {O}peradic {C}ategories.
\newblock {\em In preparation}.

\bibitem[BKW18]{BKW}
Michael Batanin, Joachim Kock, and Mark Weber.
\newblock Regular patterns, substitudes, {F}eynman categories and operads.
\newblock {\em Theory Appl. Categ.}, 33:Paper No. 7, 148--192, 2018.

\bibitem[BM99]{BrownMosa}
Ronald Brown and Ghafar~H. Mosa.
\newblock Double categories, {$2$}-categories, thin structures and connections.
\newblock {\em Theory Appl. Categ.}, 5:No. 7, 163--175, 1999.

\bibitem[BM08]{BorMan}
Dennis~V. Borisov and Yuri~I. Manin.
\newblock Generalized operads and their inner cohomomorphisms.
\newblock In {\em Geometry and dynamics of groups and spaces}, volume 265 of
  {\em Progr. Math.}, pages 247--308. Birkh\"auser, Basel, 2008.

\bibitem[BM15]{BMopcat}
Michael Batanin and Martin Markl.
\newblock Operadic categories and duoidal {D}eligne's conjecture.
\newblock {\em Adv. Math.}, 285:1630--1687, 2015.

\bibitem[BS76]{BrownHolonomy}
Ronald Brown and Christopher~B. Spencer.
\newblock Double groupoids and crossed modules.
\newblock {\em Cahiers Topologie G\'{e}om. Diff\'{e}rentielle}, 17(4):343--362,
  1976.

\bibitem[BSV18]{Vogtout}
Kai-Uwe Bux, Peter Smillie, and Karen Vogtmann.
\newblock On the bordification of outer space.
\newblock {\em J. Lond. Math. Soc. (2)}, 98(1):12--34, 2018.

\bibitem[BV73]{BoardmanVogt}
J.~M. Boardman and R.~M. Vogt.
\newblock {\em Homotopy invariant algebraic structures on topological spaces}.
\newblock Lecture Notes in Mathematics, Vol. 347. Springer-Verlag, Berlin-New
  York, 1973.

\bibitem[CK00]{CK1}
Alain Connes and Dirk Kreimer.
\newblock Renormalization in quantum field theory and the {R}iemann-{H}ilbert
  problem. {I}. {T}he {H}opf algebra structure of graphs and the main theorem.
\newblock {\em Comm. Math. Phys.}, 210(1):249--273, 2000.

\bibitem[CK01]{CK2}
Alain Connes and Dirk Kreimer.
\newblock Renormalization in quantum field theory and the {R}iemann-{H}ilbert
  problem. {II}. {T}he {$\beta$}-function, diffeomorphisms and the
  renormalization group.
\newblock {\em Comm. Math. Phys.}, 216(1):215--241, 2001.

\bibitem[Cos07a]{Costelloribbon}
Kevin Costello.
\newblock A dual version of the ribbon graph decomposition of moduli space.
\newblock {\em Geom. Topol.}, 11:1637--1652, 2007.

\bibitem[Cos07b]{Costelloenvelope}
Kevin Costello.
\newblock Topological conformal field theories and {C}alabi-{Y}au categories.
\newblock {\em Adv. Math.}, 210(1):165--214, 2007.

\bibitem[CV86]{CullerVogtmann}
Marc Culler and Karen Vogtmann.
\newblock Moduli of graphs and automorphisms of free groups.
\newblock {\em Invent. Math.}, 84(1):91--119, 1986.

\bibitem[Del90]{TannakaDel}
P.~Deligne.
\newblock Cat\'egories {T}annakiennes.
\newblock In {\em The {G}rothendieck {F}estschrift, {V}ol.\ {II}}, volume~87 of
  {\em Progr. Math.}, pages 111--195. Birkh\"auser Boston, Boston, MA, 1990.

\bibitem[DK15]{DyKap1}
T.~Dyckerhoff and M.~Kapranov.
\newblock Crossed simplicial groups and structured surfaces.
\newblock In {\em Stacks and categories in geometry, topology, and algebra},
  volume 643 of {\em Contemp. Math.}, pages 37--110. Amer. Math. Soc.,
  Providence, RI, 2015.

\bibitem[DK18]{Dykap2}
Tobias Dyckerhoff and Mikhail Kapranov.
\newblock Triangulated surfaces in triangulated categories.
\newblock {\em J. Eur. Math. Soc. (JEMS)}, 20(6):1473--1524, 2018.

\bibitem[Ehr63]{EhresmannDouble}
Charles Ehresmann.
\newblock Cat\'{e}gories structur\'{e}es.
\newblock {\em Ann. Sci. \'{E}cole Norm. Sup. (3)}, 80:349--426, 1963.

\bibitem[Fio07]{Fioredouble}
Thomas~M. Fiore.
\newblock Pseudo algebras and pseudo double categories.
\newblock {\em J. Homotopy Relat. Struct.}, 2(2):119--170, 2007.

\bibitem[FL91]{Lodaycrossed}
Zbigniew Fiedorowicz and Jean-Louis Loday.
\newblock Crossed simplicial groups and their associated homology.
\newblock {\em Trans. Amer. Math. Soc.}, 326(1):57--87, 1991.

\bibitem[Fre10]{Fresse}
Benoit Fresse.
\newblock Props in model categories and homotopy invariance of structures.
\newblock {\em Georgian Math. J.}, 17(1):79--160, 2010.

\bibitem[GCKT20a]{HopfPart1}
Imma G\'alvez-Carrillo, Ralph~M. Kaufmann, and Andrew Tonks.
\newblock Three {H}opf algebras from number theory, physics \& topology, and
  their common background {I}: operadic \& simplicial aspects.
\newblock {\em Communications in Number Theory and Physics}, 14(1):1--90, 2020.

\bibitem[GCKT20b]{HopfPart2}
Imma G\'alvez-Carrillo, Ralph~M. Kaufmann, and Andrew Tonks.
\newblock Three {H}opf algebras from number theory, physics \& topology, and
  their common background {II}: general categorical formulation.
\newblock {\em Communications in Number Theory and Physics}, 14(1):91--169,
  2020.

\bibitem[Get09]{Getzler}
Ezra Getzler.
\newblock Operads revisited.
\newblock In {\em Algebra, arithmetic, and geometry: in honor of {Y}u. {I}.
  {M}anin. {V}ol. {I}}, volume 269 of {\em Progr. Math.}, pages 675--698.
  Birkh\"auser Boston Inc., Boston, MA, 2009.

\bibitem[GJ94]{GJ}
Ezra Getzler and Jones J.D.S.
\newblock Operads, homotopy algebra and iterated integrals for double loop
  spaces.
\newblock {\em http://arxiv.org/abs/hep-th/9403055}, 1994.

\bibitem[GK95]{GKcyclic}
E.~Getzler and M.~M. Kapranov.
\newblock Cyclic operads and cyclic homology.
\newblock In {\em Geometry, topology, \& physics}, Conf. Proc. Lecture Notes
  Geom. Topology, IV, pages 167--201. Int. Press, Cambridge, MA, 1995.

\bibitem[GK98]{GKmodular}
E.~Getzler and M.~M. Kapranov.
\newblock Modular operads.
\newblock {\em Compositio Math.}, 110(1):65--126, 1998.

\bibitem[HKK17]{Kontsevicharcs}
F.~Haiden, L.~Katzarkov, and M.~Kontsevich.
\newblock Flat surfaces and stability structures.
\newblock {\em Publ. Math. Inst. Hautes \'{E}tudes Sci.}, 126:247--318, 2017.

\bibitem[HS18]{HananyCoulomb}
Amihay Hanany and Marcus Sperling.
\newblock Resolutions of nilpotent orbit closures via {C}oulomb branches of
  3-dimensional {$\mathcal{N}=4$} theories.
\newblock {\em J. High Energy Phys.}, (8):189, front matter+35, 2018.

\bibitem[Igu02]{Igusa}
Kiyoshi Igusa.
\newblock {\em Higher {F}ranz-{R}eidemeister torsion}, volume~31 of {\em AMS/IP
  Studies in Advanced Mathematics}.
\newblock American Mathematical Society, Providence, RI; International Press,
  Somerville, MA, 2002.

\bibitem[IT15]{Igusacyclic}
Kiyoshi Igusa and Gordana Todorov.
\newblock Cluster categories coming from cyclic posets.
\newblock {\em Comm. Algebra}, 43(10):4367--4402, 2015.

\bibitem[IT19]{Igusacyclic2}
Kiyoshi Igusa and Gordana Todorov.
\newblock Cyclic posets and triangulation clusters.
\newblock {\em Sci. China Math.}, 62(7):1289--1316, 2019.

\bibitem[JR79]{JR}
S.~A. Joni and G.-C. Rota.
\newblock Coalgebras and bialgebras in combinatorics.
\newblock {\em Stud. Appl. Math.}, 61(2):93--139, 1979.

\bibitem[Kas95]{Kassel}
Christian Kassel.
\newblock {\em Quantum groups}, volume 155 of {\em Graduate Texts in
  Mathematics}.
\newblock Springer-Verlag, New York, 1995.

\bibitem[Kau04]{woods}
Ralph~M. Kaufmann.
\newblock Operads, moduli of surfaces and quantum algebras.
\newblock In {\em Woods {H}ole mathematics}, volume~34 of {\em Ser. Knots
  Everything}, pages 133--224. World Sci. Publ., Hackensack, NJ, 2004.

\bibitem[Kau09]{postnikov}
Ralph~M. Kaufmann.
\newblock Dimension vs. genus: a surface realization of the little {$k$}-cubes
  and an {$E_\infty$} operad.
\newblock In {\em Algebraic topology---old and new}, volume~85 of {\em Banach
  Center Publ.}, pages 241--274. Polish Acad. Sci. Inst. Math., Warsaw, 2009.

\bibitem[Kau17]{matrix}
Ralph~M. Kaufmann.
\newblock {L}ectures on {F}eynman {C}ategories.
\newblock In {\em 2016 {M}atrix {A}nnals}, volume~1 of {\em {M}atrix {B}ook
  Series}, pages 375--438. Springer, Cham, 2017.

\bibitem[Kel82]{kellybook}
Gregory~Maxwell Kelly.
\newblock {\em Basic concepts of enriched category theory}, volume~64 of {\em
  London Mathematical Society Lecture Note Series}.
\newblock Cambridge University Press, Cambridge, 1982.

\bibitem[KKGJ15]{Igusamodulated}
Igusa Kiyoshi, Orr Kent, Todorov Gordana, and Weyman Jerzy.
\newblock Modulated semi-invariants,.
\newblock arXiv: 1507.03051v2, 2015.

\bibitem[KKW15]{graphsym}
R~M Kaufmann, S~Khlebnikov, and B~Wehefritz{\textemdash}Kaufmann.
\newblock Projective representations from quantum enhanced graph symmetries.
\newblock {\em Journal of Physics: Conference Series}, 597:012048, apr 2015.

\bibitem[KKWK16]{KKWKsym}
Ralph~M. Kaufmann, Sergei Khlebnikov, and Birgit Wehefritz-Kaufmann.
\newblock Re-gauging groupoid, symmetries and degeneracies for graph
  {H}amiltonians and applications to the gyroid wire network.
\newblock {\em Ann. Henri Poincar\'{e}}, 17(6):1383--1414, 2016.

\bibitem[KL16]{decorated}
Ralph~M. Kaufmann and Jason Lucas.
\newblock {D}ecorated {F}eynman {C}ategories.
\newblock {\em arXiv:1602.00823, J.\ of Noncomm.\ Geo.\ to appear}, 2016.

\bibitem[KLP03]{KLP}
Ralph~M. Kaufmann, Muriel Livernet, and R.~C. Penner.
\newblock Arc operads and arc algebras.
\newblock {\em Geom. Topol.}, 7:511--568 (electronic), 2003.

\bibitem[KM20]{pluspaper}
Ralph~M. Kaufmann and Michael Monaco.
\newblock The plus construction.
\newblock {\em Preprint}, 2020.

\bibitem[Kre16]{Kreimercut}
Dirk Kreimer.
\newblock {C}utkosky {R}ules from {O}uter {S}pace.
\newblock {\em arXiv:1607.04861 preprint}, 2016.

\bibitem[KS10]{KSchw}
Ralph~M. Kaufmann and R.~Schwell.
\newblock Associahedra, cyclohedra and a topological solution to the
  {$A_\infty$} {D}eligne conjecture.
\newblock {\em Adv. Math.}, 223(6):2166--2199, 2010.

\bibitem[KW13]{feynmanarxiv}
Ralph~M. Kaufmann and Benjamin~C. Ward.
\newblock {F}eynman {C}ategories.
\newblock 2013.
\newblock arXiv:1312.1269v1.

\bibitem[KW17]{feynman}
Ralph~M. Kaufmann and Benjamin~C. Ward.
\newblock {F}eynman {C}ategories.
\newblock {\em Ast\'erisque}, (387):vii+161, 2017.
\newblock arXiv:1312.1269.

\bibitem[KWK18]{chapter}
Ralph~M. Kaufmann and Birgit Wehefritz-Kaufmann.
\newblock Theoretical properties of materials formed as wire network graphs
  from triply periodic {CMC} surfaces, especially the gyroid.
\newblock In {\em The role of topology in materials}, volume 189 of {\em
  Springer Ser. Solid-State Sci.}, pages 173--200. Springer, Cham, 2018.

\bibitem[KWZ12]{KWZ}
Ralph~M. Kaufmann, Benjamin~C. Ward, and J~Javier Zuniga.
\newblock The odd origin of {G}erstenhaber, {B}{V}, and the master equation.
\newblock {\em arxiv.org:1208.5543}, 2012.

\bibitem[Law70]{LawvereComp}
F.~William Lawvere.
\newblock Equality in hyperdoctrines and comprehension schema as an adjoint
  functor.
\newblock In {\em Applications of {C}ategorical {A}lgebra ({P}roc. {S}ympos.
  {P}ure {M}ath., {V}ol. {XVII}, {N}ew {Y}ork, 1968)}, pages 1--14. Amer. Math.
  Soc., Providence, R.I., 1970.

\bibitem[Lei04]{Leinster}
Tom Leinster.
\newblock {\em Higher operads, higher categories}, volume 298 of {\em London
  Mathematical Society Lecture Note Series}.
\newblock Cambridge University Press, Cambridge, 2004.

\bibitem[Lod98]{Loday}
Jean-Louis Loday.
\newblock {\em Cyclic homology}, volume 301 of {\em Grundlehren der
  Mathematischen Wissenschaften [Fundamental Principles of Mathematical
  Sciences]}.
\newblock Springer-Verlag, Berlin, second edition, 1998.
\newblock Appendix E by Mar{\'{\i}}a O. Ronco, Chapter 13 by the author in
  collaboration with Teimuraz Pirashvili.

\bibitem[Mar16]{Marklnonsigma}
Martin Markl.
\newblock Modular envelopes, {OSFT} and nonsymmetric (non-{$\Sigma$}) modular
  operads.
\newblock {\em J. Noncommut. Geom.}, 10(2):775--809, 2016.

\bibitem[ML98]{MacLane}
Saunders Mac~Lane.
\newblock {\em Categories for the working mathematician}, volume~5 of {\em
  Graduate Texts in Mathematics}.
\newblock Springer-Verlag, New York, second edition, 1998.

\bibitem[MMS09]{wheeledprops}
M.~Markl, S.~Merkulov, and S.~Shadrin.
\newblock Wheeled {PROP}s, graph complexes and the master equation.
\newblock {\em J. Pure Appl. Algebra}, 213(4):496--535, 2009.

\bibitem[MSS02]{MSS}
Martin Markl, Steve Shnider, and Jim Stasheff.
\newblock {\em Operads in algebra, topology and physics}, volume~96 of {\em
  Mathematical Surveys and Monographs}.
\newblock American Mathematical Society, Providence, RI, 2002.

\bibitem[MV09]{MerkVal}
Sergei Merkulov and Bruno Vallette.
\newblock Deformation theory of representations of prop(erad)s. {I}.
\newblock {\em J. Reine Angew. Math.}, 634:51--106, 2009.

\bibitem[OPS18]{opper2018geometric}
Sebastian Opper, Pierre-Guy Plamondon, and Sibylle Schroll.
\newblock A geometric model for the derived category of gentle algebras, 2018,
  1801.09659.

\bibitem[PR02]{NCSetPira}
T.~Pirashvili and B.~Richter.
\newblock Hochschild and cyclic homology via functor homology.
\newblock {\em $K$-Theory}, 25(1):39--49, 2002.

\bibitem[So03]{NCSet}
Jolanta S\l~omi\'{n}ska.
\newblock Decompositions of the category of noncommutative sets and
  {H}ochschild and cyclic homology.
\newblock {\em Cent. Eur. J. Math.}, 1(3):327--331, 2003.

\bibitem[SS17]{FId}
Steven~V. Sam and Andrew Snowden.
\newblock Gr\"{o}bner methods for representations of combinatorial categories.
\newblock {\em J. Amer. Math. Soc.}, 30(1):159--203, 2017.

\bibitem[SS19]{FIG}
Steven~V. Sam and Andrew Snowden.
\newblock Representations of categories of {$G$}-maps.
\newblock {\em J. Reine Angew. Math.}, 750:197--226, 2019.

\bibitem[Sta63]{Stasheff}
James~Dillon Stasheff.
\newblock Homotopy associativity of {$H$}-spaces. {I}, {II}.
\newblock {\em Trans. Amer. Math. Soc. 108 (1963), 275-292; ibid.},
  108:293--312, 1963.

\bibitem[Val07]{Vallette}
Bruno Vallette.
\newblock A {K}oszul duality for {PROP}s.
\newblock {\em Trans. Amer. Math. Soc.}, 359(10):4865--4943, 2007.

\bibitem[War16]{Ward}
Benjamin~C. Ward.
\newblock Maurer-{C}artan elements and cyclic operads.
\newblock {\em J. Noncommut. Geom.}, 10(4):1403--1464, 2016.

\bibitem[War19]{BenSix}
Benjamin~C. Ward.
\newblock Six operations formalism for generalized operads.
\newblock {\em Theory Appl. Categ.}, 34:Paper No. 6, 121--169, 2019.

\end{thebibliography}
\bibliographystyle{halpha}

\end{document}